\documentclass[11pt]{article}
\usepackage[utf8]{inputenc}
\usepackage[T1]{fontenc}
\usepackage{lmodern}
\usepackage{microtype}
\DisableLigatures{encoding = *, family = tt*}
\usepackage{needspace}
\usepackage{amsmath,amssymb,amsthm}
\usepackage{booktabs}
\usepackage{array}
\usepackage{xcolor}
\usepackage[margin=2.5cm]{geometry}
\usepackage{tikz}
\usetikzlibrary{arrows.meta,positioning,calc}
\usepackage[hidelinks]{hyperref}
\usepackage{xurl}
\definecolor{wplus}{RGB}{38,92,170}
\definecolor{wminus}{RGB}{205,65,55}
\definecolor{tblue}{RGB}{38,92,170}\definecolor{torange}{RGB}{230,130,20}
\definecolor{tgreen}{RGB}{40,150,70}\definecolor{tred}{RGB}{205,65,55}
\theoremstyle{plain}
\newtheorem{theorem}{Theorem}
\newtheorem{proposition}[theorem]{Proposition}
\newtheorem{lemma}[theorem]{Lemma}
\newtheorem{corollary}[theorem]{Corollary}
\theoremstyle{definition}
\newtheorem{definition}[theorem]{Definition}

\newcommand{\W}{\mathrm{W}}
\newcommand{\PAF}{\mathrm{PAF}}
\newcommand{\circm}{\operatorname{circ}}
\newcommand{\diag}{\operatorname{diag}}
\newcommand{\T}{^{\top}}
\newcommand{\one}{\mathbf{1}}
\newcommand{\blue}[1]{{\color{wplus}\textbf{#1}}}
\newcommand{\red}[1]{{\color{wminus}\textbf{#1}}}
\newcommand{\mono}[1]{{\ttfamily\small #1}}
\newcommand{\sq}[2]{\tikz[baseline=-0.5ex]\fill[#1] (0,0) rectangle (0.22,0.22);}

\title{Explicit weighing matrices $\W(n,k)$ for all $n\equiv 0 \pmod 4$ with $4\le n\le 108$, and all $1\le k\le n$}
\author{Adri\'an Mac\'ias-Quintero\thanks{ORCID: \href{https://orcid.org/0009-0002-3638-9559}{0009-0002-3638-9559}. Code and data, archived at Zenodo: see~\cite{zenodo}.}}
\date{October 2026}

\begin{document}
\maketitle
\begin{abstract}
We present an explicit weighing matrix $\W(n,k)$ for every order $n\equiv 0\pmod 4$ with $4\le n\le 108$ and every weight $1\le k\le n$: $1512$ matrices in total, each checked in exact integer arithmetic.
Of these, $961$ come from classical constructions (identity, direct sums, Kronecker products, Paley matrices and skew-Hadamard doubling) and $551$ are Goethals--Seidel matrices, found by exhaustive search ($24$ matrices, $n\le 24$) or by simulated annealing on the periodic autocorrelation cost ($527$ matrices, $n\ge 28$).
The note is self-contained as far as the methods are concerned: it presents all the construction methods used and the algorithm that selects among them, explains the algorithm step by step, proves every identity on which the constructions rely, and illustrates each method with a matrix taken from the catalogue. It does not reproduce the $1512$ matrices, which are available in the repository. It also gives a short parity obstruction to Goethals--Seidel matrices $\W(4m,4m-1)$ with $m$ even, and lists the generating sequences of thirteen matrices whose parameters appear as unresolved in historical tables consulted by the author; their status in the current literature is still being checked, and no claim of priority is made.
\end{abstract}

\section{Introduction}\label{sec:intro}

\begin{definition}
A \emph{weighing matrix} $\W(n,k)$ is an $n\times n$ matrix $W$ with all entries in $\{-1,0,1\}$ such that $WW\T=kI_n$, where $W\T$ is the transpose of $W$ and $I_n$ the identity matrix of order $n$. The number $k$ is its \emph{weight}.
\end{definition}

Reading the equation entry by entry, $WW\T=kI_n$ says that every row of $W$ has exactly $k$ nonzero entries (the diagonal entries of $WW\T$ are the sums of the squares of the entries of each row) and that any two different rows are orthogonal (the other entries of $WW\T$ are the dot products of two rows). Since $W$ is square, $WW\T=kI_n$ implies $W\T W=kI_n$ as well (for $k\ge1$, $W$ is invertible with inverse $k^{-1}W\T$), so the same holds for columns. A \emph{Hadamard matrix} is a $\W(n,n)$, that is, a matrix with entries $\pm1$ and $HH\T=nI_n$. We reserve the letter $H$ for Hadamard matrices only; the simplest one is
\[
H_2=\begin{pmatrix}1&1\\1&-1\end{pmatrix},\qquad H_2H_2\T=\begin{pmatrix}1\cdot1+1\cdot1&1\cdot1+1\cdot(-1)\\1\cdot1+(-1)\cdot1&1\cdot1+(-1)\cdot(-1)\end{pmatrix}=\begin{pmatrix}2&0\\0&2\end{pmatrix}=2I_2 .
\]

Historical sources on existence include the chapter on orthogonal designs of the \emph{Handbook of Combinatorial Designs}~\cite{handbook}, whose later updates are recorded in~\cite{dinitz}, and the survey of Koukouvinos and Seberry~\cite{ks97}. An undecided entry in a historical table does not, by itself, establish that a case is still open.

\paragraph{Related work.}
The conjecture that a weighing matrix $\W(n,k)$ exists for every $k<n$ when $n\equiv0\pmod4$ is discussed in~\cite{gpw74}, which proves it for $n$ a power of $2$, and in~\cite{gw74}; a general reference is~\cite{gsbook}. The Goethals--Seidel array yields weighing matrices from four ternary circulant sequences whose periodic autocorrelation functions sum to zero at every nonzero shift, the sum at shift zero being the weight; using zero-autocorrelation sequences in this way is a classical approach~\cite{ks93}; for arrays and sequences in general see~\cite{sy20}. Computer searches for weighing matrices built from circulant blocks have also been carried out, in one case with genetic algorithms~\cite{bdgks24,kkps12,kks12}. According to their abstracts, \cite{kkps12,kks12} use two circulant blocks (order $2n$) and \cite{bdgks24} partitioned group matrices, which generalize one- and two-circulant constructions; our construction instead uses four circulants in the Goethals--Seidel array (order $4m$), with an exhaustive search for $m\le6$ and simulated annealing on the autocorrelation cost for $m\ge7$. We have not compared the parameters covered by those papers in detail. Small weighing matrices have been classified~\cite{hm12}, and symmetric weighing matrices and the classification of small integer weighing matrices have been studied separately~\cite{gsa23,gbds26}; here we impose no symmetry.

\paragraph{What this note is about.}
The parameter set is
\[
\mathcal S=\{(4m,k):1\le m\le27,\ 1\le k\le4m\},\qquad |\mathcal S|=4\sum_{m=1}^{27}m=1512 .
\]
\begin{theorem}\label{thm:main}
For every $(n,k)\in\mathcal S$ there exists a weighing matrix $\W(n,k)$.
\end{theorem}
The proof (Section~\ref{sec:verify}) consists of $1512$ explicit matrices, one for each pair, together with an exact check that each satisfies the definition. The algorithm of Section~\ref{sec:algo} is the tool that \emph{discovered} these matrices; their validity does not depend on it. The paper is about \emph{one specific catalogue}: for each $(n,k)\in\mathcal S$ the repository stores exactly one matrix and the name of the method that produced it. A pair $(n,k)$ usually admits many other weighing matrices, equivalent or not; they play no role here. Whenever the text says ``the program does X'', this describes the code of the repository, and every matrix drawn in a figure as a $\W(n,k)$ is the stored matrix of that pair (or, in one explicitly labelled case, a re-run of the same code).

\paragraph{How to read the pictures.}
In every matrix picture a \blue{blue} square is the entry $+1$, a \red{red} square is $-1$ and a white square is $0$. Thick lines separate blocks. Each matrix labelled $\W(n,k)$ was checked to satisfy $WW\T=kI_n$ with exact integer arithmetic before being drawn. Auxiliary matrices (such as $Q$ in Section~\ref{sec:paley}, or single circulants in Section~\ref{sec:gs}) are drawn to explain the constructions and are not weighing matrices.

\paragraph{Plan.}
Section~\ref{sec:tools} fixes notation and proves two elementary rules about block matrices. Section~\ref{sec:gs} explains the Goethals--Seidel construction in detail. Section~\ref{sec:algo} describes the algorithm and every method it uses, each with a proof and an example. Section~\ref{sec:verify} explains the exact verification, Section~\ref{sec:results} gives the results and Section~\ref{sec:availability} the availability of code and data.

\section{Notation and elementary tools}\label{sec:tools}

\subsection{Notation}
Rows and columns of an $n\times n$ matrix are numbered $0,1,\dots,n-1$ and $X_{ij}$ is the entry in row $i$ and column $j$. We write $X\T$ for the transpose ($(X\T)_{ij}=X_{ji}$), $I_n$ for the identity matrix, $J$ for a matrix all of whose entries are $1$, $\one$ for a column vector all of whose entries are $1$, and $0$ for a zero matrix or vector whose size is clear from the context. Table~\ref{tab:notation} lists the symbols used later; each is also defined where it first appears.

\begin{table}[ht]
\centering\small
\begin{tabular}{@{}p{2.2cm}p{12.6cm}@{}}
\toprule
$\W(n,k)$ & weighing matrix of order $n$ and weight $k$ (the W is set in upright type, not italic)\\
$H$, $H_2$ & always a \emph{Hadamard} matrix ($H_2$ is the $2\times2$ matrix defined in the Introduction, Section~\ref{sec:intro})\\
$\diag(X,Y)$, $X\oplus Y$ & block diagonal matrix (direct sum), Section~\ref{sec:blocks}\\
$X\otimes Y$ & Kronecker product, Section~\ref{sec:blocks}\\
$A,B,C,D$ & four \emph{sequences} of length $m$ with entries in $\{-1,0,1\}$ (Section~\ref{sec:gs}); capital letters are used for sequences only\\
$a,b,c,d$ & the $m\times m$ circulant \emph{matrices} built from $A,B,C,D$\\
$\circm(x)$, $Z$, $R$ & circulant of a sequence $x$; cyclic shift matrix; reversal matrix\\
$G$ & Goethals--Seidel array ($n=4m$)\\
$\PAF_x(s)$, $P(s)$ & periodic autocorrelation of one sequence; sum over the four sequences\\
$w_i$, $s_i$ & number of nonzero entries, and sum of entries, of the $i$-th sequence\\
$q$, $\chi$, $Q$ & prime power, quadratic character, character matrix (Paley constructions)\\
$S$, $K$ & skew (resp.\ symmetric) conference matrix of Paley type\\
$E$, $N$, $T$ & annealing: cost, number of moves per attempt, temperature\\
\bottomrule
\end{tabular}
\caption{Notation.}\label{tab:notation}
\end{table}

\subsection{Block matrices, direct sums and Kronecker products}\label{sec:blocks}
A matrix can be cut into rectangular \emph{blocks}, which are then treated like entries. Two rules are used throughout. First, the \emph{transpose} of a block matrix is obtained by transposing the arrangement of the blocks \emph{and} each block. Second, the product of two block matrices is computed with the usual row-by-column rule applied to blocks, \emph{provided the block sizes fit and the order of the factors is kept}: if $X=(X_{ij})$ and $Y=(Y_{jl})$ are block matrices, the $(i,l)$ block of $XY$ is $\sum_j X_{ij}Y_{jl}$.

For square matrices $X$ and $Y$ the \emph{direct sum} (block diagonal matrix) is
\[
\diag(X,Y)=X\oplus Y=\begin{pmatrix}X&0\\0&Y\end{pmatrix},
\]
with zero blocks of the appropriate sizes; $\diag(X_1,\dots,X_r)$ is defined in the same way.

\begin{definition}
The \emph{Kronecker product} of an $a\times a$ matrix $X$ and a $b\times b$ matrix $Y$ is the $ab\times ab$ block matrix $X\otimes Y$ whose $(i,j)$ block, of size $b\times b$, is the matrix $X_{ij}\,Y$. In other words, \emph{every entry of $X$ is replaced by a copy of $Y$ multiplied by that entry}.
\end{definition}
For example
\[
H_2\otimes I_2=\begin{pmatrix}I_2&I_2\\I_2&-I_2\end{pmatrix}=\left(\begin{smallmatrix}1&0&1&0\\0&1&0&1\\1&0&-1&0\\0&1&0&-1\end{smallmatrix}\right),\qquad
I_2\otimes H_2=\begin{pmatrix}H_2&0\\0&H_2\end{pmatrix}=\left(\begin{smallmatrix}1&1&0&0\\1&-1&0&0\\0&0&1&1\\0&0&1&-1\end{smallmatrix}\right).
\]
The two products differ: \emph{the order of the two factors matters}. Two consequences of the definition: $I_a\otimes Y=\diag(Y,\dots,Y)$ ($a$ copies), because the $(i,j)$ block is $\delta_{ij}Y$ (with $\delta_{ij}=1$ if $i=j$ and $0$ otherwise); and $I_a\otimes I_b=I_{ab}$.

\begin{lemma}\label{lem:kron}
Let $X,Z$ be $a\times a$ matrices and $Y,U$ be $b\times b$ matrices. Then $(X\otimes Y)(Z\otimes U)=(XZ)\otimes(YU)$ and $(X\otimes Y)\T=X\T\otimes Y\T$.
\end{lemma}
\begin{proof}
By the block rule, the $(i,l)$ block of $(X\otimes Y)(Z\otimes U)$ is
\[
\sum_{j}(X_{ij}Y)(Z_{jl}U)=\Bigl(\sum_j X_{ij}Z_{jl}\Bigr)YU=(XZ)_{il}\,YU,
\]
because the numbers $X_{ij}$ and $Z_{jl}$ can be moved in front of the matrices. This is the $(i,l)$ block of $(XZ)\otimes(YU)$. For the transpose, the $(i,j)$ block of $(X\otimes Y)\T$ is the transpose of the $(j,i)$ block of $X\otimes Y$, namely $(X_{ji}Y)\T=X_{ji}Y\T=(X\T)_{ij}Y\T$, which is the $(i,j)$ block of $X\T\otimes Y\T$.
\end{proof}
\section{The Goethals--Seidel construction}\label{sec:gs}

Most of the catalogue comes from classical constructions (Section~\ref{sec:algo}), but the hardest pairs are solved with a construction of Goethals and Seidel~\cite{gs70} (for its use with ternary sequences to build weighing matrices, see also~\cite{ks93}). It builds a large matrix from \emph{four short sequences}, and turns the question ``is this large matrix a weighing matrix?'' into a question about the sequences alone. This section explains the construction and proves everything it needs.

\subsection{Circulant matrices}\label{sec:circ}

A \emph{sequence of length $m$} is a list $x=(x_0,x_1,\dots,x_{m-1})$ of $m$ numbers; in this note its entries always belong to $\{-1,0,1\}$.

\begin{definition}
The \emph{circulant} of $x$, written $\circm(x)$, is the $m\times m$ matrix whose entry in row $i$ and column $j$ ($0\le i,j\le m-1$) is
\[
\circm(x)_{ij}=x_{(j-i)\bmod m}.
\]
\end{definition}
Here $(j-i)\bmod m$ is the remainder of $j-i$ on division by $m$, a number in $\{0,\dots,m-1\}$; for instance, for $m=4$, $(0-1)\bmod 4=3$ and $(2-3)\bmod 4=3$, while $(3-1)\bmod 4=2$. The index $(j-i)\bmod m$ depends only on $j-i$ (modulo $m$), so it is constant along each diagonal of the matrix, with wrap-around. For $m=4$ the indices and the resulting matrix are
\[
\bigl((j-i)\bmod 4\bigr)_{ij}=\begin{pmatrix}0&1&2&3\\3&0&1&2\\2&3&0&1\\1&2&3&0\end{pmatrix}
\quad\Longrightarrow\quad
\circm(x)=\begin{pmatrix}x_0&x_1&x_2&x_3\\x_3&x_0&x_1&x_2\\x_2&x_3&x_0&x_1\\x_1&x_2&x_3&x_0\end{pmatrix}.
\]
In words: row $0$ is $x$ itself, and each following row is the previous row moved one place to the right, the entry that falls off the right end re-entering at the left (a \emph{cyclic shift}). This is why every row, and every column, contains the same entries as $x$. As an example, for $m=3$ and the sequence $A=(-1,-1,1)$ we obtain
\[
a=\circm(A)=\begin{pmatrix}-1&-1&1\\1&-1&-1\\-1&1&-1\end{pmatrix}
\qquad(\text{drawn in Figure~\ref{fig:circ}}).
\]

Let $Z=\circm(0,1,0,\dots,0)$ be the circulant of the sequence with a single $1$ in position $1$; $Z_{ij}=1$ exactly when $j\equiv i+1\pmod m$. (For $m=1$ everything below is trivial: all matrices are $1\times1$.)

\begin{lemma}\label{lem:circ}
Let $m\ge2$.
\begin{enumerate}
\item[(i)] For every $t\ge0$, $(Z^t)_{ij}=1$ if $j\equiv i+t\pmod m$ and $0$ otherwise. In particular $Z^m=I_m$ and $Z\T=Z^{m-1}$.
\item[(ii)] $\circm(x)=\sum_{t=0}^{m-1}x_tZ^t$.
\item[(iii)] The transpose of a circulant is a circulant, a product of two circulants is a circulant, and any two circulants commute: $XY=YX$. Consequently every circulant $X$ commutes with every circulant $Y$ and with $Y\T$, and $XX\T=X\T X$.
\end{enumerate}
\end{lemma}
\begin{proof}
(i) By induction on $t$; $Z^0=I_m$ satisfies the claim. If it holds for $t$ then $(Z^{t+1})_{ij}=\sum_l(Z^t)_{il}Z_{lj}$; the only nonzero term has $l\equiv i+t$ and requires $j\equiv l+1\equiv i+t+1$, in which case it equals $1$. For $t=m$ we get $j\equiv i$, i.e.\ $Z^m=I_m$, and $(Z^{m-1})_{ij}=1$ iff $j\equiv i-1$ iff $i\equiv j+1$ iff $(Z\T)_{ij}=Z_{ji}=1$.

(ii) The $(i,j)$ entry of $\sum_t x_tZ^t$ is $x_t$ for the unique $t\in\{0,\dots,m-1\}$ with $t\equiv j-i$, that is, $x_{(j-i)\bmod m}$.

(iii) By (ii) a circulant is a polynomial in $Z$. Since $Z^tZ^u=Z^{t+u}=Z^uZ^t$, any two polynomials in $Z$ commute. A product of polynomials in $Z$ is again a polynomial in $Z$, and, because $Z^m=I_m$ (so $Z^r=Z^{r\bmod m}$), every polynomial in $Z$ equals $\sum_{t=0}^{m-1}c_tZ^t=\circm(c_0,\dots,c_{m-1})$ for suitable numbers $c_t$. Finally $X\T=\sum_tx_t(Z\T)^t=\sum_tx_tZ^{(m-1)t}$ is again a polynomial in $Z$. All remaining claims follow.
\end{proof}

\subsection{The reversal matrix $R$}\label{sec:R}

\begin{definition}
The \emph{reversal matrix} $R$ is the $m\times m$ matrix with $R_{ij}=1$ if $i+j=m-1$ and $R_{ij}=0$ otherwise (ones on the anti-diagonal).
\end{definition}
For $m=4$,
\[
R=\begin{pmatrix}0&0&0&1\\0&0&1&0\\0&1&0&0\\1&0&0&0\end{pmatrix}.
\]
Multiplying a column vector by $R$ reverses the order of its entries ($(Rv)_i=v_{m-1-i}$); $XR$ reverses the order of the \emph{columns} of $X$ and $RX$ the order of its \emph{rows}.

\begin{lemma}\label{lem:R}
(i) $R\T=R$. (ii) $R^2=I_m$. (iii) For every circulant $X$, $XR=RX\T$ (hence also $RX=X\T R$, by applying (iii) to the circulant $X\T$).
\end{lemma}
\begin{proof}
(i) The condition $i+j=m-1$ is symmetric in $i$ and $j$.

(ii) $(R^2)_{ij}=\sum_lR_{il}R_{lj}$. In row $i$ of $R$ the only nonzero entry is $R_{i,\,m-1-i}=1$, so the sum has the single term $l=m-1-i$, equal to $R_{m-1-i,\,j}$, which is $1$ exactly when $(m-1-i)+j=m-1$, i.e.\ $j=i$. Hence $(R^2)_{ij}=1$ if $i=j$ and $0$ otherwise.

(iii) Let $X=\circm(x)$. Since $R_{lj}=1$ only for $l=m-1-j$, $(XR)_{ij}=X_{i,\,m-1-j}=x_{(m-1-j-i)\bmod m}$. Since $R_{il}=1$ only for $l=m-1-i$, $(RX\T)_{ij}=(X\T)_{m-1-i,\,j}=X_{j,\,m-1-i}=x_{(m-1-i-j)\bmod m}$. The two entries coincide.
\end{proof}
Here is (iii) for $m=4$: reversing the columns of $X=\circm(x)$ or reversing the rows of $X\T$ gives the same matrix,
\[
XR=\begin{pmatrix}x_3&x_2&x_1&x_0\\x_2&x_1&x_0&x_3\\x_1&x_0&x_3&x_2\\x_0&x_3&x_2&x_1\end{pmatrix}=RX\T,\qquad
X\T=\begin{pmatrix}x_0&x_3&x_2&x_1\\x_1&x_0&x_3&x_2\\x_2&x_1&x_0&x_3\\x_3&x_2&x_1&x_0\end{pmatrix}.
\]

\subsection{The Goethals--Seidel array}\label{sec:array}

Let $A,B,C,D$ be four sequences of length $m$ and let $a=\circm(A)$, $b=\circm(B)$, $c=\circm(C)$, $d=\circm(D)$ be their circulants. \emph{Capital letters are always sequences (lists of $m$ numbers); lower-case letters are always the corresponding $m\times m$ circulant matrices.} The sequences are the unknowns of the problem. The \emph{Goethals--Seidel array} is the $n\times n$ matrix, with $n=4m$, cut into sixteen blocks of size $m\times m$,
\begin{equation}\label{eq:G}
G=\begin{pmatrix}
a & bR & cR & dR\\
-bR & a & d\T R & -c\T R\\
-cR & -d\T R & a & b\T R\\
-dR & c\T R & -b\T R & a
\end{pmatrix}.
\end{equation}
Each of the sixteen blocks is $\pm$ a circulant, possibly transposed and/or with columns reversed (a block is the zero matrix exactly when the corresponding sequence consists of zeros only, which does happen for some matrices of the catalogue), so in every case every entry of $G$ lies in $\{-1,0,1\}$. The only question is whether $GG\T=kI_n$ for some $k$.

\begin{center}
\begin{tikzpicture}[x=1.7cm,y=-1.0cm,font=\small]
\node[draw=gray,fill=gray!12,minimum width=1.7cm,minimum height=1.0cm] at (0.5,0.5) {\color{black}$a$};
\node[draw=gray,fill=white,minimum width=1.7cm,minimum height=1.0cm] at (1.5,0.5) {\color{black}$bR$};
\node[draw=gray,fill=white,minimum width=1.7cm,minimum height=1.0cm] at (2.5,0.5) {\color{black}$cR$};
\node[draw=gray,fill=white,minimum width=1.7cm,minimum height=1.0cm] at (3.5,0.5) {\color{black}$dR$};
\node[draw=gray,fill=white,minimum width=1.7cm,minimum height=1.0cm] at (0.5,1.5) {\color{wminus}$-bR$};
\node[draw=gray,fill=gray!12,minimum width=1.7cm,minimum height=1.0cm] at (1.5,1.5) {\color{black}$a$};
\node[draw=gray,fill=white,minimum width=1.7cm,minimum height=1.0cm] at (2.5,1.5) {\color{black}$d^{\top}R$};
\node[draw=gray,fill=white,minimum width=1.7cm,minimum height=1.0cm] at (3.5,1.5) {\color{wminus}$-c^{\top}R$};
\node[draw=gray,fill=white,minimum width=1.7cm,minimum height=1.0cm] at (0.5,2.5) {\color{wminus}$-cR$};
\node[draw=gray,fill=white,minimum width=1.7cm,minimum height=1.0cm] at (1.5,2.5) {\color{wminus}$-d^{\top}R$};
\node[draw=gray,fill=gray!12,minimum width=1.7cm,minimum height=1.0cm] at (2.5,2.5) {\color{black}$a$};
\node[draw=gray,fill=white,minimum width=1.7cm,minimum height=1.0cm] at (3.5,2.5) {\color{black}$b^{\top}R$};
\node[draw=gray,fill=white,minimum width=1.7cm,minimum height=1.0cm] at (0.5,3.5) {\color{wminus}$-dR$};
\node[draw=gray,fill=white,minimum width=1.7cm,minimum height=1.0cm] at (1.5,3.5) {\color{black}$c^{\top}R$};
\node[draw=gray,fill=white,minimum width=1.7cm,minimum height=1.0cm] at (2.5,3.5) {\color{wminus}$-b^{\top}R$};
\node[draw=gray,fill=gray!12,minimum width=1.7cm,minimum height=1.0cm] at (3.5,3.5) {\color{black}$a$};
\draw[line width=0.8pt] (0,0) rectangle (4,4);
\node[font=\footnotesize,anchor=west] at (4.15,0.5) {block-row 1};\node[font=\footnotesize,anchor=west] at (4.15,1.5) {block-row 2};\node[font=\footnotesize,anchor=west] at (4.15,2.5) {block-row 3};\node[font=\footnotesize,anchor=west] at (4.15,3.5) {block-row 4};
\end{tikzpicture}\hspace{1.2cm}\begin{tikzpicture}[x=0.8cm,y=-0.8cm,font=\small]
\node[draw=gray,fill=gray!12,minimum width=0.8cm,minimum height=0.8cm] at (0.5,0.5) {$\Gamma$};
\node[draw=gray,fill=white,minimum width=0.8cm,minimum height=0.8cm] at (1.5,0.5) {$0$};
\node[draw=gray,fill=white,minimum width=0.8cm,minimum height=0.8cm] at (2.5,0.5) {$0$};
\node[draw=gray,fill=white,minimum width=0.8cm,minimum height=0.8cm] at (3.5,0.5) {$0$};
\node[draw=gray,fill=white,minimum width=0.8cm,minimum height=0.8cm] at (0.5,1.5) {$0$};
\node[draw=gray,fill=gray!12,minimum width=0.8cm,minimum height=0.8cm] at (1.5,1.5) {$\Gamma$};
\node[draw=gray,fill=white,minimum width=0.8cm,minimum height=0.8cm] at (2.5,1.5) {$0$};
\node[draw=gray,fill=white,minimum width=0.8cm,minimum height=0.8cm] at (3.5,1.5) {$0$};
\node[draw=gray,fill=white,minimum width=0.8cm,minimum height=0.8cm] at (0.5,2.5) {$0$};
\node[draw=gray,fill=white,minimum width=0.8cm,minimum height=0.8cm] at (1.5,2.5) {$0$};
\node[draw=gray,fill=gray!12,minimum width=0.8cm,minimum height=0.8cm] at (2.5,2.5) {$\Gamma$};
\node[draw=gray,fill=white,minimum width=0.8cm,minimum height=0.8cm] at (3.5,2.5) {$0$};
\node[draw=gray,fill=white,minimum width=0.8cm,minimum height=0.8cm] at (0.5,3.5) {$0$};
\node[draw=gray,fill=white,minimum width=0.8cm,minimum height=0.8cm] at (1.5,3.5) {$0$};
\node[draw=gray,fill=white,minimum width=0.8cm,minimum height=0.8cm] at (2.5,3.5) {$0$};
\node[draw=gray,fill=gray!12,minimum width=0.8cm,minimum height=0.8cm] at (3.5,3.5) {$\Gamma$};
\draw[line width=0.8pt] (0,0) rectangle (4,4);\node[font=\footnotesize] at (2,4.5) {$GG^{\top}=I_4\otimes\Gamma$};
\end{tikzpicture}
\end{center}

\subsection{The key identity}\label{sec:key}

\begin{theorem}[Goethals--Seidel]\label{thm:gs}
Let $\Gamma=aa\T+bb\T+cc\T+dd\T$. Then
\[
GG\T=I_4\otimes\Gamma=\begin{pmatrix}\Gamma&0&0&0\\0&\Gamma&0&0\\0&0&\Gamma&0\\0&0&0&\Gamma\end{pmatrix}.
\]
\end{theorem}
(The matrix $I_4$ appears because $G$ is cut into $4\times4$ blocks: by the definition of $\otimes$, $I_4\otimes\Gamma$ has the block $\Gamma$ in the four diagonal positions and $0$ elsewhere.)

\begin{proof}
We use four rules: (R1) $(XR)\T=R\T X\T=RX\T$ (Lemma~\ref{lem:R}(i)); (R2) $RR=I_m$; (R3) $RY\T=YR$ and $RY=Y\T R$ for every circulant $Y$ (Lemma~\ref{lem:R}(iii)); (R4) the circulants $a,b,c,d$ and their transposes commute with each other (Lemma~\ref{lem:circ}).

The block in block-row $r$ and block-column $s$ of $GG\T$ is $\sum_{c=1}^4G_{rc}G_{sc}\T$, where $G_{rc}$ denotes the block of $G$ in block-row $r$ and block-column $c$ (``block-row $r$ times the transpose of block-row $s$''). This uses the rule of Section~\ref{sec:blocks}: the transpose of $G$ has in position $(c,s)$ the block $G_{sc}\T$; that is, transposing a block matrix swaps the positions of the blocks \emph{and also} transposes each block.

\emph{Block $(1,2)$.} The four products are
\begin{align*}
a(-bR)\T&=-aRb\T=-abR &&\text{by (R1), then (R3): }Rb\T=bR,\\
(bR)a\T&=bRa\T=baR &&\text{by (R3): }Ra\T=aR,\\
(cR)(d\T R)\T&=cR\,Rd=cd &&\text{by (R1) with }(d\T)\T=d,\ \text{then (R2)},\\
(dR)(-c\T R)\T&=-dR\,Rc=-dc &&\text{in the same way.}
\end{align*}
Their sum is $-abR+baR+cd-dc=0$, because $ab=ba$ and $cd=dc$ by (R4). The other five blocks above the diagonal are treated identically; Table~\ref{tab:offdiag} lists the four products of each, and their simplified sum, which vanishes by (R4) each time. Since $GG\T$ is symmetric, the blocks below the diagonal are the transposes of those above, hence $0$.

\begin{table}[ht]
\centering\small
\renewcommand{\arraystretch}{1.25}
\begin{tabular}{@{}c>{\raggedright\arraybackslash}p{9.2cm}p{4.2cm}@{}}
\toprule
block & the four products $G_{rc}G_{sc}\T$, $c=1,2,3,4$ & sum after (R1)--(R3)\\
\midrule
$(1,2)$ & $a(-bR)\T,\ (bR)a\T,\ (cR)(d\T R)\T,\ (dR)(-c\T R)\T$ & $(ba-ab)R+(cd-dc)$\\
$(1,3)$ & $a(-cR)\T,\ (bR)(-d\T R)\T,\ (cR)a\T,\ (dR)(b\T R)\T$ & $(ca-ac)R+(db-bd)$\\
$(1,4)$ & $a(-dR)\T,\ (bR)(c\T R)\T,\ (cR)(-b\T R)\T,\ (dR)a\T$ & $(da-ad)R+(bc-cb)$\\
$(2,3)$ & $(-bR)(-cR)\T,\ a(-d\T R)\T,\ (d\T R)a\T,\ (-c\T R)(b\T R)\T$ & $(bc\T-c\T b)+(d\T a-ad\T)R$\\
$(2,4)$ & $(-bR)(-dR)\T,\ a(c\T R)\T,\ (d\T R)(-b\T R)\T,\ (-c\T R)a\T$ & $(bd\T-d\T b)+(ac\T-c\T a)R$\\
$(3,4)$ & $(-cR)(-dR)\T,\ (-d\T R)(c\T R)\T,\ a(-b\T R)\T,\ (b\T R)a\T$ & $(cd\T-d\T c)+(b\T a-ab\T)R$\\
\bottomrule
\end{tabular}
\caption{Off-diagonal blocks of $GG\T$. Every bracket is of the form $XY-YX$ with $X,Y$ circulants or transposed circulants, hence $0$.}\label{tab:offdiag}
\end{table}

\emph{Diagonal blocks.} Using (R1), (R2) and $(Y\T)\T=Y$:
\begin{align*}
(1,1):\quad& aa\T+(bR)(bR)\T+(cR)(cR)\T+(dR)(dR)\T\\
&=aa\T+bRRb\T+cRRc\T+dRRd\T=aa\T+bb\T+cc\T+dd\T=\Gamma,\\
(2,2):\quad& (-bR)(-bR)\T+aa\T+(d\T R)(d\T R)\T+(-c\T R)(-c\T R)\T\\
&=bb\T+aa\T+d\T d+c\T c=\Gamma,\\
(3,3):\quad& (-cR)(-cR)\T+(-d\T R)(-d\T R)\T+aa\T+(b\T R)(b\T R)\T=cc\T+d\T d+aa\T+b\T b=\Gamma,\\
(4,4):\quad& (-dR)(-dR)\T+(c\T R)(c\T R)\T+(-b\T R)(-b\T R)\T+aa\T=dd\T+c\T c+b\T b+aa\T=\Gamma,
\end{align*}
where in the last three lines $d\T d=dd\T$, $c\T c=cc\T$, $b\T b=bb\T$ because a circulant commutes with its transpose (Lemma~\ref{lem:circ}(iii)).
\end{proof}

\subsection{Periodic autocorrelation}\label{sec:paf}

Theorem~\ref{thm:gs} says that everything is decided by the $m\times m$ matrix $\Gamma$. It is described by numbers that measure how much a sequence resembles its own cyclic shifts.

\begin{definition}
The \emph{periodic autocorrelation function} (PAF) of a sequence $x$ of length $m$ is
\[
\PAF_x(s)=\sum_{t=0}^{m-1}x_t\,x_{(t+s)\bmod m},\qquad s=0,1,\dots,m-1,
\]
and for four sequences we put $P(s)=\PAF_A(s)+\PAF_B(s)+\PAF_C(s)+\PAF_D(s)$.
\end{definition}
In words, $\PAF_x(s)$ is the dot product of $x$ with $x$ cyclically shifted by $s$ places. For $x=A=(-1,-1,1)$:
\[
\begin{array}{r@{\;}c@{\;}c@{\;}c@{\quad}l}
x&=(-1,&-1,&\phantom{-}1)&\\
\text{shift by }1:&(-1,&\phantom{-}1,&-1)&\PAF_A(1)=(-1)(-1)+(-1)(1)+(1)(-1)=-1,\\
\text{shift by }2:&(\phantom{-}1,&-1,&-1)&\PAF_A(2)=(-1)(1)+(-1)(-1)+(1)(-1)=-1,
\end{array}
\]
and $\PAF_A(0)=1+1+1=3$, the number of nonzero entries. In general $\PAF_x(0)=\sum_tx_t^2$ is the number of nonzero entries of $x$ (because, the entries being $-1$, $0$ or $1$, the square $x_t^2$ is $1$ if $x_t\ne0$ and $0$ if $x_t=0$), and $\PAF_x(s)=\PAF_x(m-s)$ for $1\le s\le m-1$ (replace $t$ by $t-s$ in the sum, with indices taken modulo $m$).

\begin{lemma}\label{lem:gamma}
$\Gamma=\circm\bigl(P(0),P(1),\dots,P(m-1)\bigr)$.
\end{lemma}
\begin{proof}
By Lemma~\ref{lem:circ}(iii), $\Gamma$ is a circulant, so it suffices to compute row $0$. For $x$ with circulant $X$, $(XX\T)_{0j}=\sum_lX_{0l}X_{jl}=\sum_l x_l\,x_{(l-j)\bmod m}$. Putting $t=l-j$ this is $\sum_t x_{(t+j)\bmod m}\,x_t=\PAF_x(j)$. Adding the four sequences gives $\Gamma_{0j}=P(j)$.
\end{proof}

\begin{corollary}\label{cor:paf}
For $k\ge 1$: $GG\T=kI_{4m}\iff P(0)=k$ and $P(s)=0$ for $1\le s\le m-1$. In that case $G$ is a weighing matrix $\W(4m,k)$.
\end{corollary}
\begin{proof}
By Theorem~\ref{thm:gs}, $GG\T=kI_{4m}$ iff $\Gamma=kI_m$, and by Lemma~\ref{lem:gamma} the matrix $\Gamma$ equals $kI_m$ iff its first row $(P(0),\dots,P(m-1))$ equals $(k,0,\dots,0)$.
\end{proof}
This is the point of the construction: instead of testing the orthogonality of the rows of an $n\times n$ matrix, one only has to make the $m-1$ numbers $P(1),\dots,P(m-1)$ vanish, and to have $P(0)=k$.

\subsection{A worked example: $\W(12,7)$}\label{sec:ex127}

Take $m=3$ and the four sequences stored in the catalogue for $\W(12,7)$ (found by the exhaustive search of Section~\ref{sec:exh}):
\[
A=(-1,-1,1),\quad B=(-1,0,0),\quad C=(-1,-1,0),\quad D=(-1,0,0).
\]
Their circulants are drawn in Figure~\ref{fig:circ}. Their PAFs are
\begin{center}
\begin{tabular}{lccc}
\toprule
 & $s=0$ & $s=1$ & $s=2$ \\
\midrule
$\PAF_A$ & $3$ & $-1$ & $-1$ \\
$\PAF_B$ & $1$ & $0$ & $0$ \\
$\PAF_C$ & $2$ & $1$ & $1$ \\
$\PAF_D$ & $1$ & $0$ & $0$ \\
\midrule
$P$ & $7$ & $0$ & $0$ \\
\bottomrule
\end{tabular}
\end{center}
so $P(0)=7$ and $P(1)=P(2)=0$: by Corollary~\ref{cor:paf}, $G$ is a $\W(12,7)$. We can also see it directly: by Lemma~\ref{lem:gamma}, $\Gamma=\circm(7,0,0)=7I_3$, since
\[
\underbrace{\left(\begin{smallmatrix}3&-1&-1\\-1&3&-1\\-1&-1&3\end{smallmatrix}\right)}_{aa\T}+\underbrace{\left(\begin{smallmatrix}1&0&0\\0&1&0\\0&0&1\end{smallmatrix}\right)}_{bb\T}+\underbrace{\left(\begin{smallmatrix}2&1&1\\1&2&1\\1&1&2\end{smallmatrix}\right)}_{cc\T}+\underbrace{\left(\begin{smallmatrix}1&0&0\\0&1&0\\0&0&1\end{smallmatrix}\right)}_{dd\T}=\left(\begin{smallmatrix}7&0&0\\0&7&0\\0&0&7\end{smallmatrix}\right)
\]
and $GG\T=I_4\otimes 7I_3=7I_{12}$. Figure~\ref{fig:circ} shows $a,b,c,d$ and the $12\times12$ matrix $G$.

\begin{figure}[ht]
\centering
\begin{tikzpicture}[x=0.400cm,y=0.400cm,baseline=(mid)]
\coordinate (mid) at (1.5,-1.5);
\foreach \c/\r in {2/0,0/1,1/2}{\fill[wplus] (\c,-\r) rectangle ++(1,-1);}
\foreach \c/\r in {0/0,1/0,1/1,2/1,0/2,2/2}{\fill[wminus] (\c,-\r) rectangle ++(1,-1);}
\draw[xstep=1,ystep=1,gray!45,line width=0.12pt] (0,-3) grid (3,0);
\draw[line width=0.6pt] (0,0) rectangle (3,-3);
\node[font=\scriptsize,below,align=center] at (1.5,-3) {$a$};
\end{tikzpicture}
\hspace{0.25cm}
\begin{tikzpicture}[x=0.400cm,y=0.400cm,baseline=(mid)]
\coordinate (mid) at (1.5,-1.5);
\foreach \c/\r in {0/0,1/1,2/2}{\fill[wminus] (\c,-\r) rectangle ++(1,-1);}
\draw[xstep=1,ystep=1,gray!45,line width=0.12pt] (0,-3) grid (3,0);
\draw[line width=0.6pt] (0,0) rectangle (3,-3);
\node[font=\scriptsize,below,align=center] at (1.5,-3) {$b$};
\end{tikzpicture}
\hspace{0.25cm}
\begin{tikzpicture}[x=0.400cm,y=0.400cm,baseline=(mid)]
\coordinate (mid) at (1.5,-1.5);
\foreach \c/\r in {0/0,1/0,1/1,2/1,0/2,2/2}{\fill[wminus] (\c,-\r) rectangle ++(1,-1);}
\draw[xstep=1,ystep=1,gray!45,line width=0.12pt] (0,-3) grid (3,0);
\draw[line width=0.6pt] (0,0) rectangle (3,-3);
\node[font=\scriptsize,below,align=center] at (1.5,-3) {$c$};
\end{tikzpicture}
\hspace{0.25cm}
\begin{tikzpicture}[x=0.400cm,y=0.400cm,baseline=(mid)]
\coordinate (mid) at (1.5,-1.5);
\foreach \c/\r in {0/0,1/1,2/2}{\fill[wminus] (\c,-\r) rectangle ++(1,-1);}
\draw[xstep=1,ystep=1,gray!45,line width=0.12pt] (0,-3) grid (3,0);
\draw[line width=0.6pt] (0,0) rectangle (3,-3);
\node[font=\scriptsize,below,align=center] at (1.5,-3) {$d$};
\end{tikzpicture}
\hspace{0.90cm}
\begin{tikzpicture}[x=0.270cm,y=0.270cm,baseline=(mid)]
\coordinate (mid) at (6.0,-6.0);
\foreach \c/\r in {2/0,0/1,1/2,2/3,5/3,9/3,11/3,1/4,3/4,10/4,11/4,0/5,4/5,9/5,10/5,1/6,2/6,5/6,8/6,0/7,1/7,4/7,6/7,0/8,2/8,3/8,7/8,2/9,8/9,11/9,1/10,7/10,9/10,0/11,6/11,10/11}{\fill[wplus] (\c,-\r) rectangle ++(1,-1);}
\foreach \c/\r in {0/0,1/0,5/0,7/0,8/0,11/0,1/1,2/1,4/1,6/1,7/1,10/1,0/2,2/2,3/2,6/2,8/2,9/2,3/3,4/3,8/3,4/4,5/4,7/4,3/5,5/5,6/5,6/6,7/6,11/6,7/7,8/7,10/7,6/8,8/8,9/8,3/9,5/9,9/9,10/9,4/10,5/10,10/10,11/10,3/11,4/11,9/11,11/11}{\fill[wminus] (\c,-\r) rectangle ++(1,-1);}
\draw[xstep=1,ystep=1,gray!45,line width=0.12pt] (0,-12) grid (12,0);
\draw[line width=0.6pt] (3,0) -- (3,-12);
\draw[line width=0.6pt] (0,-3) -- (12,-3);
\draw[line width=0.6pt] (6,0) -- (6,-12);
\draw[line width=0.6pt] (0,-6) -- (12,-6);
\draw[line width=0.6pt] (9,0) -- (9,-12);
\draw[line width=0.6pt] (0,-9) -- (12,-9);
\draw[line width=0.6pt] (0,0) rectangle (12,-12);
\node[font=\scriptsize,below,align=center] at (6.0,-12) {$G=\W(12,7)$};
\end{tikzpicture}
\caption{The four circulants of the example and the resulting $\W(12,7)=G$ (stored matrix of the catalogue). Thick lines separate the sixteen blocks $3\times3$. The small matrices $a,b,c,d$ are auxiliary and are not weighing matrices.}\label{fig:circ}
\end{figure}
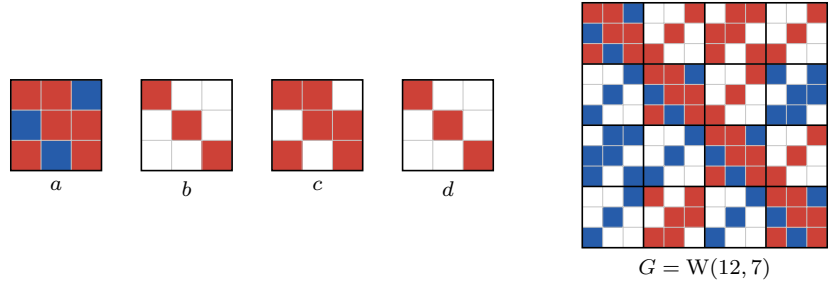

\subsection{Necessary conditions: the profile}\label{sec:profile}

For $i=0,1,2,3$ let $w_i$ be the number of nonzero entries of the $i$-th sequence and $s_i$ the sum of its entries (so $A,B,C,D$ have profiles $(w_0,s_0),\dots,(w_3,s_3)$; in the example above $w=(3,1,2,1)$ and $s=(-1,-1,-2,-1)$). Let us call the list $(w_i,s_i)_{i=0}^3$ the \emph{profile}. Each sequence has $u_i^{+}$ entries $+1$ and $u_i^{-}$ entries $-1$, so $w_i=u_i^{+}+u_i^{-}$ and $s_i=u_i^{+}-u_i^{-}$. Hence $w_i-s_i=2u_i^{-}$ is even; $|s_i|\le w_i$, because $|u_i^{+}-u_i^{-}|\le u_i^{+}+u_i^{-}$; and $w_i\le m$, because the sequence has only $m$ positions.

\begin{proposition}\label{prop:profile}
If $G$ is a $\W(4m,k)$, then
\begin{equation}\label{eq:profile}
\sum_iw_i=k,\qquad\sum_is_i^2=k,\qquad w_i\equiv s_i\pmod2,\qquad|s_i|\le w_i\le m .
\end{equation}
\end{proposition}
\begin{proof}
The last two conditions were just shown. By Corollary~\ref{cor:paf}, $P(0)=k$ and $P(s)=0$ for $s\ne0$. Now $P(0)=\sum_iw_i$, because $\PAF_x(0)$ is the number of nonzero entries of $x$. Moreover, for any sequence $x$,
\[
\sum_{s=0}^{m-1}\PAF_x(s)=\sum_{t=0}^{m-1}x_t\sum_{s=0}^{m-1}x_{(t+s)\bmod m}=\sum_tx_t\Bigl(\sum_ux_u\Bigr)=\Bigl(\sum_tx_t\Bigr)^2,
\]
because for fixed $t$ the index $(t+s)\bmod m$ runs through all residues as $s$ does. Summing over the four sequences, $\sum_{s}P(s)=\sum_is_i^2$; the left side equals $P(0)=k$.
\end{proof}
Negating a sequence does not change its PAF (each product $x_tx_{t+s}$ keeps its sign), and permuting the four sequences does not change $P$. Hence, when \emph{enumerating} profiles, one may assume $s_0\ge s_1\ge s_2\ge s_3\ge0$. We call a profile satisfying~\eqref{eq:profile} \emph{admissible}, and an \emph{admissible sum tuple} a tuple $(s_0,s_1,s_2,s_3)$ with $s_0\ge s_1\ge s_2\ge s_3\ge 0$ for which at least one choice of weights $w_i$ completes it to an admissible profile. The conditions~\eqref{eq:profile} are necessary, not sufficient. The ordering convention is only a device for enumeration: the profiles reported later (Table~\ref{tab:cases}) keep the order of the stored sequences and need not be sorted.

\begin{proposition}\label{prop:parity}
If $m$ is even, there is no Goethals--Seidel matrix $\W(4m,4m-1)$.
\end{proposition}
\begin{proof}
By~\eqref{eq:profile}, $\sum_iw_i=4m-1$ with every $w_i\le m$, so exactly one $w_i$ equals $m-1$ and the other three equal $m$. For $m$ even, three of the $w_i$ are even and one is odd, hence (since $s_i\equiv w_i$) three of the $s_i$ are even and one is odd. An even square is $\equiv0$ and an odd square is $\equiv1\pmod 4$, so $\sum s_i^2\equiv1\pmod4$, whereas $k=4m-1\equiv3\pmod4$. Contradiction.
\end{proof}
In the range $n\le108$, Proposition~\ref{prop:parity} concerns the thirteen pairs $(8j,8j-1)$, $1\le j\le13$. Seven of them ($n=8,24,32,48,72,80,104$, where $n-1$ is a prime power $\equiv3\pmod4$) are obtained by Paley's construction, and the remaining six, $(16,15)$, $(40,39)$, $(56,55)$, $(64,63)$, $(88,87)$, $(96,95)$, by skew-Hadamard doubling (Section~\ref{sec:skew}). The proposition is a statement about mathematics; \emph{how the program notices} the impossibility for these six pairs is described in Section~\ref{sec:filter} (for $(16,15)$ the exhaustive search finds no solution, for the other five the profile filter finds no admissible profile).
\section{The algorithm}\label{sec:algo}

\subsection{Overview}\label{sec:overview}

The main program, \texttt{all\_chains.py}, processes the orders $n=4,8,\dots,108$ in increasing order and, for each $n$, the weights $k=1,2,\dots,n$. For each pair $(n,k)$ it tries the methods below \emph{in the order listed}; the first method that produces a matrix passing the exact check of Section~\ref{sec:verify} \emph{wins}: the matrix is written to the \emph{store} (one file \texttt{solutions/W\{n\}\_\{k\}/solution\_001.json} containing the matrix, the name of the method and a \emph{recipe}, i.e.\ the data needed to rebuild it) and none of the remaining methods is tried for this pair. Figure~\ref{fig:flow} shows the logic. A matrix already in the store is called \emph{available}; the direct-sum and Kronecker methods build new matrices from available ones, which is possible because smaller orders are processed first. The small matrices $\W(1,1)=(1)$, $\W(2,1)=I_2$ and $\W(2,2)=H_2$ (the \emph{atoms}) are always available. Only $\W(2,1)$ and $\W(2,2)$ can serve as factors of the Kronecker method (with $a=2$); $\W(1,1)$ is included in the code for completeness, but no method uses it. For $k=2$ the method \texttt{h2blocks} also uses $H_2$, since it is the product $I_{n/2}\otimes H_2$ (Method~1).

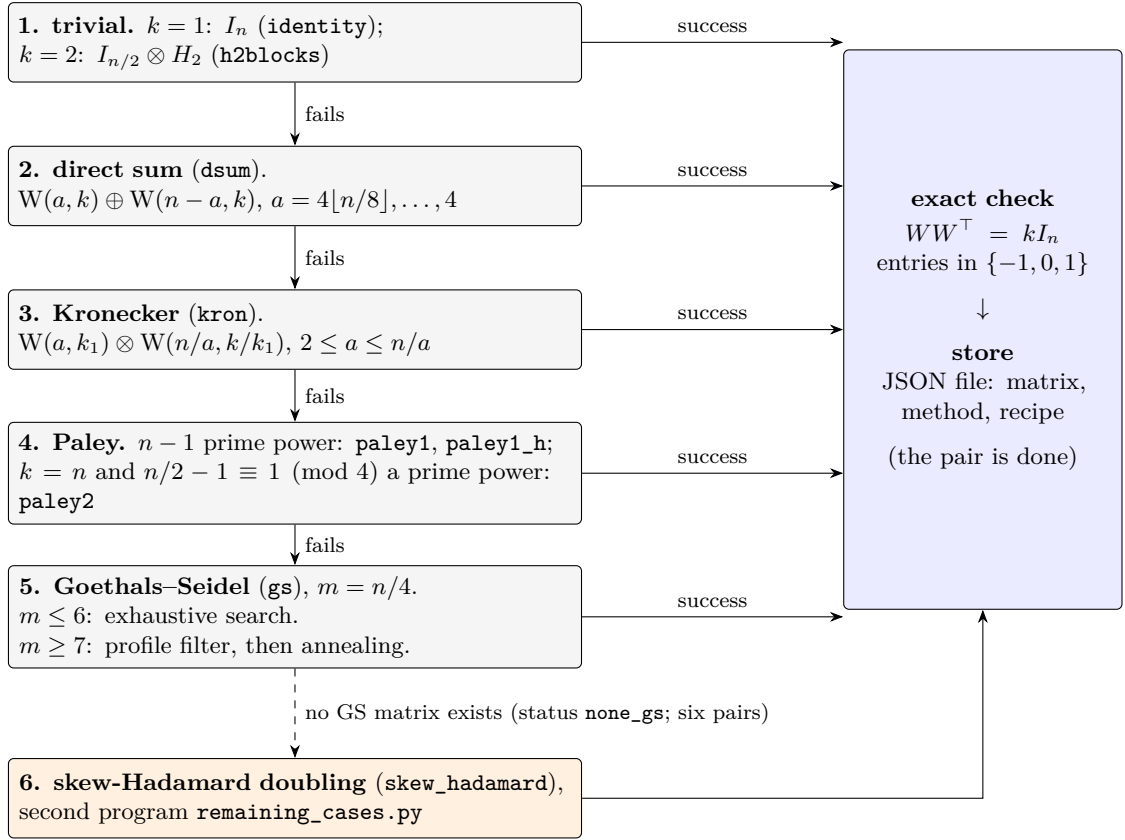
\begin{figure}[ht]
\centering

\begin{tikzpicture}[font=\footnotesize,
  box/.style={draw,rounded corners=2pt,align=left,minimum height=1.05cm,text width=7.3cm,fill=gray!8,inner sep=4pt},
  big/.style={draw,rounded corners=2pt,align=center,fill=blue!8,text width=3.4cm,inner sep=4pt},
  arr/.style={-{Stealth[length=5pt]}}]
\node[font=\footnotesize] (st) at (0,1.2) {pair $(n,k)$ (orders in increasing order; for each $n$, $k=1,\dots,n$)};
\node[box] (m1) at (0,0) {\textbf{1. trivial.} $k=1$: $I_n$ (\texttt{identity});\\ $k=2$: $I_{n/2}\otimes H_2$ (\texttt{h2blocks})};
\node[box] (m2) at (0,-1.9) {\textbf{2. direct sum} (\texttt{dsum}).\\ $\W(a,k)\oplus\W(n-a,k)$, $a=4\lfloor n/8\rfloor,\dots,4$};
\node[box] (m3) at (0,-3.8) {\textbf{3. Kronecker} (\texttt{kron}).\\ $\W(a,k_1)\otimes\W(n/a,k/k_1)$, $2\le a\le n/a$};
\node[box] (m4) at (0,-5.7) {\textbf{4. Paley.} $n-1$ prime power: \texttt{paley1}, \texttt{paley1\_h};\\ $k=n$ and $n/2-1\equiv1\pmod 4$ a prime power: \texttt{paley2}};
\node[box] (m5) at (0,-7.6) {\textbf{5. Goethals--Seidel} (\texttt{gs}), $m=n/4$.\\ $m\le6$: exhaustive search.\\ $m\ge7$: profile filter, then annealing.};
\node[box,fill=orange!12] (m6) at (0,-10.0) {\textbf{6. skew-Hadamard doubling} (\texttt{skew\_hadamard}),\\ second program \texttt{remaining\_cases.py}};
\node[big,minimum height=7.4cm] (store) at (9.1,-3.8) {\textbf{exact check}\\[2pt] $WW^{\top}=kI_n$\\ entries in $\{-1,0,1\}$\\[6pt] $\downarrow$\\[6pt] \textbf{store}\\ JSON file: matrix,\\ method, recipe\\[6pt] (the pair is done)};
\draw[arr] (m1.south) -- node[right,font=\scriptsize] {fails} (m2.north);
\draw[arr] (m2.south) -- node[right,font=\scriptsize] {fails} (m3.north);
\draw[arr] (m3.south) -- node[right,font=\scriptsize] {fails} (m4.north);
\draw[arr] (m4.south) -- node[right,font=\scriptsize] {fails} (m5.north);
\draw[arr,dashed] (m5.south) -- node[right,font=\scriptsize] {no GS matrix exists (status \texttt{none\_gs}; six pairs)} (m6.north);
\foreach \i in {1,2,3,4,5}{\draw[arr] (m\i.east) -- node[above,font=\scriptsize] {success} (m\i.east -| store.west);}
\draw[arr] (m6.east) -| (store.south);
\end{tikzpicture}

\caption{Processing of one pair $(n,k)$. Each method is tried only if all previous ones failed; as soon as one produces a matrix that passes the exact check, the matrix is stored and the processing of this pair stops. The six pairs for which no Goethals--Seidel matrix exists (Proposition~\ref{prop:parity}) are completed afterwards by a second program, \texttt{remaining\_cases.py}.}\label{fig:flow}
\end{figure}

Table~\ref{tab:methods} lists the methods, the name under which each appears in the JSON files, and the number of pairs for which it produced the stored matrix. The sections that follow describe them in detail. Every pair is covered exactly once.

\begin{table}[ht]
\centering\small
\begin{tabular}{@{}llp{6.2cm}r@{}}
\toprule
Stage & JSON \texttt{method} & Construction & Pairs \\
\midrule
1 & \texttt{identity} & $I_n$, for $k=1$ & 27\\
1 & \texttt{h2blocks} & $I_{n/2}\otimes H_2$, for $k=2$ & 27\\
1 & \texttt{dsum} & $\W(a,k)\oplus\W(n-a,k)$ & 672\\
1 & \texttt{kron} & $\W(a,k_1)\otimes\W(b,k_2)$, $ab=n$, $k_1k_2=k$ & 201\\
1 & \texttt{paley1} & Paley I, $\W(q+1,q)$, $q\equiv3\pmod4$ & 16\\
1 & \texttt{paley1\_h} & Paley I plus identity, $\W(q+1,q+1)$ & 8\\
1 & \texttt{paley2} & Paley II, $\W(2(q+1),2(q+1))$, $q\equiv1\pmod4$ & 4\\
2 & \texttt{gs} & Goethals--Seidel array: exhaustive search ($n\le24$): 24; annealing ($n\ge28$): 527 & 551\\
3 & \texttt{skew\_hadamard} & doubling of a skew-Hadamard matrix, $\W(n,n-1)$ & 6\\
\midrule
 & & Total & 1512\\
\bottomrule
\end{tabular}
\caption{Methods and number of pairs $(n,k)\in\mathcal S$ whose stored matrix they produced. In the Paley rows, $q$ is an odd prime power, defined in Method~4.}\label{tab:methods}
\end{table}

\subsection{Stage 1: classical constructions}\label{sec:stage1}

\subsubsection*{Method 1: trivial cases ($k=1$ and $k=2$)}

For $k=1$ the program stores the identity matrix $I_n$ (method \texttt{identity}): it has a single entry $+1$ in each row and column, and $I_nI_n\T=I_n=1\cdot I_n$. For $k=2$ it stores (method \texttt{h2blocks}) the matrix
\[
I_{n/2}\otimes H_2=\diag(H_2,H_2,\dots,H_2)\qquad(n/2\text{ copies}),
\]
the two expressions being equal by the definition of $\otimes$ (Section~\ref{sec:blocks}). It is a $\W(n,2)$: by the block rule, $\diag(H_2,\dots,H_2)\diag(H_2\T,\dots,H_2\T)=\diag(H_2H_2\T,\dots,H_2H_2\T)=\diag(2I_2,\dots,2I_2)=2I_n$, using $H_2H_2\T=2I_2$ (Section~\ref{sec:intro}). (This is not the general Kronecker method below: \texttt{h2blocks} is tried first and only for $k=2$.)

\emph{Example.} Figure~\ref{fig:basic} shows the stored matrices $\W(8,1)=I_8$ and $\W(8,2)=I_4\otimes H_2$.

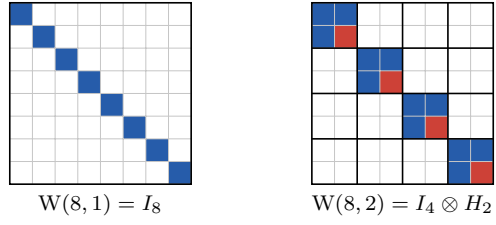
\begin{figure}[ht]
\centering
\begin{tikzpicture}[x=0.300cm,y=0.300cm,baseline=(mid)]
\coordinate (mid) at (4.0,-4.0);
\foreach \c/\r in {0/0,1/1,2/2,3/3,4/4,5/5,6/6,7/7}{\fill[wplus] (\c,-\r) rectangle ++(1,-1);}
\draw[xstep=1,ystep=1,gray!45,line width=0.12pt] (0,-8) grid (8,0);
\draw[line width=0.6pt] (0,0) rectangle (8,-8);
\node[font=\scriptsize,below,align=center] at (4.0,-8) {$\W(8,1)=I_8$};
\end{tikzpicture}
\hspace{1.20cm}
\begin{tikzpicture}[x=0.300cm,y=0.300cm,baseline=(mid)]
\coordinate (mid) at (4.0,-4.0);
\foreach \c/\r in {0/0,1/0,0/1,2/2,3/2,2/3,4/4,5/4,4/5,6/6,7/6,6/7}{\fill[wplus] (\c,-\r) rectangle ++(1,-1);}
\foreach \c/\r in {1/1,3/3,5/5,7/7}{\fill[wminus] (\c,-\r) rectangle ++(1,-1);}
\draw[xstep=1,ystep=1,gray!45,line width=0.12pt] (0,-8) grid (8,0);
\draw[line width=0.6pt] (2,0) -- (2,-8);
\draw[line width=0.6pt] (0,-2) -- (8,-2);
\draw[line width=0.6pt] (4,0) -- (4,-8);
\draw[line width=0.6pt] (0,-4) -- (8,-4);
\draw[line width=0.6pt] (6,0) -- (6,-8);
\draw[line width=0.6pt] (0,-6) -- (8,-6);
\draw[line width=0.6pt] (0,0) rectangle (8,-8);
\node[font=\scriptsize,below,align=center] at (4.0,-8) {$\W(8,2)=I_4\otimes H_2$};
\end{tikzpicture}
\caption{Method 1 at order $8$: the stored matrices $\W(8,1)$ and $\W(8,2)$.}\label{fig:basic}
\end{figure}

\subsubsection*{Method 2: direct sum}

\begin{lemma}\label{lem:dsum}
If $W_1$ is a $\W(a,k)$ and $W_2$ is a $\W(b,k)$ \emph{with the same weight $k$}, then $\diag(W_1,W_2)$ is a $\W(a+b,k)$.
\end{lemma}
\begin{proof}
Its entries lie in $\{-1,0,1\}$. By the rules for block matrices,
\[
\begin{pmatrix}W_1&0\\0&W_2\end{pmatrix}\begin{pmatrix}W_1\T&0\\0&W_2\T\end{pmatrix}
=\begin{pmatrix}W_1W_1\T+0\cdot0&W_1\cdot0+0\cdot W_2\T\\0\cdot W_1\T+W_2\cdot0&0\cdot0+W_2W_2\T\end{pmatrix}
=\begin{pmatrix}kI_a&0\\0&kI_b\end{pmatrix}=kI_{a+b}.\qedhere
\]
\end{proof}

\emph{What the program does.} For a pair $(n,k)$ it considers splittings $n=a+b$ in which $a$ and $b$ are multiples of $4$ (only orders divisible by $4$ are stored), starting with the most balanced one,
\[
a=4\lfloor n/8\rfloor,\quad\text{then }a-4,\ a-8,\ \dots,\ 4,
\]
and with $b=n-a$ (for $n=4$ this list is empty, since $a=0$). The first $a$ for which $k\le a$ and both $\W(a,k)$ and $\W(b,k)$ are available gives the recipe $\W(a,k)\oplus\W(b,k)$ (method \texttt{dsum}). Thus always $a\le n/2\le b$: the program never puts the larger block first (in all $672$ stored \texttt{dsum} matrices $a\le b$; $a=b$ occurs in $336$ of them, and $a<b$ in the other $336$; $a=b$ is possible only when $n\equiv0\pmod8$, because $4\lfloor n/8\rfloor=n/2$ exactly then). The condition $k\le a$ is needed because $\W(a,k)$ cannot exist for $k>a$. This is the most frequently used method ($672$ of the $1512$ pairs).

\emph{Examples.} (a) $n=8$, $k=3$: $a=4\lfloor8/8\rfloor=4=b$, and $\W(8,3)=\W(4,3)\oplus\W(4,3)$ (Figure~\ref{fig:dsum1}, top). (b) $n=12$, $k=3$: $a=4\lfloor12/8\rfloor=4$, $b=8$, and $\W(12,3)=\W(4,3)\oplus\W(8,3)$ with \emph{different} blocks (Figure~\ref{fig:dsum1}, bottom). (c) $n=20$, $k=7$: $a=4\lfloor20/8\rfloor=8$, $b=12$, and $\W(20,7)=\W(8,7)\oplus\W(12,7)$ (Figure~\ref{fig:dsum2}); here the two blocks even come from different methods: $\W(8,7)$ from Paley's construction (Method 4) and $\W(12,7)$ from the Goethals--Seidel array of Section~\ref{sec:ex127}.

\begin{figure}[ht]
\centering
\begin{tikzpicture}[x=0.400cm,y=0.400cm,baseline=(mid)]
\coordinate (mid) at (2.0,-2.0);
\foreach \c/\r in {1/0,2/0,3/0,2/1,3/2,1/3}{\fill[wplus] (\c,-\r) rectangle ++(1,-1);}
\foreach \c/\r in {0/1,3/1,0/2,1/2,0/3,2/3}{\fill[wminus] (\c,-\r) rectangle ++(1,-1);}
\draw[xstep=1,ystep=1,gray!45,line width=0.12pt] (0,-4) grid (4,0);
\draw[line width=0.6pt] (0,0) rectangle (4,-4);
\node[font=\scriptsize,below,align=center] at (2.0,-4) {$\W(4,3)$};
\end{tikzpicture}
\hspace{0.15cm}{\large $\oplus$}\hspace{0.15cm}
\begin{tikzpicture}[x=0.400cm,y=0.400cm,baseline=(mid)]
\coordinate (mid) at (2.0,-2.0);
\foreach \c/\r in {1/0,2/0,3/0,2/1,3/2,1/3}{\fill[wplus] (\c,-\r) rectangle ++(1,-1);}
\foreach \c/\r in {0/1,3/1,0/2,1/2,0/3,2/3}{\fill[wminus] (\c,-\r) rectangle ++(1,-1);}
\draw[xstep=1,ystep=1,gray!45,line width=0.12pt] (0,-4) grid (4,0);
\draw[line width=0.6pt] (0,0) rectangle (4,-4);
\node[font=\scriptsize,below,align=center] at (2.0,-4) {$\W(4,3)$};
\end{tikzpicture}
\hspace{0.15cm}{\large $=$}\hspace{0.15cm}
\begin{tikzpicture}[x=0.340cm,y=0.340cm,baseline=(mid)]
\coordinate (mid) at (4.0,-4.0);
\foreach \c/\r in {1/0,2/0,3/0,2/1,3/2,1/3,5/4,6/4,7/4,6/5,7/6,5/7}{\fill[wplus] (\c,-\r) rectangle ++(1,-1);}
\foreach \c/\r in {0/1,3/1,0/2,1/2,0/3,2/3,4/5,7/5,4/6,5/6,4/7,6/7}{\fill[wminus] (\c,-\r) rectangle ++(1,-1);}
\draw[xstep=1,ystep=1,gray!45,line width=0.12pt] (0,-8) grid (8,0);
\draw[line width=0.6pt] (4,0) -- (4,-8);
\draw[line width=0.6pt] (0,-4) -- (8,-4);
\draw[line width=0.6pt] (0,0) rectangle (8,-8);
\node[font=\scriptsize,below,align=center] at (4.0,-8) {$\W(8,3)$ ($a=b=4$)};
\end{tikzpicture}

\medskip
\begin{tikzpicture}[x=0.400cm,y=0.400cm,baseline=(mid)]
\coordinate (mid) at (2.0,-2.0);
\foreach \c/\r in {1/0,2/0,3/0,2/1,3/2,1/3}{\fill[wplus] (\c,-\r) rectangle ++(1,-1);}
\foreach \c/\r in {0/1,3/1,0/2,1/2,0/3,2/3}{\fill[wminus] (\c,-\r) rectangle ++(1,-1);}
\draw[xstep=1,ystep=1,gray!45,line width=0.12pt] (0,-4) grid (4,0);
\draw[line width=0.6pt] (0,0) rectangle (4,-4);
\node[font=\scriptsize,below,align=center] at (2.0,-4) {$\W(4,3)$};
\end{tikzpicture}
\hspace{0.15cm}{\large $\oplus$}\hspace{0.15cm}
\begin{tikzpicture}[x=0.340cm,y=0.340cm,baseline=(mid)]
\coordinate (mid) at (4.0,-4.0);
\foreach \c/\r in {1/0,2/0,3/0,2/1,3/2,1/3,5/4,6/4,7/4,6/5,7/6,5/7}{\fill[wplus] (\c,-\r) rectangle ++(1,-1);}
\foreach \c/\r in {0/1,3/1,0/2,1/2,0/3,2/3,4/5,7/5,4/6,5/6,4/7,6/7}{\fill[wminus] (\c,-\r) rectangle ++(1,-1);}
\draw[xstep=1,ystep=1,gray!45,line width=0.12pt] (0,-8) grid (8,0);
\draw[line width=0.6pt] (0,0) rectangle (8,-8);
\node[font=\scriptsize,below,align=center] at (4.0,-8) {$\W(8,3)$};
\end{tikzpicture}
\hspace{0.15cm}{\large $=$}\hspace{0.15cm}
\begin{tikzpicture}[x=0.300cm,y=0.300cm,baseline=(mid)]
\coordinate (mid) at (6.0,-6.0);
\foreach \c/\r in {1/0,2/0,3/0,2/1,3/2,1/3,5/4,6/4,7/4,6/5,7/6,5/7,9/8,10/8,11/8,10/9,11/10,9/11}{\fill[wplus] (\c,-\r) rectangle ++(1,-1);}
\foreach \c/\r in {0/1,3/1,0/2,1/2,0/3,2/3,4/5,7/5,4/6,5/6,4/7,6/7,8/9,11/9,8/10,9/10,8/11,10/11}{\fill[wminus] (\c,-\r) rectangle ++(1,-1);}
\draw[xstep=1,ystep=1,gray!45,line width=0.12pt] (0,-12) grid (12,0);
\draw[line width=0.6pt] (4,0) -- (4,-12);
\draw[line width=0.6pt] (0,-4) -- (12,-4);
\draw[line width=0.6pt] (0,0) rectangle (12,-12);
\node[font=\scriptsize,below,align=center] at (6.0,-12) {$\W(12,3)$ ($a=4$, $b=8$)};
\end{tikzpicture}
\caption{Direct sums of stored matrices with $a=b$ (top) and $a<b$ (bottom).}\label{fig:dsum1}
\end{figure}
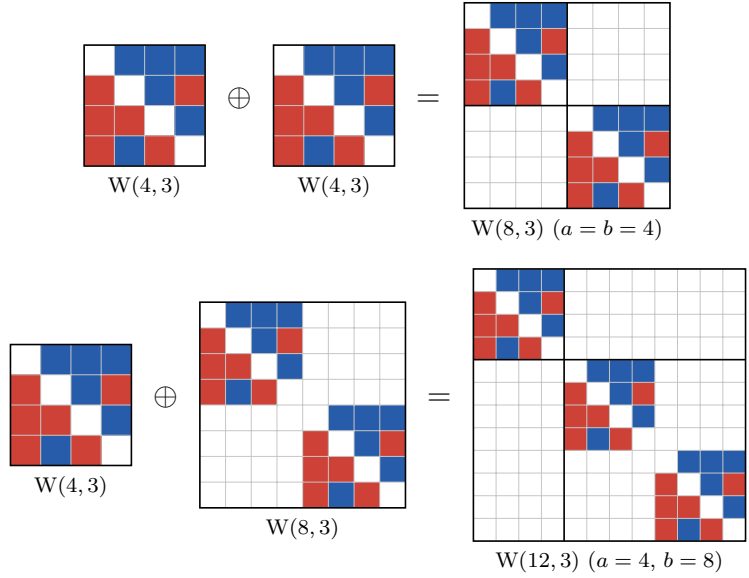
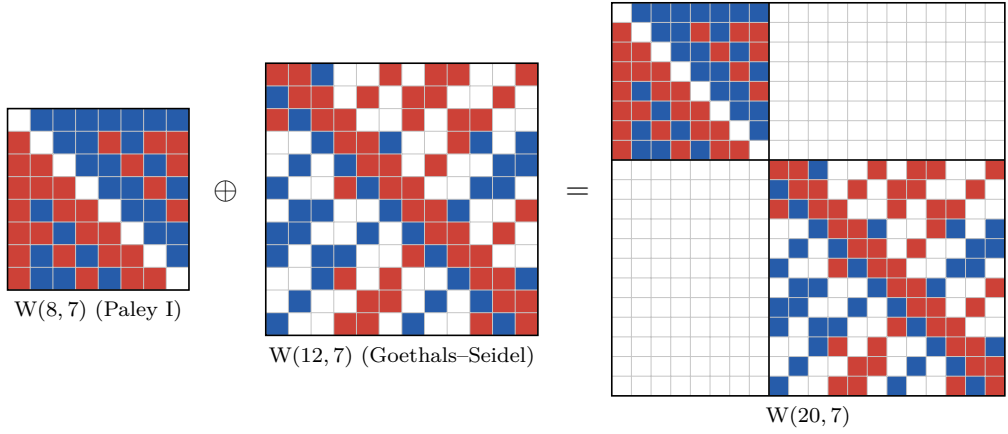
\begin{figure}[ht]
\centering
\begin{tikzpicture}[x=0.300cm,y=0.300cm,baseline=(mid)]
\coordinate (mid) at (4.0,-4.0);
\foreach \c/\r in {1/0,2/0,3/0,4/0,5/0,6/0,7/0,2/1,3/1,5/1,3/2,4/2,6/2,4/3,5/3,7/3,1/4,5/4,6/4,2/5,6/5,7/5,1/6,3/6,7/6,1/7,2/7,4/7}{\fill[wplus] (\c,-\r) rectangle ++(1,-1);}
\foreach \c/\r in {0/1,4/1,6/1,7/1,0/2,1/2,5/2,7/2,0/3,1/3,2/3,6/3,0/4,2/4,3/4,7/4,0/5,1/5,3/5,4/5,0/6,2/6,4/6,5/6,0/7,3/7,5/7,6/7}{\fill[wminus] (\c,-\r) rectangle ++(1,-1);}
\draw[xstep=1,ystep=1,gray!45,line width=0.12pt] (0,-8) grid (8,0);
\draw[line width=0.6pt] (0,0) rectangle (8,-8);
\node[font=\scriptsize,below,align=center] at (4.0,-8) {$\W(8,7)$ (Paley I)};
\end{tikzpicture}
\hspace{0.15cm}{\large $\oplus$}\hspace{0.15cm}
\begin{tikzpicture}[x=0.300cm,y=0.300cm,baseline=(mid)]
\coordinate (mid) at (6.0,-6.0);
\foreach \c/\r in {2/0,0/1,1/2,2/3,5/3,9/3,11/3,1/4,3/4,10/4,11/4,0/5,4/5,9/5,10/5,1/6,2/6,5/6,8/6,0/7,1/7,4/7,6/7,0/8,2/8,3/8,7/8,2/9,8/9,11/9,1/10,7/10,9/10,0/11,6/11,10/11}{\fill[wplus] (\c,-\r) rectangle ++(1,-1);}
\foreach \c/\r in {0/0,1/0,5/0,7/0,8/0,11/0,1/1,2/1,4/1,6/1,7/1,10/1,0/2,2/2,3/2,6/2,8/2,9/2,3/3,4/3,8/3,4/4,5/4,7/4,3/5,5/5,6/5,6/6,7/6,11/6,7/7,8/7,10/7,6/8,8/8,9/8,3/9,5/9,9/9,10/9,4/10,5/10,10/10,11/10,3/11,4/11,9/11,11/11}{\fill[wminus] (\c,-\r) rectangle ++(1,-1);}
\draw[xstep=1,ystep=1,gray!45,line width=0.12pt] (0,-12) grid (12,0);
\draw[line width=0.6pt] (0,0) rectangle (12,-12);
\node[font=\scriptsize,below,align=center] at (6.0,-12) {$\W(12,7)$ (Goethals--Seidel)};
\end{tikzpicture}
\hspace{0.15cm}{\large $=$}\hspace{0.15cm}
\begin{tikzpicture}[x=0.260cm,y=0.260cm,baseline=(mid)]
\coordinate (mid) at (10.0,-10.0);
\foreach \c/\r in {1/0,2/0,3/0,4/0,5/0,6/0,7/0,2/1,3/1,5/1,3/2,4/2,6/2,4/3,5/3,7/3,1/4,5/4,6/4,2/5,6/5,7/5,1/6,3/6,7/6,1/7,2/7,4/7,10/8,8/9,9/10,10/11,13/11,17/11,19/11,9/12,11/12,18/12,19/12,8/13,12/13,17/13,18/13,9/14,10/14,13/14,16/14,8/15,9/15,12/15,14/15,8/16,10/16,11/16,15/16,10/17,16/17,19/17,9/18,15/18,17/18,8/19,14/19,18/19}{\fill[wplus] (\c,-\r) rectangle ++(1,-1);}
\foreach \c/\r in {0/1,4/1,6/1,7/1,0/2,1/2,5/2,7/2,0/3,1/3,2/3,6/3,0/4,2/4,3/4,7/4,0/5,1/5,3/5,4/5,0/6,2/6,4/6,5/6,0/7,3/7,5/7,6/7,8/8,9/8,13/8,15/8,16/8,19/8,9/9,10/9,12/9,14/9,15/9,18/9,8/10,10/10,11/10,14/10,16/10,17/10,11/11,12/11,16/11,12/12,13/12,15/12,11/13,13/13,14/13,14/14,15/14,19/14,15/15,16/15,18/15,14/16,16/16,17/16,11/17,13/17,17/17,18/17,12/18,13/18,18/18,19/18,11/19,12/19,17/19,19/19}{\fill[wminus] (\c,-\r) rectangle ++(1,-1);}
\draw[xstep=1,ystep=1,gray!45,line width=0.12pt] (0,-20) grid (20,0);
\draw[line width=0.6pt] (8,0) -- (8,-20);
\draw[line width=0.6pt] (0,-8) -- (20,-8);
\draw[line width=0.6pt] (0,0) rectangle (20,-20);
\node[font=\scriptsize,below,align=center] at (10.0,-20) {$\W(20,7)$};
\end{tikzpicture}
\caption{$\W(20,7)=\W(8,7)\oplus\W(12,7)$: a direct sum with $a=8<b=12$, as stored in the catalogue.}\label{fig:dsum2}
\end{figure}

\subsubsection*{Method 3: Kronecker product}

\begin{lemma}\label{lem:kronw}
If $W_1$ is a $\W(a,k_1)$ and $W_2$ is a $\W(b,k_2)$ then $W_1\otimes W_2$ is a $\W(ab,k_1k_2)$.
\end{lemma}
\begin{proof}
The entries of $W_1\otimes W_2$ are products of an entry of $W_1$ and an entry of $W_2$, so they lie in $\{-1,0,1\}$. By Lemma~\ref{lem:kron},
\[
\begin{aligned}
(W_1\otimes W_2)(W_1\otimes W_2)\T&=(W_1\otimes W_2)(W_1\T\otimes W_2\T)=(W_1W_1\T)\otimes(W_2W_2\T)\\
&=(k_1I_a)\otimes(k_2I_b)=k_1k_2\,(I_a\otimes I_b)=k_1k_2I_{ab}.\qedhere\end{aligned}
\]
\end{proof}

\emph{What the program does.} It runs through the divisors $a$ of $n$ with $2\le a\le n/a$ in increasing order, writes $b=n/a$ (so $a\le b$ always), and for each $a$ runs through the divisors $k_1$ of $k$ in increasing order, with $k_2=k/k_1$. The first choice with $k_1\le a$, $k_2\le b$ and both $\W(a,k_1)$ and $\W(b,k_2)$ available (stored or atoms) gives the recipe $\W(a,k_1)\otimes\W(b,k_2)$ (method \texttt{kron}; $201$ pairs).

\emph{Examples.} (a) $n=8$, $k=6$: $a=2$, $b=4$; $k_1=1$ fails ($k_2=6>b$), $k_1=2$ gives $k_2=3$ and the atom $\W(2,2)=H_2$ together with the stored $\W(4,3)$ succeed: $\W(8,6)=H_2\otimes\W(4,3)=\begin{pmatrix}V&V\\V&-V\end{pmatrix}$ with $V=\W(4,3)$, i.e.\ every entry of $H_2$ is replaced by $\pm V$ (Figure~\ref{fig:kron}, top). (b) $n=16$, $k=9$: the direct sum is impossible ($k>a$ for $a=8,4$); for $a=2$, $b=8$ no divisor $k_1$ of $9$ works ($k_1=1$ gives $k_2=9>8$, and $k_1=3,9$ exceed $a=2$); for $a=4=b$, $k_1=1$ fails and $k_1=3$, $k_2=3$ succeeds: $\W(16,9)=\W(4,3)\otimes\W(4,3)$. Here the zero entries of the left factor give zero blocks (Figure~\ref{fig:kron}, bottom).

\begin{figure}[ht]
\centering
\begin{tikzpicture}[x=0.450cm,y=0.450cm,baseline=(mid)]
\coordinate (mid) at (1.0,-1.0);
\foreach \c/\r in {0/0,1/0,0/1}{\fill[wplus] (\c,-\r) rectangle ++(1,-1);}
\foreach \c/\r in {1/1}{\fill[wminus] (\c,-\r) rectangle ++(1,-1);}
\draw[xstep=1,ystep=1,gray!45,line width=0.12pt] (0,-2) grid (2,0);
\draw[line width=0.6pt] (0,0) rectangle (2,-2);
\node[font=\scriptsize,below,align=center] at (1.0,-2) {$H_2$};
\end{tikzpicture}
\hspace{0.15cm}{\large $\otimes$}\hspace{0.15cm}
\begin{tikzpicture}[x=0.400cm,y=0.400cm,baseline=(mid)]
\coordinate (mid) at (2.0,-2.0);
\foreach \c/\r in {1/0,2/0,3/0,2/1,3/2,1/3}{\fill[wplus] (\c,-\r) rectangle ++(1,-1);}
\foreach \c/\r in {0/1,3/1,0/2,1/2,0/3,2/3}{\fill[wminus] (\c,-\r) rectangle ++(1,-1);}
\draw[xstep=1,ystep=1,gray!45,line width=0.12pt] (0,-4) grid (4,0);
\draw[line width=0.6pt] (0,0) rectangle (4,-4);
\node[font=\scriptsize,below,align=center] at (2.0,-4) {$V=\W(4,3)$};
\end{tikzpicture}
\hspace{0.15cm}{\large $=$}\hspace{0.15cm}
\begin{tikzpicture}[x=0.320cm,y=0.320cm,baseline=(mid)]
\coordinate (mid) at (4.0,-4.0);
\foreach \c/\r in {1/0,2/0,3/0,5/0,6/0,7/0,2/1,6/1,3/2,7/2,1/3,5/3,1/4,2/4,3/4,2/5,4/5,7/5,3/6,4/6,5/6,1/7,4/7,6/7}{\fill[wplus] (\c,-\r) rectangle ++(1,-1);}
\foreach \c/\r in {0/1,3/1,4/1,7/1,0/2,1/2,4/2,5/2,0/3,2/3,4/3,6/3,5/4,6/4,7/4,0/5,3/5,6/5,0/6,1/6,7/6,0/7,2/7,5/7}{\fill[wminus] (\c,-\r) rectangle ++(1,-1);}
\draw[xstep=1,ystep=1,gray!45,line width=0.12pt] (0,-8) grid (8,0);
\draw[line width=0.6pt] (4,0) -- (4,-8);
\draw[line width=0.6pt] (0,-4) -- (8,-4);
\draw[line width=0.6pt] (0,0) rectangle (8,-8);
\node[font=\scriptsize,below,align=center] at (4.0,-8) {$\W(8,6)$};
\end{tikzpicture}

\medskip
\begin{tikzpicture}[x=0.400cm,y=0.400cm,baseline=(mid)]
\coordinate (mid) at (2.0,-2.0);
\foreach \c/\r in {1/0,2/0,3/0,2/1,3/2,1/3}{\fill[wplus] (\c,-\r) rectangle ++(1,-1);}
\foreach \c/\r in {0/1,3/1,0/2,1/2,0/3,2/3}{\fill[wminus] (\c,-\r) rectangle ++(1,-1);}
\draw[xstep=1,ystep=1,gray!45,line width=0.12pt] (0,-4) grid (4,0);
\draw[line width=0.6pt] (0,0) rectangle (4,-4);
\node[font=\scriptsize,below,align=center] at (2.0,-4) {$V$};
\end{tikzpicture}
\hspace{0.15cm}{\large $\otimes$}\hspace{0.15cm}
\begin{tikzpicture}[x=0.400cm,y=0.400cm,baseline=(mid)]
\coordinate (mid) at (2.0,-2.0);
\foreach \c/\r in {1/0,2/0,3/0,2/1,3/2,1/3}{\fill[wplus] (\c,-\r) rectangle ++(1,-1);}
\foreach \c/\r in {0/1,3/1,0/2,1/2,0/3,2/3}{\fill[wminus] (\c,-\r) rectangle ++(1,-1);}
\draw[xstep=1,ystep=1,gray!45,line width=0.12pt] (0,-4) grid (4,0);
\draw[line width=0.6pt] (0,0) rectangle (4,-4);
\node[font=\scriptsize,below,align=center] at (2.0,-4) {$V$};
\end{tikzpicture}
\hspace{0.15cm}{\large $=$}\hspace{0.15cm}
\begin{tikzpicture}[x=0.240cm,y=0.240cm,baseline=(mid)]
\coordinate (mid) at (8.0,-8.0);
\foreach \c/\r in {5/0,6/0,7/0,9/0,10/0,11/0,13/0,14/0,15/0,6/1,10/1,14/1,7/2,11/2,15/2,5/3,9/3,13/3,9/4,10/4,11/4,0/5,3/5,10/5,12/5,15/5,0/6,1/6,11/6,12/6,13/6,0/7,2/7,9/7,12/7,14/7,13/8,14/8,15/8,0/9,3/9,4/9,7/9,14/9,0/10,1/10,4/10,5/10,15/10,0/11,2/11,4/11,6/11,13/11,5/12,6/12,7/12,0/13,3/13,6/13,8/13,11/13,0/14,1/14,7/14,8/14,9/14,0/15,2/15,5/15,8/15,10/15}{\fill[wplus] (\c,-\r) rectangle ++(1,-1);}
\foreach \c/\r in {4/1,7/1,8/1,11/1,12/1,15/1,4/2,5/2,8/2,9/2,12/2,13/2,4/3,6/3,8/3,10/3,12/3,14/3,1/4,2/4,3/4,13/4,14/4,15/4,2/5,8/5,11/5,14/5,3/6,8/6,9/6,15/6,1/7,8/7,10/7,13/7,1/8,2/8,3/8,5/8,6/8,7/8,2/9,6/9,12/9,15/9,3/10,7/10,12/10,13/10,1/11,5/11,12/11,14/11,1/12,2/12,3/12,9/12,10/12,11/12,2/13,4/13,7/13,10/13,3/14,4/14,5/14,11/14,1/15,4/15,6/15,9/15}{\fill[wminus] (\c,-\r) rectangle ++(1,-1);}
\draw[xstep=1,ystep=1,gray!45,line width=0.12pt] (0,-16) grid (16,0);
\draw[line width=0.6pt] (4,0) -- (4,-16);
\draw[line width=0.6pt] (0,-4) -- (16,-4);
\draw[line width=0.6pt] (8,0) -- (8,-16);
\draw[line width=0.6pt] (0,-8) -- (16,-8);
\draw[line width=0.6pt] (12,0) -- (12,-16);
\draw[line width=0.6pt] (0,-12) -- (16,-12);
\draw[line width=0.6pt] (0,0) rectangle (16,-16);
\node[font=\scriptsize,below,align=center] at (8.0,-16) {$\W(16,9)$};
\end{tikzpicture}
\caption{Kronecker products as stored in the catalogue: $\W(8,6)=H_2\otimes V$ and $\W(16,9)=V\otimes V$, with $V=\W(4,3)$.}\label{fig:kron}
\end{figure}
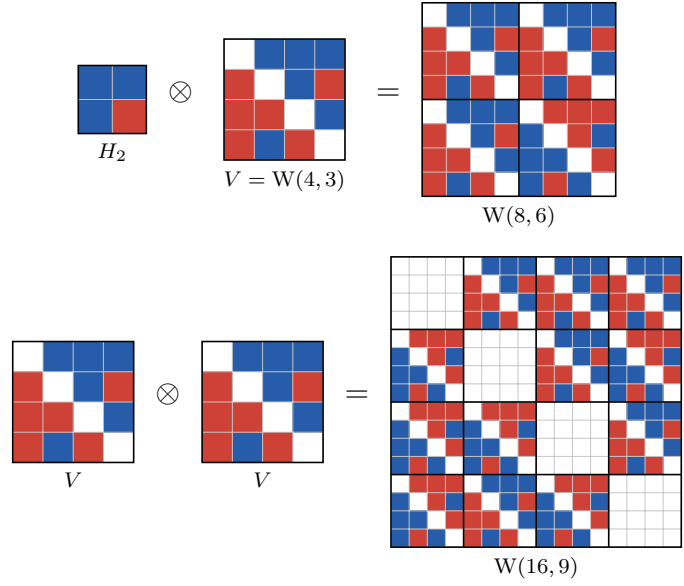
\subsubsection*{Method 4: Paley constructions}\label{sec:paley}

The Paley constructions~\cite{paley} use the arithmetic of finite fields. We first collect the (standard) definitions and four facts.

\paragraph{Finite fields and the quadratic character.}
Let $q$ be a \emph{prime power}, that is, $q=p^e$ with $p$ a prime number and $e\ge1$ (for example $q=3,7,11,25,27,49$; the number $15=3\cdot5$ is not a prime power). A \emph{field} is a set with an addition and a multiplication in which one can add, subtract, multiply and divide by nonzero elements with the usual rules (the rational or the real numbers are fields). When $q$ is a prime power there is, up to renaming its elements, exactly one field with $q$ elements; it is written $\mathrm{GF}(q)$ (from \emph{Galois field}). This is a standard result of algebra that we use without proof. For prime $q$, $\mathrm{GF}(q)$ is simply $\{0,1,\dots,q-1\}$ with addition and multiplication modulo $q$ (division works because $q$ is prime: for example, for $q=7$, $3\cdot5=15\equiv1$, so $5$ is the inverse of $3$). For $q=p^e$ with $e\ge2$ (needed here only for $q=25,27,49$) the elements are the polynomials $c_0+c_1x+\dots+c_{e-1}x^{e-1}$ of degree $<e$ with coefficients $c_t\in\{0,\dots,p-1\}$ (modulo $p$); they are added coefficient by coefficient and multiplied as polynomials, the result then being replaced by its remainder on division by a fixed \emph{irreducible} polynomial $f$ of degree $e$ (a polynomial that is not a product of two polynomials of smaller degree), in the same way that for integers one takes the remainder modulo a prime. For example, for $q=25=5^2$ one may take $f=x^2-2$ (irreducible because $2$ is not a square modulo $5$): the elements are $a+bx$ with $a,b\in\{0,\dots,4\}$ and $x^2=2$, and $(1+x)(1+2x)=1+3x+2x^2=1+3x+4=3x$. (This only illustrates the arithmetic: here $x^2=2$, $x^4=4$ and $x^8=1$, so the class of $x$ has order $8$ and not $24$, and $x^2-2$ is \emph{not} a polynomial for which $x$ is a generator. The program searches by itself for a polynomial $f$ for which $x$ is a generator, in the sense below. Another choice gives an equivalent matrix, obtained by permuting rows and columns.) Here $q$ is odd, so $q-1$ is even. The nonzero elements of $\mathrm{GF}(q)$ form a cyclic group under multiplication (another standard fact that we do not prove), i.e.\ there is a \emph{generator} $g$ such that every nonzero element is $g^t$ for exactly one $t\in\{0,\dots,q-2\}$. An element is a \emph{square} if it equals $y^2$ for some $y$; the nonzero squares are exactly the even powers of $g$.

\begin{definition}
The \emph{quadratic character} $\chi$ of $\mathrm{GF}(q)$ is the function $\chi(0)=0$, $\chi(x)=1$ if $x$ is a nonzero square, and $\chi(x)=-1$ otherwise; in terms of the generator, $\chi(g^t)=(-1)^t$.
\end{definition}
For $q=7$ the nonzero squares are $1^2=1$, $2^2=4$, $3^2=2$ (and $4^2=2$, $5^2=4$, $6^2=1$ repeat them), so $\chi(1)=\chi(2)=\chi(4)=1$ and $\chi(3)=\chi(5)=\chi(6)=-1$.

\begin{lemma}\label{lem:chi}
(a) $\chi(xy)=\chi(x)\chi(y)$. (b) $\sum_{x}\chi(x)=0$ (the sum runs over all $q$ elements). (c) $\chi(-1)=1$ if $q\equiv1\pmod4$ and $\chi(-1)=-1$ if $q\equiv3\pmod4$. (d) For every $d\ne0$, $\sum_{x}\chi(x)\chi(x+d)=-1$.
\end{lemma}
\begin{proof}
(a) If $x=g^t$ and $y=g^u$ nonzero then $xy=g^{t+u}$ and $(-1)^{t+u}=(-1)^t(-1)^u$ (exponents are taken modulo $q-1$, which is even, so the parity is well defined); if $x$ or $y$ is $0$ both sides are $0$. (b) Among the $q-1$ nonzero elements, exactly $(q-1)/2$ are even powers of $g$ and $(q-1)/2$ odd powers. (c) $-1$ is the only element $\ne1$ whose square is $1$, and $g^{(q-1)/2}$ has this property; so $-1=g^{(q-1)/2}$ and $\chi(-1)=(-1)^{(q-1)/2}$, which is $1$ iff $q\equiv1\pmod4$. (d) The term $x=0$ is $0$, and the term $x=-d$ is $0$. For $x\ne0$ write $x+d=x(1+dx^{-1})$; by (a), $\chi(x)\chi(x+d)=\chi(x)^2\chi(1+dx^{-1})=\chi(1+dx^{-1})$ since $\chi(x)^2=1$. As $x$ runs through the nonzero elements, $dx^{-1}$ runs through all nonzero elements, so $1+dx^{-1}$ runs through all elements except $1$. Hence the sum equals $\sum_{y\ne1}\chi(y)=0-\chi(1)=-1$ by (b).
\end{proof}

\paragraph{The matrix $Q$.}
Number the rows and columns of a $q\times q$ matrix by the elements of $\mathrm{GF}(q)$ (for prime $q$: $0,1,\dots,q-1$, in this order) and define $Q_{xy}=\chi(y-x)$. For prime $q$ the entry depends on $y-x$ modulo $q$ only, so $Q$ is the circulant of the sequence $(\chi(0),\chi(1),\dots,\chi(q-1))$; for $q=7$ this is $(0,1,1,-1,1,-1,-1)$.

\begin{lemma}\label{lem:Q}
(i) $Q\one=0$ and $\one\T Q=0$ (all row sums and column sums vanish). (ii) $QQ\T=qI_q-J$. (iii) $Q\T=\chi(-1)\,Q$: $Q$ is skew-symmetric ($Q\T=-Q$) if $q\equiv3\pmod4$ and symmetric if $q\equiv1\pmod4$.
\end{lemma}
\begin{proof}
(i) The sum of row $x$ is $\sum_y\chi(y-x)=\sum_z\chi(z)=0$ by Lemma~\ref{lem:chi}(b), as $y-x$ runs through all elements when $y$ does; the same for columns. (ii) $(QQ\T)_{xx'}=\sum_yQ_{xy}Q_{x'y}=\sum_y\chi(y-x)\chi(y-x')$. If $x=x'$ the sum counts the $q-1$ values of $y\ne x$ (where $\chi(y-x)^2=1$), giving $q-1$. If $x\ne x'$ put $z=y-x$ and $d=x-x'\ne0$: the sum is $\sum_z\chi(z)\chi(z+d)=-1$ by Lemma~\ref{lem:chi}(d). So $QQ\T$ has $q-1$ on the diagonal and $-1$ elsewhere, which is $qI_q-J$. (iii) $(Q\T)_{xy}=Q_{yx}=\chi(x-y)=\chi(-1)\chi(y-x)$ by Lemma~\ref{lem:chi}(a).
\end{proof}

\paragraph{Paley I.}
The program applies Paley I when $q=n-1$ is a prime power; since $n$ is divisible by $4$, then automatically $q\equiv3\pmod4$. Border $Q$ with a row of $+1$ and a column of $-1$:
\[
S=\begin{pmatrix}0&\one\T\\-\one&Q\end{pmatrix}\qquad((q+1)\times(q+1)).
\]
\begin{proposition}\label{prop:paley1}
For $q\equiv3\pmod4$: $S$ is a $\W(q+1,q)$, $S\T=-S$, and $I_{q+1}+S$ is a Hadamard matrix, i.e.\ a $\W(q+1,q+1)$.
\end{proposition}
\begin{proof}
The entries are in $\{-1,0,1\}$. With $S\T=\begin{pmatrix}0&-\one\T\\\one&Q\T\end{pmatrix}$ the block rule gives
\[
\begin{aligned}
SS\T&=\begin{pmatrix}0\cdot0+\one\T\one&0\cdot(-\one\T)+\one\T Q\T\\(-\one)\cdot0+Q\one&(-\one)(-\one\T)+QQ\T\end{pmatrix}\\
&=\begin{pmatrix}q&(Q\one)\T\\Q\one&J+qI_q-J\end{pmatrix}=\begin{pmatrix}q&0\\0&qI_q\end{pmatrix}=qI_{q+1},\end{aligned}
\]
using $\one\T\one=q$, $\one\one\T=J$, $Q\one=0$ (Lemma~\ref{lem:Q}(i)) and $QQ\T=qI_q-J$ (ii). So $S$ is a $\W(q+1,q)$. Since $q\equiv3\pmod4$, $Q\T=-Q$ (Lemma~\ref{lem:Q}(iii)), and the borders are negatives of each other, so $S\T=-S$. Then
\[
(I+S)(I+S)\T=I+S+S\T+SS\T=I+S-S+qI=(q+1)I_{q+1}.
\]
$I+S$ has entries $\pm1$ (its diagonal is $1$, the off-diagonal entries are those of $S$ and are $\pm1$ because $S$ has $q$ nonzero entries per row, all off the diagonal), so it is a Hadamard matrix of order $q+1$.
\end{proof}
The program stores $S$ for $(n,k)=(q+1,q)$ (method \texttt{paley1}, $16$ pairs) and $I+S$ for $(n,k)=(q+1,q+1)$ (method \texttt{paley1\_h}, $8$ pairs: $n=12,20,28,44,60,68,84,108$, i.e.\ $q=11,19,27,43,59,67,83,107$). Note that the pair $(n,n)$ is only reached by this method if the earlier methods failed for it.

\emph{Example.} For $q=3$ the only nonzero square is $1$, so $\chi(0),\chi(1),\chi(2)=0,1,-1$, 
\[
Q=\begin{pmatrix}0&1&-1\\-1&0&1\\1&-1&0\end{pmatrix},\qquad S=\begin{pmatrix}0&1&1&1\\-1&0&1&-1\\-1&-1&0&1\\-1&1&-1&0\end{pmatrix}=\W(4,3).
\]
(Check: rows $1$ and $2$ of $S$ have dot product $(-1)(-1)+0\cdot(-1)+1\cdot0+(-1)\cdot1=0$; each row has three nonzero entries.) This $S$ is the stored $\W(4,3)$ used in Figures~\ref{fig:dsum1}--\ref{fig:kron}. For $q=7$ (character values above) one obtains the stored $\W(8,7)$, and for $q=11$ the stored $\W(12,11)=S$ and $\W(12,12)=I+S$ (Figure~\ref{fig:paley1}).

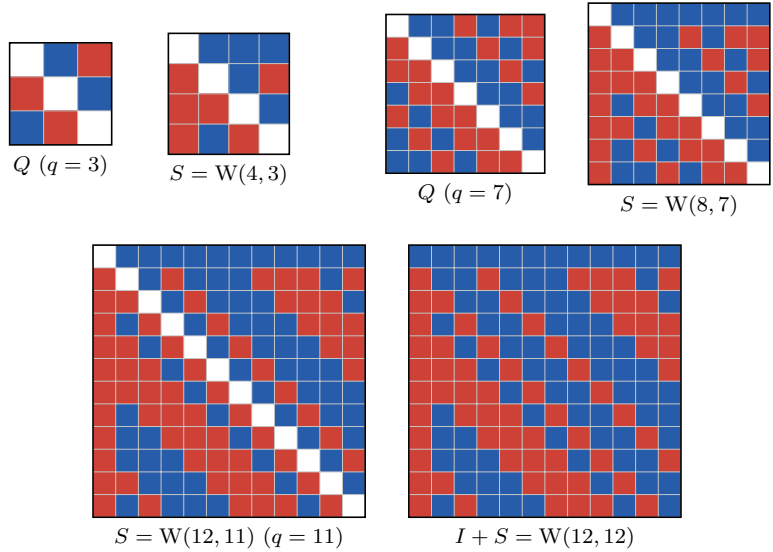
\begin{figure}[ht]
\centering
\begin{tikzpicture}[x=0.450cm,y=0.450cm,baseline=(mid)]
\coordinate (mid) at (1.5,-1.5);
\foreach \c/\r in {1/0,2/1,0/2}{\fill[wplus] (\c,-\r) rectangle ++(1,-1);}
\foreach \c/\r in {2/0,0/1,1/2}{\fill[wminus] (\c,-\r) rectangle ++(1,-1);}
\draw[xstep=1,ystep=1,gray!45,line width=0.12pt] (0,-3) grid (3,0);
\draw[line width=0.6pt] (0,0) rectangle (3,-3);
\node[font=\scriptsize,below,align=center] at (1.5,-3) {$Q$ ($q=3$)};
\end{tikzpicture}
\hspace{0.30cm}
\begin{tikzpicture}[x=0.400cm,y=0.400cm,baseline=(mid)]
\coordinate (mid) at (2.0,-2.0);
\foreach \c/\r in {1/0,2/0,3/0,2/1,3/2,1/3}{\fill[wplus] (\c,-\r) rectangle ++(1,-1);}
\foreach \c/\r in {0/1,3/1,0/2,1/2,0/3,2/3}{\fill[wminus] (\c,-\r) rectangle ++(1,-1);}
\draw[xstep=1,ystep=1,gray!45,line width=0.12pt] (0,-4) grid (4,0);
\draw[line width=0.6pt] (0,0) rectangle (4,-4);
\node[font=\scriptsize,below,align=center] at (2.0,-4) {$S=\W(4,3)$};
\end{tikzpicture}
\hspace{0.90cm}
\begin{tikzpicture}[x=0.300cm,y=0.300cm,baseline=(mid)]
\coordinate (mid) at (3.5,-3.5);
\foreach \c/\r in {1/0,2/0,4/0,2/1,3/1,5/1,3/2,4/2,6/2,0/3,4/3,5/3,1/4,5/4,6/4,0/5,2/5,6/5,0/6,1/6,3/6}{\fill[wplus] (\c,-\r) rectangle ++(1,-1);}
\foreach \c/\r in {3/0,5/0,6/0,0/1,4/1,6/1,0/2,1/2,5/2,1/3,2/3,6/3,0/4,2/4,3/4,1/5,3/5,4/5,2/6,4/6,5/6}{\fill[wminus] (\c,-\r) rectangle ++(1,-1);}
\draw[xstep=1,ystep=1,gray!45,line width=0.12pt] (0,-7) grid (7,0);
\draw[line width=0.6pt] (0,0) rectangle (7,-7);
\node[font=\scriptsize,below,align=center] at (3.5,-7) {$Q$ ($q=7$)};
\end{tikzpicture}
\hspace{0.30cm}
\begin{tikzpicture}[x=0.300cm,y=0.300cm,baseline=(mid)]
\coordinate (mid) at (4.0,-4.0);
\foreach \c/\r in {1/0,2/0,3/0,4/0,5/0,6/0,7/0,2/1,3/1,5/1,3/2,4/2,6/2,4/3,5/3,7/3,1/4,5/4,6/4,2/5,6/5,7/5,1/6,3/6,7/6,1/7,2/7,4/7}{\fill[wplus] (\c,-\r) rectangle ++(1,-1);}
\foreach \c/\r in {0/1,4/1,6/1,7/1,0/2,1/2,5/2,7/2,0/3,1/3,2/3,6/3,0/4,2/4,3/4,7/4,0/5,1/5,3/5,4/5,0/6,2/6,4/6,5/6,0/7,3/7,5/7,6/7}{\fill[wminus] (\c,-\r) rectangle ++(1,-1);}
\draw[xstep=1,ystep=1,gray!45,line width=0.12pt] (0,-8) grid (8,0);
\draw[line width=0.6pt] (0,0) rectangle (8,-8);
\node[font=\scriptsize,below,align=center] at (4.0,-8) {$S=\W(8,7)$};
\end{tikzpicture}

\medskip
\begin{tikzpicture}[x=0.300cm,y=0.300cm,baseline=(mid)]
\coordinate (mid) at (6.0,-6.0);
\foreach \c/\r in {1/0,2/0,3/0,4/0,5/0,6/0,7/0,8/0,9/0,10/0,11/0,2/1,4/1,5/1,6/1,10/1,3/2,5/2,6/2,7/2,11/2,1/3,4/3,6/3,7/3,8/3,2/4,5/4,7/4,8/4,9/4,3/5,6/5,8/5,9/5,10/5,4/6,7/6,9/6,10/6,11/6,1/7,5/7,8/7,10/7,11/7,1/8,2/8,6/8,9/8,11/8,1/9,2/9,3/9,7/9,10/9,2/10,3/10,4/10,8/10,11/10,1/11,3/11,4/11,5/11,9/11}{\fill[wplus] (\c,-\r) rectangle ++(1,-1);}
\foreach \c/\r in {0/1,3/1,7/1,8/1,9/1,11/1,0/2,1/2,4/2,8/2,9/2,10/2,0/3,2/3,5/3,9/3,10/3,11/3,0/4,1/4,3/4,6/4,10/4,11/4,0/5,1/5,2/5,4/5,7/5,11/5,0/6,1/6,2/6,3/6,5/6,8/6,0/7,2/7,3/7,4/7,6/7,9/7,0/8,3/8,4/8,5/8,7/8,10/8,0/9,4/9,5/9,6/9,8/9,11/9,0/10,1/10,5/10,6/10,7/10,9/10,0/11,2/11,6/11,7/11,8/11,10/11}{\fill[wminus] (\c,-\r) rectangle ++(1,-1);}
\draw[xstep=1,ystep=1,gray!45,line width=0.12pt] (0,-12) grid (12,0);
\draw[line width=0.6pt] (0,0) rectangle (12,-12);
\node[font=\scriptsize,below,align=center] at (6.0,-12) {$S=\W(12,11)$ ($q=11$)};
\end{tikzpicture}
\hspace{0.30cm}
\begin{tikzpicture}[x=0.300cm,y=0.300cm,baseline=(mid)]
\coordinate (mid) at (6.0,-6.0);
\foreach \c/\r in {0/0,1/0,2/0,3/0,4/0,5/0,6/0,7/0,8/0,9/0,10/0,11/0,1/1,2/1,4/1,5/1,6/1,10/1,2/2,3/2,5/2,6/2,7/2,11/2,1/3,3/3,4/3,6/3,7/3,8/3,2/4,4/4,5/4,7/4,8/4,9/4,3/5,5/5,6/5,8/5,9/5,10/5,4/6,6/6,7/6,9/6,10/6,11/6,1/7,5/7,7/7,8/7,10/7,11/7,1/8,2/8,6/8,8/8,9/8,11/8,1/9,2/9,3/9,7/9,9/9,10/9,2/10,3/10,4/10,8/10,10/10,11/10,1/11,3/11,4/11,5/11,9/11,11/11}{\fill[wplus] (\c,-\r) rectangle ++(1,-1);}
\foreach \c/\r in {0/1,3/1,7/1,8/1,9/1,11/1,0/2,1/2,4/2,8/2,9/2,10/2,0/3,2/3,5/3,9/3,10/3,11/3,0/4,1/4,3/4,6/4,10/4,11/4,0/5,1/5,2/5,4/5,7/5,11/5,0/6,1/6,2/6,3/6,5/6,8/6,0/7,2/7,3/7,4/7,6/7,9/7,0/8,3/8,4/8,5/8,7/8,10/8,0/9,4/9,5/9,6/9,8/9,11/9,0/10,1/10,5/10,6/10,7/10,9/10,0/11,2/11,6/11,7/11,8/11,10/11}{\fill[wminus] (\c,-\r) rectangle ++(1,-1);}
\draw[xstep=1,ystep=1,gray!45,line width=0.12pt] (0,-12) grid (12,0);
\draw[line width=0.6pt] (0,0) rectangle (12,-12);
\node[font=\scriptsize,below,align=center] at (6.0,-12) {$I+S=\W(12,12)$};
\end{tikzpicture}
\caption{Paley I. Top: $Q$ and $S=\W(4,3)$ for $q=3$; $Q$ and $S=\W(8,7)$ for $q=7$. Bottom: $S=\W(12,11)$ and $I+S=\W(12,12)$ for $q=11$. All $S$ and $I+S$ shown are stored matrices.}\label{fig:paley1}
\end{figure}

\paragraph{Paley II.}
The program applies Paley II only to pairs $(n,n)$ (with $n$ even) for which $q=n/2-1\ge3$ is a prime power with $q\equiv1\pmod4$, and only after methods 1--3 and Paley I have failed. Now $Q\T=Q$ (Lemma~\ref{lem:Q}(iii)); define the \emph{symmetric conference matrix} and the block matrix
\[
K=\begin{pmatrix}0&\one\T\\\one&Q\end{pmatrix},\qquad
H=\begin{pmatrix}K+I&K-I\\K-I&-K-I\end{pmatrix}\qquad(\text{each block of order }q+1).
\]
\begin{proposition}\label{prop:paley2}
For $q\equiv1\pmod4$: $K\T=K$, $K^2=qI_{q+1}$, and $H$ is a Hadamard matrix of order $n=2(q+1)$.
\end{proposition}
\begin{proof}
$K\T=K$ because $Q\T=Q$ and the two borders are equal. As in the proof of Proposition~\ref{prop:paley1},
\[
KK\T=\begin{pmatrix}\one\T\one&\one\T Q\T\\Q\one&\one\one\T+QQ\T\end{pmatrix}=\begin{pmatrix}q&0\\0&J+qI_q-J\end{pmatrix}=qI_{q+1},
\]
so $K^2=KK\T=qI_{q+1}$. The matrix $K$ has zero diagonal and entries $\pm1$ elsewhere, so $K+I$ and $K-I$ have all entries equal to $\pm1$ (diagonal $+1$, respectively $-1$): $H$ is a $\pm1$ matrix. Since $K$ is symmetric, so is $H$, hence $HH\T=H^2$, whose blocks are
\begin{align*}
(1,1)&:\ (K+I)^2+(K-I)^2=2K^2+2I=2(q+1)I,\\
(1,2)&:\ (K+I)(K-I)+(K-I)(-K-I)=(K^2-I)+(-K^2+I)=0,\\
(2,1)&:\ (K-I)(K+I)+(-K-I)(K-I)=(K^2-I)+(-K^2+I)=0,\\
(2,2)&:\ (K-I)^2+(-K-I)^2=2K^2+2I=2(q+1)I.
\end{align*}
Hence $HH\T=2(q+1)I_{2(q+1)}$.
\end{proof}
Paley II is used for exactly four pairs, $(n,n)$ with $n=36,52,76,100$ and $q=17,25,37,49$ (method \texttt{paley2}); for the other even $n=2(q+1)$ with $q\equiv1\pmod4$ the pair $(n,n)$ is already solved by an earlier method (for instance for $q=5$, $n=12$, by Paley I with $q=11$). 

\emph{Example.} $\W(36,36)$, $q=17$. Here $17$ is prime, the nonzero squares modulo $17$ are $1,2,4,8,9,13,15,16$ and $\chi(-1)=\chi(16)=1$. Figure~\ref{fig:paley2} shows the $18\times18$ matrix $K$ and the stored Hadamard matrix $\W(36,36)=H$.

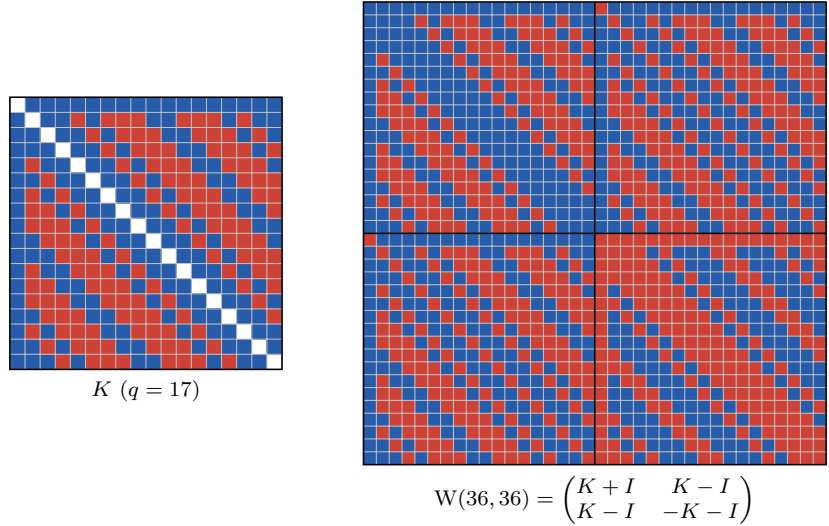
\begin{figure}[ht]
\centering
\begin{tikzpicture}[x=0.200cm,y=0.200cm,baseline=(mid)]
\coordinate (mid) at (9.0,-9.0);
\foreach \c/\r in {1/0,2/0,3/0,4/0,5/0,6/0,7/0,8/0,9/0,10/0,11/0,12/0,13/0,14/0,15/0,16/0,17/0,0/1,2/1,3/1,5/1,9/1,10/1,14/1,16/1,17/1,0/2,1/2,3/2,4/2,6/2,10/2,11/2,15/2,17/2,0/3,1/3,2/3,4/3,5/3,7/3,11/3,12/3,16/3,0/4,2/4,3/4,5/4,6/4,8/4,12/4,13/4,17/4,0/5,1/5,3/5,4/5,6/5,7/5,9/5,13/5,14/5,0/6,2/6,4/6,5/6,7/6,8/6,10/6,14/6,15/6,0/7,3/7,5/7,6/7,8/7,9/7,11/7,15/7,16/7,0/8,4/8,6/8,7/8,9/8,10/8,12/8,16/8,17/8,0/9,1/9,5/9,7/9,8/9,10/9,11/9,13/9,17/9,0/10,1/10,2/10,6/10,8/10,9/10,11/10,12/10,14/10,0/11,2/11,3/11,7/11,9/11,10/11,12/11,13/11,15/11,0/12,3/12,4/12,8/12,10/12,11/12,13/12,14/12,16/12,0/13,4/13,5/13,9/13,11/13,12/13,14/13,15/13,17/13,0/14,1/14,5/14,6/14,10/14,12/14,13/14,15/14,16/14,0/15,2/15,6/15,7/15,11/15,13/15,14/15,16/15,17/15,0/16,1/16,3/16,7/16,8/16,12/16,14/16,15/16,17/16,0/17,1/17,2/17,4/17,8/17,9/17,13/17,15/17,16/17}{\fill[wplus] (\c,-\r) rectangle ++(1,-1);}
\foreach \c/\r in {4/1,6/1,7/1,8/1,11/1,12/1,13/1,15/1,5/2,7/2,8/2,9/2,12/2,13/2,14/2,16/2,6/3,8/3,9/3,10/3,13/3,14/3,15/3,17/3,1/4,7/4,9/4,10/4,11/4,14/4,15/4,16/4,2/5,8/5,10/5,11/5,12/5,15/5,16/5,17/5,1/6,3/6,9/6,11/6,12/6,13/6,16/6,17/6,1/7,2/7,4/7,10/7,12/7,13/7,14/7,17/7,1/8,2/8,3/8,5/8,11/8,13/8,14/8,15/8,2/9,3/9,4/9,6/9,12/9,14/9,15/9,16/9,3/10,4/10,5/10,7/10,13/10,15/10,16/10,17/10,1/11,4/11,5/11,6/11,8/11,14/11,16/11,17/11,1/12,2/12,5/12,6/12,7/12,9/12,15/12,17/12,1/13,2/13,3/13,6/13,7/13,8/13,10/13,16/13,2/14,3/14,4/14,7/14,8/14,9/14,11/14,17/14,1/15,3/15,4/15,5/15,8/15,9/15,10/15,12/15,2/16,4/16,5/16,6/16,9/16,10/16,11/16,13/16,3/17,5/17,6/17,7/17,10/17,11/17,12/17,14/17}{\fill[wminus] (\c,-\r) rectangle ++(1,-1);}
\draw[xstep=1,ystep=1,gray!45,line width=0.12pt] (0,-18) grid (18,0);
\draw[line width=0.6pt] (0,0) rectangle (18,-18);
\node[font=\scriptsize,below,align=center] at (9.0,-18) {$K$ ($q=17$)};
\end{tikzpicture}
\hspace{0.80cm}
\begin{tikzpicture}[x=0.170cm,y=0.170cm,baseline=(mid)]
\coordinate (mid) at (18.0,-18.0);
\foreach \c/\r in {0/0,1/0,2/0,3/0,4/0,5/0,6/0,7/0,8/0,9/0,10/0,11/0,12/0,13/0,14/0,15/0,16/0,17/0,19/0,20/0,21/0,22/0,23/0,24/0,25/0,26/0,27/0,28/0,29/0,30/0,31/0,32/0,33/0,34/0,35/0,0/1,1/1,2/1,3/1,5/1,9/1,10/1,14/1,16/1,17/1,18/1,20/1,21/1,23/1,27/1,28/1,32/1,34/1,35/1,0/2,1/2,2/2,3/2,4/2,6/2,10/2,11/2,15/2,17/2,18/2,19/2,21/2,22/2,24/2,28/2,29/2,33/2,35/2,0/3,1/3,2/3,3/3,4/3,5/3,7/3,11/3,12/3,16/3,18/3,19/3,20/3,22/3,23/3,25/3,29/3,30/3,34/3,0/4,2/4,3/4,4/4,5/4,6/4,8/4,12/4,13/4,17/4,18/4,20/4,21/4,23/4,24/4,26/4,30/4,31/4,35/4,0/5,1/5,3/5,4/5,5/5,6/5,7/5,9/5,13/5,14/5,18/5,19/5,21/5,22/5,24/5,25/5,27/5,31/5,32/5,0/6,2/6,4/6,5/6,6/6,7/6,8/6,10/6,14/6,15/6,18/6,20/6,22/6,23/6,25/6,26/6,28/6,32/6,33/6,0/7,3/7,5/7,6/7,7/7,8/7,9/7,11/7,15/7,16/7,18/7,21/7,23/7,24/7,26/7,27/7,29/7,33/7,34/7,0/8,4/8,6/8,7/8,8/8,9/8,10/8,12/8,16/8,17/8,18/8,22/8,24/8,25/8,27/8,28/8,30/8,34/8,35/8,0/9,1/9,5/9,7/9,8/9,9/9,10/9,11/9,13/9,17/9,18/9,19/9,23/9,25/9,26/9,28/9,29/9,31/9,35/9,0/10,1/10,2/10,6/10,8/10,9/10,10/10,11/10,12/10,14/10,18/10,19/10,20/10,24/10,26/10,27/10,29/10,30/10,32/10,0/11,2/11,3/11,7/11,9/11,10/11,11/11,12/11,13/11,15/11,18/11,20/11,21/11,25/11,27/11,28/11,30/11,31/11,33/11,0/12,3/12,4/12,8/12,10/12,11/12,12/12,13/12,14/12,16/12,18/12,21/12,22/12,26/12,28/12,29/12,31/12,32/12,34/12,0/13,4/13,5/13,9/13,11/13,12/13,13/13,14/13,15/13,17/13,18/13,22/13,23/13,27/13,29/13,30/13,32/13,33/13,35/13,0/14,1/14,5/14,6/14,10/14,12/14,13/14,14/14,15/14,16/14,18/14,19/14,23/14,24/14,28/14,30/14,31/14,33/14,34/14,0/15,2/15,6/15,7/15,11/15,13/15,14/15,15/15,16/15,17/15,18/15,20/15,24/15,25/15,29/15,31/15,32/15,34/15,35/15,0/16,1/16,3/16,7/16,8/16,12/16,14/16,15/16,16/16,17/16,18/16,19/16,21/16,25/16,26/16,30/16,32/16,33/16,35/16,0/17,1/17,2/17,4/17,8/17,9/17,13/17,15/17,16/17,17/17,18/17,19/17,20/17,22/17,26/17,27/17,31/17,33/17,34/17,1/18,2/18,3/18,4/18,5/18,6/18,7/18,8/18,9/18,10/18,11/18,12/18,13/18,14/18,15/18,16/18,17/18,0/19,2/19,3/19,5/19,9/19,10/19,14/19,16/19,17/19,22/19,24/19,25/19,26/19,29/19,30/19,31/19,33/19,0/20,1/20,3/20,4/20,6/20,10/20,11/20,15/20,17/20,23/20,25/20,26/20,27/20,30/20,31/20,32/20,34/20,0/21,1/21,2/21,4/21,5/21,7/21,11/21,12/21,16/21,24/21,26/21,27/21,28/21,31/21,32/21,33/21,35/21,0/22,2/22,3/22,5/22,6/22,8/22,12/22,13/22,17/22,19/22,25/22,27/22,28/22,29/22,32/22,33/22,34/22,0/23,1/23,3/23,4/23,6/23,7/23,9/23,13/23,14/23,20/23,26/23,28/23,29/23,30/23,33/23,34/23,35/23,0/24,2/24,4/24,5/24,7/24,8/24,10/24,14/24,15/24,19/24,21/24,27/24,29/24,30/24,31/24,34/24,35/24,0/25,3/25,5/25,6/25,8/25,9/25,11/25,15/25,16/25,19/25,20/25,22/25,28/25,30/25,31/25,32/25,35/25,0/26,4/26,6/26,7/26,9/26,10/26,12/26,16/26,17/26,19/26,20/26,21/26,23/26,29/26,31/26,32/26,33/26,0/27,1/27,5/27,7/27,8/27,10/27,11/27,13/27,17/27,20/27,21/27,22/27,24/27,30/27,32/27,33/27,34/27,0/28,1/28,2/28,6/28,8/28,9/28,11/28,12/28,14/28,21/28,22/28,23/28,25/28,31/28,33/28,34/28,35/28,0/29,2/29,3/29,7/29,9/29,10/29,12/29,13/29,15/29,19/29,22/29,23/29,24/29,26/29,32/29,34/29,35/29,0/30,3/30,4/30,8/30,10/30,11/30,13/30,14/30,16/30,19/30,20/30,23/30,24/30,25/30,27/30,33/30,35/30,0/31,4/31,5/31,9/31,11/31,12/31,14/31,15/31,17/31,19/31,20/31,21/31,24/31,25/31,26/31,28/31,34/31,0/32,1/32,5/32,6/32,10/32,12/32,13/32,15/32,16/32,20/32,21/32,22/32,25/32,26/32,27/32,29/32,35/32,0/33,2/33,6/33,7/33,11/33,13/33,14/33,16/33,17/33,19/33,21/33,22/33,23/33,26/33,27/33,28/33,30/33,0/34,1/34,3/34,7/34,8/34,12/34,14/34,15/34,17/34,20/34,22/34,23/34,24/34,27/34,28/34,29/34,31/34,0/35,1/35,2/35,4/35,8/35,9/35,13/35,15/35,16/35,21/35,23/35,24/35,25/35,28/35,29/35,30/35,32/35}{\fill[wplus] (\c,-\r) rectangle ++(1,-1);}
\foreach \c/\r in {18/0,4/1,6/1,7/1,8/1,11/1,12/1,13/1,15/1,19/1,22/1,24/1,25/1,26/1,29/1,30/1,31/1,33/1,5/2,7/2,8/2,9/2,12/2,13/2,14/2,16/2,20/2,23/2,25/2,26/2,27/2,30/2,31/2,32/2,34/2,6/3,8/3,9/3,10/3,13/3,14/3,15/3,17/3,21/3,24/3,26/3,27/3,28/3,31/3,32/3,33/3,35/3,1/4,7/4,9/4,10/4,11/4,14/4,15/4,16/4,19/4,22/4,25/4,27/4,28/4,29/4,32/4,33/4,34/4,2/5,8/5,10/5,11/5,12/5,15/5,16/5,17/5,20/5,23/5,26/5,28/5,29/5,30/5,33/5,34/5,35/5,1/6,3/6,9/6,11/6,12/6,13/6,16/6,17/6,19/6,21/6,24/6,27/6,29/6,30/6,31/6,34/6,35/6,1/7,2/7,4/7,10/7,12/7,13/7,14/7,17/7,19/7,20/7,22/7,25/7,28/7,30/7,31/7,32/7,35/7,1/8,2/8,3/8,5/8,11/8,13/8,14/8,15/8,19/8,20/8,21/8,23/8,26/8,29/8,31/8,32/8,33/8,2/9,3/9,4/9,6/9,12/9,14/9,15/9,16/9,20/9,21/9,22/9,24/9,27/9,30/9,32/9,33/9,34/9,3/10,4/10,5/10,7/10,13/10,15/10,16/10,17/10,21/10,22/10,23/10,25/10,28/10,31/10,33/10,34/10,35/10,1/11,4/11,5/11,6/11,8/11,14/11,16/11,17/11,19/11,22/11,23/11,24/11,26/11,29/11,32/11,34/11,35/11,1/12,2/12,5/12,6/12,7/12,9/12,15/12,17/12,19/12,20/12,23/12,24/12,25/12,27/12,30/12,33/12,35/12,1/13,2/13,3/13,6/13,7/13,8/13,10/13,16/13,19/13,20/13,21/13,24/13,25/13,26/13,28/13,31/13,34/13,2/14,3/14,4/14,7/14,8/14,9/14,11/14,17/14,20/14,21/14,22/14,25/14,26/14,27/14,29/14,32/14,35/14,1/15,3/15,4/15,5/15,8/15,9/15,10/15,12/15,19/15,21/15,22/15,23/15,26/15,27/15,28/15,30/15,33/15,2/16,4/16,5/16,6/16,9/16,10/16,11/16,13/16,20/16,22/16,23/16,24/16,27/16,28/16,29/16,31/16,34/16,3/17,5/17,6/17,7/17,10/17,11/17,12/17,14/17,21/17,23/17,24/17,25/17,28/17,29/17,30/17,32/17,35/17,0/18,18/18,19/18,20/18,21/18,22/18,23/18,24/18,25/18,26/18,27/18,28/18,29/18,30/18,31/18,32/18,33/18,34/18,35/18,1/19,4/19,6/19,7/19,8/19,11/19,12/19,13/19,15/19,18/19,19/19,20/19,21/19,23/19,27/19,28/19,32/19,34/19,35/19,2/20,5/20,7/20,8/20,9/20,12/20,13/20,14/20,16/20,18/20,19/20,20/20,21/20,22/20,24/20,28/20,29/20,33/20,35/20,3/21,6/21,8/21,9/21,10/21,13/21,14/21,15/21,17/21,18/21,19/21,20/21,21/21,22/21,23/21,25/21,29/21,30/21,34/21,1/22,4/22,7/22,9/22,10/22,11/22,14/22,15/22,16/22,18/22,20/22,21/22,22/22,23/22,24/22,26/22,30/22,31/22,35/22,2/23,5/23,8/23,10/23,11/23,12/23,15/23,16/23,17/23,18/23,19/23,21/23,22/23,23/23,24/23,25/23,27/23,31/23,32/23,1/24,3/24,6/24,9/24,11/24,12/24,13/24,16/24,17/24,18/24,20/24,22/24,23/24,24/24,25/24,26/24,28/24,32/24,33/24,1/25,2/25,4/25,7/25,10/25,12/25,13/25,14/25,17/25,18/25,21/25,23/25,24/25,25/25,26/25,27/25,29/25,33/25,34/25,1/26,2/26,3/26,5/26,8/26,11/26,13/26,14/26,15/26,18/26,22/26,24/26,25/26,26/26,27/26,28/26,30/26,34/26,35/26,2/27,3/27,4/27,6/27,9/27,12/27,14/27,15/27,16/27,18/27,19/27,23/27,25/27,26/27,27/27,28/27,29/27,31/27,35/27,3/28,4/28,5/28,7/28,10/28,13/28,15/28,16/28,17/28,18/28,19/28,20/28,24/28,26/28,27/28,28/28,29/28,30/28,32/28,1/29,4/29,5/29,6/29,8/29,11/29,14/29,16/29,17/29,18/29,20/29,21/29,25/29,27/29,28/29,29/29,30/29,31/29,33/29,1/30,2/30,5/30,6/30,7/30,9/30,12/30,15/30,17/30,18/30,21/30,22/30,26/30,28/30,29/30,30/30,31/30,32/30,34/30,1/31,2/31,3/31,6/31,7/31,8/31,10/31,13/31,16/31,18/31,22/31,23/31,27/31,29/31,30/31,31/31,32/31,33/31,35/31,2/32,3/32,4/32,7/32,8/32,9/32,11/32,14/32,17/32,18/32,19/32,23/32,24/32,28/32,30/32,31/32,32/32,33/32,34/32,1/33,3/33,4/33,5/33,8/33,9/33,10/33,12/33,15/33,18/33,20/33,24/33,25/33,29/33,31/33,32/33,33/33,34/33,35/33,2/34,4/34,5/34,6/34,9/34,10/34,11/34,13/34,16/34,18/34,19/34,21/34,25/34,26/34,30/34,32/34,33/34,34/34,35/34,3/35,5/35,6/35,7/35,10/35,11/35,12/35,14/35,17/35,18/35,19/35,20/35,22/35,26/35,27/35,31/35,33/35,34/35,35/35}{\fill[wminus] (\c,-\r) rectangle ++(1,-1);}
\draw[xstep=1,ystep=1,gray!45,line width=0.12pt] (0,-36) grid (36,0);
\draw[line width=0.6pt] (18,0) -- (18,-36);
\draw[line width=0.6pt] (0,-18) -- (36,-18);
\draw[line width=0.6pt] (0,0) rectangle (36,-36);
\node[font=\scriptsize,below,align=center] at (18.0,-36) {$\W(36,36)=\begin{pmatrix}K+I&K-I\\K-I&-K-I\end{pmatrix}$};
\end{tikzpicture}
\caption{Paley II for $q=17$: the symmetric conference matrix $K$ (order $18$, auxiliary) and the stored matrix $\W(36,36)=\begin{pmatrix}K+I&K-I\\K-I&-K-I\end{pmatrix}$.}\label{fig:paley2}
\end{figure}
\subsection{Stage 2: the Goethals--Seidel search}\label{sec:stage2}

A pair $(n,k)$ that none of the methods of Stage~1 solves is searched for among the Goethals--Seidel matrices of Section~\ref{sec:gs}, with $m=n/4$: the program looks for four sequences $A,B,C,D$ of length $m$ with $P(0)=k$ and $P(s)=0$ for $1\le s\le m-1$ (Corollary~\ref{cor:paf}), and then builds $G$ from~\eqref{eq:G}. There are two branches, according to the size of $m$ (Figure~\ref{fig:flow}): for $m\le6$ ($n\le24$) the search is exhaustive; for $m\ge7$ ($n\ge28$) the profile filter is applied first and then simulated annealing. The threshold is the program's option \texttt{--exhaustive-max-m}, default $6$; the numbers $24$ and $527$ of Table~\ref{tab:methods} correspond to this default.

\subsubsection{Exhaustive search for $m\le 6$}\label{sec:exh}

There are only $3^m$ ternary sequences of length $m$ ($729$ for $m=6$). The program computes the vector $\bigl(\PAF_x(0),\dots,\PAF_x(m-1)\bigr)$ for every sequence $x$ and, for every unordered pair $\{A,B\}$ (with $A=B$ allowed; $729\cdot730/2=266{,}085$ pairs for $m=6$), the vector $P_{AB}=\PAF_A+\PAF_B$. These vectors are sorted. Since $P=P_{AB}+P_{CD}$, a Goethals--Seidel matrix $\W(4m,k)$ exists \emph{if and only if} there are two pairs with $P_{AB}+P_{CD}=(k,0,\dots,0)$, that is, $P_{CD}=(k,0,\dots,0)-P_{AB}$; for each pair $\{A,B\}$ the program looks this vector up in the sorted table by binary search (``meet in the middle''). If no pair matches, there is \emph{no} Goethals--Seidel matrix $\W(4m,k)$: the negative answer is a proof (status \texttt{none\_gs}). For $(n,k)=(16,15)$, $m=4$, no match exists, in agreement with Proposition~\ref{prop:parity}; this is the only pair of the catalogue with $n\le24$ in this situation.

\emph{Example.} $\W(12,7)$, $m=3$ (the example of Section~\ref{sec:ex127}). For $A=(-1,-1,1)$ and $B=(-1,0,0)$ we have $P_{AB}=(3,-1,-1)+(1,0,0)=(4,-1,-1)$, and for $C=(-1,-1,0)$ and $D=(-1,0,0)$ we have $P_{CD}=(2,1,1)+(1,0,0)=(3,1,1)$. Indeed $P_{AB}+P_{CD}=(7,0,0)$: the lookup of $(7,0,0)-P_{AB}=(3,1,1)$ in the table of the pair vectors succeeds and returns the pair $\{C,D\}$. These are the sequences stored in the catalogue (the program returns the first match it finds).

\subsubsection{The profile filter for $m\ge7$}\label{sec:filter}

Before annealing, the program lists the admissible sum tuples (Section~\ref{sec:profile}): all $(s_0,s_1,s_2,s_3)$ with $m\ge s_0\ge s_1\ge s_2\ge s_3\ge0$ and $\sum s_i^2=k$, for which weights $w_i=s_i+2u_i^{-}$ with integers $0\le u_i^{-}\le\lfloor(m-s_i)/2\rfloor$ and $\sum w_i=k$ exist; equivalently, $k-\sum s_i$ is even (automatic, since $s_i^2\equiv s_i\pmod2$) and
\[
0\le k-\sum_is_i\le2\sum_i\Bigl\lfloor\frac{m-s_i}{2}\Bigr\rfloor .
\]
If the list is empty, there is no admissible profile and hence no Goethals--Seidel matrix: status \texttt{none\_gs}. This happens for exactly five pairs of the catalogue, $(40,39)$, $(56,55)$, $(64,63)$, $(88,87)$ and $(96,95)$ ($m=10,14,16,22,24$); together with $(16,15)$, found by the exhaustive search, these are the six pairs of Proposition~\ref{prop:parity} that are not Paley orders. For instance, for $(40,39)$, $m=10$:
\begin{center}
\begin{tabular}{cccc}\toprule $(s_0,s_1,s_2,s_3)$ & $\sum s_i$ & $k-\sum s_i$ & $2\sum_i\lfloor(m-s_i)/2\rfloor$ \\ \midrule (6,1,1,1) & 9 & 30 & 28\\ (5,3,2,1) & 11 & 28 & 26\\ \bottomrule\end{tabular}

\small (the tuples with $\sum s_i^2=39$; each would need $k-\sum s_i\le2\sum\lfloor(m-s_i)/2\rfloor$)
\end{center}
so the list is empty. In the other cases the list is non-empty, and it is the starting point of the annealing. For $(68,55)$, $m=17$, there are three admissible sum tuples, $(7,2,1,1)$, $(6,3,3,1)$ and $(5,5,2,1)$; the stored matrix has sums $(5,5,1,2)$, a rearrangement of the third one, and weights $(17,13,9,16)$. At the other extreme, the filter leaves only three admissible sum tuples for $(108,95)$ and five for $(108,103)$ (a sum tuple may admit several assignments of the weights $w_i$, hence several full profiles).

\subsubsection{Simulated annealing for $m\ge7$}\label{sec:anneal}

For $m\ge7$ the number of candidates, $3^{4m}$, is far too large for an exhaustive search. \emph{Simulated annealing}~\cite{kgv} is a randomized search that explores candidates by small random changes, always accepting changes that improve a \emph{cost} and sometimes accepting changes that worsen it, with a probability that decreases during the run (the ``temperature'' $T$ is lowered), so that the search can escape from local minima. Here the cost measures how far $P(1),\dots,P(m-1)$ are from $0$. The program does the following for each pair; one pass through steps 1--6 is called an \emph{attempt}.
\begin{enumerate}
\item \emph{Sample a profile.} Choose $(s_0,\dots,s_3)$ uniformly at random among the admissible sum tuples. Write $w_i=s_i+2u_i^{-}$, where $u_i^{-}$ is the number of entries $-1$ in the $i$-th sequence, with $0\le u_i^{-}\le\lfloor(m-s_i)/2\rfloor$. Distribute the $\tfrac12(k-\sum_is_i)$ units of $u^{-}$ one at a time, each to a randomly chosen sequence that still has capacity, and finally shuffle the order of the four pairs $(w_i,s_i)$. (Uniform choice of the sum tuple does not make all full profiles equally likely.)
\item \emph{Initial state.} For each sequence choose $w_i$ positions at random, put $+1$ on $(w_i+s_i)/2$ of them and $-1$ on the others, and $0$ elsewhere.
\item \emph{Cost.} $E=\sum_{s=1}^{m-1}|P(s)|$. By Corollary~\ref{cor:paf}, $E=0$ together with the profile (which forces $P(0)=k$) is equivalent to $GG\T=kI$. The vector $P$ is computed once with the fast Fourier transform and afterwards updated incrementally. Changing one entry $x_p$ by $\eta$ (for instance from $-1$ to $0$, $\eta=+1$) changes $P(s)$ by $\eta\,(x_{p+s}+x_{p-s})$ for $s\ne0$ and $P(0)$ by $\eta(2x_p+\eta)$, indices modulo $m$ and $x_p$ the entry \emph{before} the change. Indeed, in $\PAF_x(s)=\sum_tx_tx_{t+s}$ only the two terms $t=p$ and $t=p-s$ contain $x_p$ (for $s\ne0$), contributing $x_px_{p+s}+x_{p-s}x_p$; and $\PAF_x(0)=\sum x_t^2$ changes from $x_p^2$ to $(x_p+\eta)^2$. Each update costs $O(m)$ operations; no $n\times n$ matrix is formed during the search.
\item \emph{Moves.} Choose a sequence containing at least two distinct values, choose two positions holding different values, and swap them (two single-entry changes applied one after the other). With probability $0.2$ a second swap, in a randomly chosen sequence, is composed with the first. A swap preserves the pair $(w_i,s_i)$ of its sequence, so the walk never leaves the chosen profile.
\item \emph{Acceptance and cooling.} A move that changes the cost by $\Delta E$ is accepted with probability $\min\{1,e^{-\Delta E/T}\}$ (always when $\Delta E\le0$); a rejected move is undone. The temperature decreases geometrically from $T_0=3$ to $T_1=0.25$ over $N=\max(20000,300m)$ moves ($N=20000$ for all $m\le27$).
\item \emph{Restart.} If $E=0$ is not reached within $N$ moves, the state is discarded and a new attempt starts at step~1 (with a new random profile), until a solution is found or a time or attempt limit given as program options is reached (the default time limit is $6\cdot10^6$ seconds per pair, effectively unbounded). A solution is rebuilt as the array $G$ and passed to the exact check.
\end{enumerate}
These are the default parameters of the program (the options \texttt{--T0}, \texttt{--T1}, \texttt{--iters}); the optional warm start from $\W(n,k-2)$ (\texttt{--seed-prev}) was not used for any pair of the catalogue (the log records \texttt{seeded=False} for all annealed pairs). The failure of the annealing to find a solution is not evidence that none exists.

\emph{Example of a move.} Let $m=5$ and $x=(1,-1,0,1,1)$, with $\PAF_x=(4,1,-1,-1,1)$. Swapping the entries in positions $0$ and $1$ (values $1$ and $-1$) gives $x'=(-1,1,0,1,1)$. As two single-entry changes: position $0$ changes by $\eta=-2$ (from $1$ to $-1$), and then position $1$ changes by $\eta=+2$ (from $-1$ to $1$). The first change adds $\eta(x_{p+s}+x_{p-s})$, $p=0$, to $\PAF(s)$ for $s=1,\dots,4$, namely $(0,-2,-2,0)$, computed on $x$; the second adds $(-2,4,4,-2)$, computed on the intermediate state $(-1,-1,0,1,1)$ with $p=1$; and $\PAF(0)$ changes by $\eta(2x_p+\eta)$ twice, in total by $0$. The result is $\PAF_{x'}=(4,-1,1,1,-1)$, which is also what one gets by recomputing from the definition.

\emph{Example of a search: $\W(68,55)$.} The catalogue matrix for $(n,k)=(68,55)$, $m=17$, was found on the fourth attempt of the original run ($1.65$~s). Its profile is $(w_i,s_i)=((17,5),(13,5),(9,1),(16,2))$ and $\sum w_i=55=\sum s_i^2=25+25+1+4$; the sequences are
\begin{center}
\mono{A = +-+---++++++-++-+}\\ \mono{B = +00++++--+--+00++}\\ \mono{C = +--0++00000-+0-0+}\\ \mono{D = +++---+-+-++0+-+-}
\end{center}
(written with $+,-,0$ for $1,-1,0$), and $P(s)=0$ for all $s\ne0$, $P(0)=55$. The log of the original run does not record the course of the cost. To \emph{illustrate} an annealing search, Figure~\ref{fig:trace} shows the cost $E$ along a re-run of the same annealing code (same moves, acceptance rule and schedule, instrumented to record $E$) for the same pair, with Python's random seed fixed to 1; this seed is the first seed in a scan $1,2,\dots$ whose run needed three or four attempts, to keep the figure short. It is therefore an illustration, not the log of the catalogue run, and the matrix it finds is a different one. The first attempt samples the profile $(w_i,s_i)=((12,2),(17,1),(9,1),(17,7))$ and starts with $E=84$; after $20{,}000$ moves it has $E=4$ and is discarded. The second and third attempts end with $E=12$ and $E=4$. The fourth, with profile $((16,6),(11,3),(17,3),(11,1))$, starts with $E=80$ and reaches $E=0$ after $3{,}700$ moves.

\begin{figure}[ht]
\centering
\begin{tikzpicture}[font=\scriptsize]
\draw[->] (0,0) -- (11.10,0) node[right] {moves};
\draw[->] (0,0) -- (0,4.20) node[above] {cost $E$};
\draw (0,0.000) -- (-0.08,0.000) node[left] {0};
\draw[gray!30,line width=0.2pt] (0,0.000) -- (11.00,0.000);
\draw (0,0.800) -- (-0.08,0.800) node[left] {20};
\draw[gray!30,line width=0.2pt] (0,0.800) -- (11.00,0.800);
\draw (0,1.600) -- (-0.08,1.600) node[left] {40};
\draw[gray!30,line width=0.2pt] (0,1.600) -- (11.00,1.600);
\draw (0,2.400) -- (-0.08,2.400) node[left] {60};
\draw[gray!30,line width=0.2pt] (0,2.400) -- (11.00,2.400);
\draw (0,3.200) -- (-0.08,3.200) node[left] {80};
\draw[gray!30,line width=0.2pt] (0,3.200) -- (11.00,3.200);
\draw (0,4.000) -- (-0.08,4.000) node[left] {100};
\draw[gray!30,line width=0.2pt] (0,4.000) -- (11.00,4.000);
\draw[tblue,line width=0.55pt] plot coordinates {(0.000,3.360) (0.009,1.440) (0.017,1.440) (0.026,1.440) (0.035,1.120) (0.043,0.800) (0.052,1.120) (0.060,0.640) (0.069,0.800) (0.078,0.320) (0.086,0.960) (0.095,0.960) (0.104,0.320) (0.112,0.320) (0.121,0.320) (0.130,0.320) (0.138,0.480) (0.147,0.320) (0.155,0.320) (0.164,0.320) (0.173,0.320) (0.181,0.320) (0.190,0.320) (0.199,0.320) (0.207,0.480) (0.216,0.960) (0.224,0.320) (0.233,0.320) (0.242,0.480) (0.250,1.120) (0.259,0.800) (0.268,0.640) (0.276,0.640) (0.285,0.320) (0.294,0.640) (0.302,0.960) (0.311,0.800) (0.319,0.640) (0.328,0.800) (0.337,0.640) (0.345,0.640) (0.354,0.480) (0.363,0.480) (0.371,0.640) (0.380,0.800) (0.389,0.800) (0.397,0.480) (0.406,0.800) (0.414,0.640) (0.423,0.640) (0.432,0.960) (0.440,0.960) (0.449,0.960) (0.458,0.640) (0.466,0.800) (0.475,0.480) (0.484,0.640) (0.492,0.640) (0.501,0.640) (0.509,0.480) (0.518,0.640) (0.527,0.800) (0.535,0.800) (0.544,0.480) (0.553,0.480) (0.561,0.480) (0.570,0.480) (0.578,0.480) (0.587,0.480) (0.596,0.640) (0.604,0.640) (0.622,0.480) (0.639,0.640) (0.648,0.640) (0.656,0.480) (0.665,0.480) (0.673,0.480) (0.682,0.480) (0.691,0.480) (0.699,0.320) (0.708,0.320) (0.717,0.960) (0.725,0.800) (0.734,0.640) (0.743,0.640) (0.751,0.640) (0.760,0.480) (0.768,0.480) (0.777,0.320) (0.786,0.320) (0.794,0.320) (0.803,0.320) (0.812,0.320) (0.820,0.320) (0.829,0.320) (0.838,0.480) (0.846,0.480) (0.855,0.480) (0.863,0.480) (0.872,0.320) (0.881,0.320) (0.889,0.320) (0.898,0.320) (0.907,0.320) (0.915,0.320) (0.924,0.320) (0.932,0.320) (0.941,0.480) (0.950,0.320) (0.958,0.320) (0.967,0.320) (0.976,0.320) (0.984,0.320) (0.993,0.320) (1.002,0.160) (1.010,0.320) (1.019,0.160) (1.027,0.160) (1.036,0.160) (1.045,0.160) (1.053,0.160) (1.062,0.160) (1.071,0.160) (1.079,0.160) (1.088,0.160) (1.097,0.160) (1.105,0.160) (1.114,0.320) (1.122,0.320) (1.131,0.320) (1.140,0.320) (1.148,0.480) (1.157,0.480) (1.166,0.480) (1.174,0.480) (1.183,0.480) (1.192,0.480) (1.200,0.480) (1.209,0.800) (1.217,0.480) (1.226,0.480) (1.235,0.480) (1.243,0.480) (1.252,0.480) (1.261,0.480) (1.269,0.480) (1.278,0.480) (1.286,0.320) (1.295,0.320) (1.304,0.320) (1.312,0.480) (1.321,0.320) (1.330,0.320) (1.338,0.320) (1.347,0.320) (1.356,0.800) (1.364,0.640) (1.373,0.480) (1.381,0.480) (1.390,0.480) (1.399,0.480) (1.407,0.480) (1.416,0.320) (1.425,0.320) (1.433,0.320) (1.442,0.320) (1.451,0.320) (1.459,0.320) (1.468,0.320) (1.476,0.320) (1.485,0.320) (1.494,0.320) (1.502,0.320) (1.511,0.320) (1.520,0.320) (1.528,0.320) (1.537,0.320) (1.546,0.320) (1.554,0.320) (1.563,0.320) (1.571,0.320) (1.580,0.320) (1.589,0.320) (1.597,0.320) (1.606,0.320) (1.615,0.320) (1.623,0.320) (1.632,0.320) (1.641,0.320) (1.649,0.320) (1.658,0.320) (1.666,0.320) (1.675,0.320) (1.684,0.320) (1.692,0.320) (1.701,0.320) (1.710,0.320) (1.718,0.320) (1.727,0.320) (1.735,0.320) (1.744,0.320) (1.753,0.320) (1.761,0.320) (1.770,0.320) (1.779,0.320) (1.787,0.320) (1.796,0.320) (1.805,0.320) (1.813,0.320) (1.822,0.320) (1.830,0.320) (1.839,0.320) (1.848,0.320) (1.856,0.320) (1.865,0.320) (1.874,0.320) (1.882,0.320) (1.891,0.320) (1.900,0.320) (1.908,0.320) (1.917,0.320) (1.925,0.320) (1.934,0.320) (1.943,0.320) (1.951,0.320) (1.960,0.320) (1.969,0.320) (1.977,0.320) (1.986,0.320) (1.995,0.320) (2.003,0.320) (2.012,0.320) (2.020,0.320) (2.029,0.320) (2.038,0.320) (2.046,0.320) (2.055,0.320) (2.064,0.320) (2.072,0.320) (2.081,0.320) (2.089,0.320) (2.098,0.320) (2.107,0.320) (2.115,0.320) (2.124,0.320) (2.133,0.320) (2.141,0.320) (2.150,0.320) (2.159,0.320) (2.167,0.320) (2.176,0.320) (2.184,0.320) (2.193,0.320) (2.202,0.320) (2.210,0.320) (2.219,0.320) (2.228,0.320) (2.236,0.320) (2.245,0.320) (2.254,0.320) (2.262,0.320) (2.271,0.320) (2.279,0.320) (2.288,0.320) (2.297,0.320) (2.305,0.160) (2.314,0.160) (2.323,0.160) (2.331,0.160) (2.340,0.160) (2.349,0.160) (2.357,0.160) (2.366,0.160) (2.374,0.160) (2.383,0.160) (2.392,0.160) (2.400,0.160) (2.409,0.160) (2.418,0.160) (2.426,0.160) (2.435,0.160) (2.443,0.160) (2.452,0.160) (2.461,0.160) (2.469,0.160) (2.478,0.160) (2.487,0.160) (2.495,0.160) (2.504,0.160) (2.513,0.160) (2.521,0.160) (2.530,0.160) (2.538,0.160) (2.547,0.160) (2.556,0.160) (2.564,0.160) (2.573,0.160) (2.582,0.160) (2.590,0.160) (2.599,0.160) (2.608,0.160) (2.616,0.160) (2.625,0.160) (2.633,0.160) (2.642,0.160) (2.651,0.160) (2.659,0.160) (2.668,0.160) (2.677,0.160) (2.685,0.160) (2.694,0.160) (2.703,0.160) (2.711,0.160) (2.720,0.160) (2.728,0.160) (2.737,0.160) (2.746,0.160) (2.754,0.160) (2.763,0.160) (2.772,0.160) (2.780,0.160) (2.789,0.160) (2.797,0.160) (2.806,0.160) (2.815,0.160) (2.823,0.160) (2.832,0.160) (2.841,0.160) (2.849,0.160) (2.858,0.160) (2.867,0.160) (2.875,0.160) (2.884,0.160) (2.892,0.160) (2.901,0.160) (2.910,0.160) (2.918,0.160) (2.927,0.160) (2.936,0.160) (2.944,0.160) (2.953,0.160) (2.962,0.160) (2.970,0.160) (2.979,0.160) (2.987,0.160) (2.996,0.160) (3.005,0.160) (3.013,0.160) (3.022,0.160) (3.031,0.160) (3.039,0.160) (3.048,0.160) (3.057,0.160) (3.065,0.160) (3.074,0.160) (3.082,0.160) (3.091,0.160) (3.100,0.160) (3.108,0.160) (3.117,0.160) (3.126,0.160) (3.134,0.160) (3.143,0.160) (3.151,0.160) (3.160,0.160) (3.169,0.160) (3.177,0.160) (3.186,0.160) (3.195,0.160) (3.203,0.160) (3.212,0.160) (3.221,0.160) (3.229,0.160) (3.238,0.160) (3.246,0.160) (3.255,0.160) (3.264,0.160) (3.272,0.160) (3.281,0.160) (3.290,0.160) (3.298,0.160) (3.307,0.160) (3.316,0.160) (3.324,0.160) (3.333,0.160) (3.341,0.160) (3.350,0.160) (3.359,0.160) (3.367,0.160) (3.376,0.160) (3.385,0.160) (3.393,0.160) (3.402,0.160) (3.411,0.160) (3.419,0.160) (3.428,0.160) (3.436,0.160) (3.445,0.160) (3.454,0.160)};
\node[tblue,above] at (1.73,4.05) {attempt 1};
\draw[dashed,gray] (3.454,0) -- (3.454,4.00);
\draw[torange,line width=0.55pt] plot coordinates {(3.454,2.720) (3.462,1.120) (3.471,1.120) (3.480,2.080) (3.488,0.480) (3.497,0.800) (3.505,0.640) (3.514,0.800) (3.523,0.800) (3.531,0.640) (3.540,0.640) (3.549,0.960) (3.557,0.640) (3.566,0.640) (3.575,0.640) (3.583,1.280) (3.592,1.120) (3.600,1.280) (3.609,0.640) (3.618,0.640) (3.626,0.640) (3.635,0.480) (3.644,0.640) (3.652,0.800) (3.661,0.480) (3.670,0.640) (3.678,0.960) (3.687,1.280) (3.695,0.800) (3.704,0.800) (3.713,0.480) (3.721,0.640) (3.730,0.640) (3.739,0.320) (3.747,0.480) (3.756,1.120) (3.765,0.640) (3.773,0.960) (3.782,0.480) (3.790,0.960) (3.799,0.640) (3.808,0.800) (3.816,0.960) (3.825,0.640) (3.834,0.640) (3.842,0.640) (3.851,0.640) (3.859,0.640) (3.868,0.640) (3.877,0.640) (3.885,0.320) (3.894,0.480) (3.903,0.480) (3.911,0.800) (3.920,0.640) (3.929,0.640) (3.937,0.640) (3.946,0.640) (3.954,0.640) (3.963,0.640) (3.972,0.480) (3.980,0.480) (3.989,0.480) (3.998,0.480) (4.006,0.480) (4.015,0.640) (4.024,0.640) (4.032,0.480) (4.041,0.640) (4.049,0.480) (4.058,0.480) (4.067,0.480) (4.075,0.480) (4.084,0.480) (4.093,0.480) (4.101,0.640) (4.110,0.480) (4.119,0.640) (4.127,0.480) (4.136,0.640) (4.144,0.480) (4.153,0.640) (4.162,0.640) (4.170,0.640) (4.179,0.480) (4.188,0.480) (4.196,0.480) (4.205,0.640) (4.214,0.800) (4.222,0.640) (4.231,0.480) (4.239,0.480) (4.248,0.480) (4.257,0.480) (4.265,0.480) (4.274,0.480) (4.283,0.480) (4.291,0.640) (4.300,0.640) (4.308,0.640) (4.317,0.640) (4.326,0.640) (4.334,0.640) (4.343,0.480) (4.352,0.800) (4.360,0.960) (4.369,0.800) (4.378,0.480) (4.386,0.640) (4.395,0.640) (4.403,0.480) (4.412,0.640) (4.421,0.640) (4.429,0.640) (4.438,0.480) (4.447,0.480) (4.455,0.480) (4.464,0.480) (4.473,0.480) (4.481,0.480) (4.490,0.480) (4.498,0.480) (4.507,0.480) (4.516,0.640) (4.524,0.640) (4.533,0.800) (4.542,0.640) (4.550,0.480) (4.559,0.640) (4.568,0.640) (4.576,0.640) (4.585,0.640) (4.593,0.640) (4.602,0.640) (4.611,0.640) (4.619,0.480) (4.628,0.480) (4.637,0.480) (4.645,0.480) (4.654,0.480) (4.662,0.480) (4.671,0.640) (4.680,0.640) (4.688,0.640) (4.697,0.640) (4.706,0.640) (4.714,0.640) (4.723,0.640) (4.732,0.640) (4.740,0.640) (4.749,0.640) (4.757,0.640) (4.766,0.640) (4.775,0.640) (4.783,0.480) (4.792,0.480) (4.801,0.480) (4.809,0.480) (4.818,0.480) (4.827,0.480) (4.835,0.480) (4.844,0.480) (4.852,0.480) (4.861,0.480) (4.870,0.480) (4.878,0.480) (4.887,0.480) (4.896,0.480) (4.904,0.480) (4.913,0.480) (4.922,0.480) (4.930,0.480) (4.939,0.480) (4.947,0.480) (4.956,0.480) (4.965,0.480) (4.973,0.480) (4.982,0.480) (4.991,0.480) (4.999,0.480) (5.008,0.480) (5.016,0.480) (5.025,0.480) (5.034,0.480) (5.042,0.480) (5.051,0.480) (5.060,0.480) (5.068,0.480) (5.077,0.480) (5.086,0.480) (5.094,0.480) (5.103,0.480) (5.111,0.480) (5.120,0.480) (5.129,0.480) (5.137,0.480) (5.146,0.480) (5.155,0.480) (5.163,0.480) (5.172,0.480) (5.181,0.480) (5.189,0.480) (5.198,0.480) (5.206,0.480) (5.215,0.480) (5.224,0.480) (5.232,0.480) (5.241,0.480) (5.250,0.480) (5.258,0.480) (5.267,0.480) (5.276,0.480) (5.284,0.480) (5.293,0.480) (5.301,0.480) (5.310,0.480) (5.319,0.480) (5.327,0.480) (5.336,0.480) (5.345,0.480) (5.353,0.480) (5.362,0.480) (5.370,0.480) (5.379,0.480) (5.388,0.480) (5.396,0.480) (5.405,0.480) (5.414,0.480) (5.422,0.640) (5.431,0.480) (5.440,0.480) (5.448,0.480) (5.457,0.480) (5.465,0.480) (5.474,0.480) (5.483,0.480) (5.491,0.480) (5.500,0.480) (5.509,0.480) (5.517,0.480) (5.526,0.480) (5.535,0.480) (5.543,0.480) (5.552,0.480) (5.560,0.480) (5.569,0.480) (5.578,0.480) (5.586,0.480) (5.595,0.480) (5.604,0.480) (5.612,0.480) (5.621,0.480) (5.630,0.480) (5.638,0.480) (5.647,0.480) (5.655,0.480) (5.664,0.480) (5.673,0.480) (5.681,0.480) (5.690,0.480) (5.699,0.480) (5.707,0.480) (5.724,0.480) (5.733,0.480) (5.742,0.480) (5.750,0.480) (5.759,0.480) (5.768,0.480) (5.776,0.480) (5.785,0.480) (5.794,0.480) (5.802,0.480) (5.811,0.480) (5.819,0.480) (5.828,0.480) (5.837,0.480) (5.845,0.480) (5.854,0.480) (5.863,0.480) (5.871,0.480) (5.880,0.480) (5.889,0.480) (5.897,0.480) (5.906,0.480) (5.914,0.480) (5.923,0.480) (5.932,0.480) (5.940,0.480) (5.949,0.480) (5.958,0.480) (5.966,0.480) (5.975,0.480) (5.984,0.480) (5.992,0.480) (6.001,0.480) (6.009,0.480) (6.018,0.480) (6.027,0.480) (6.035,0.480) (6.044,0.480) (6.053,0.480) (6.061,0.480) (6.070,0.480) (6.078,0.480) (6.087,0.480) (6.096,0.480) (6.104,0.480) (6.113,0.480) (6.122,0.480) (6.130,0.480) (6.139,0.480) (6.148,0.480) (6.156,0.480) (6.165,0.480) (6.173,0.480) (6.182,0.480) (6.191,0.480) (6.199,0.480) (6.208,0.480) (6.217,0.480) (6.225,0.480) (6.234,0.480) (6.243,0.480) (6.251,0.480) (6.260,0.480) (6.268,0.480) (6.277,0.480) (6.286,0.480) (6.294,0.480) (6.303,0.480) (6.312,0.480) (6.320,0.480) (6.329,0.480) (6.338,0.480) (6.346,0.480) (6.355,0.480) (6.363,0.480) (6.372,0.480) (6.381,0.480) (6.389,0.480) (6.398,0.480) (6.407,0.480) (6.415,0.480) (6.424,0.480) (6.432,0.480) (6.441,0.480) (6.450,0.480) (6.458,0.480) (6.467,0.480) (6.476,0.480) (6.484,0.480) (6.493,0.480) (6.502,0.480) (6.510,0.480) (6.519,0.480) (6.527,0.480) (6.536,0.480) (6.545,0.480) (6.553,0.480) (6.562,0.480) (6.571,0.480) (6.579,0.480) (6.588,0.480) (6.597,0.480) (6.605,0.480) (6.614,0.480) (6.622,0.480) (6.631,0.480) (6.640,0.480) (6.648,0.480) (6.657,0.480) (6.666,0.480) (6.674,0.480) (6.683,0.480) (6.692,0.480) (6.700,0.480) (6.709,0.480) (6.717,0.480) (6.726,0.480) (6.735,0.480) (6.743,0.480) (6.752,0.480) (6.761,0.480) (6.769,0.480) (6.778,0.480) (6.786,0.480) (6.795,0.480) (6.804,0.480) (6.812,0.480) (6.821,0.480) (6.830,0.480) (6.838,0.480) (6.847,0.480) (6.856,0.480) (6.864,0.480) (6.873,0.480) (6.881,0.480) (6.890,0.480) (6.899,0.480) (6.907,0.480)};
\node[torange,above] at (5.18,4.05) {attempt 2};
\draw[dashed,gray] (6.907,0) -- (6.907,4.00);
\draw[tgreen,line width=0.55pt] plot coordinates {(6.907,4.000) (6.916,0.480) (6.925,0.960) (6.933,1.120) (6.942,0.800) (6.951,1.120) (6.959,0.960) (6.968,1.280) (6.976,1.120) (6.985,0.960) (6.994,1.120) (7.002,0.800) (7.011,0.480) (7.020,0.640) (7.028,0.160) (7.037,0.160) (7.046,0.160) (7.054,0.160) (7.063,0.160) (7.071,0.640) (7.080,1.280) (7.089,0.960) (7.097,0.960) (7.106,0.480) (7.115,0.640) (7.123,0.640) (7.132,0.640) (7.141,0.960) (7.149,0.640) (7.158,0.480) (7.166,0.480) (7.175,0.800) (7.184,0.960) (7.192,0.800) (7.201,0.800) (7.210,0.800) (7.218,0.480) (7.227,0.480) (7.235,0.480) (7.244,0.480) (7.253,0.800) (7.261,0.800) (7.270,0.480) (7.279,0.480) (7.287,0.480) (7.296,0.800) (7.305,0.640) (7.313,0.480) (7.322,0.480) (7.330,0.480) (7.339,0.480) (7.348,0.320) (7.356,0.320) (7.365,0.320) (7.374,0.320) (7.382,0.320) (7.391,0.320) (7.400,0.320) (7.408,0.320) (7.417,0.480) (7.425,0.480) (7.434,0.640) (7.443,0.640) (7.451,0.480) (7.460,0.480) (7.469,0.480) (7.477,0.480) (7.486,0.480) (7.495,0.480) (7.503,0.640) (7.512,0.480) (7.520,0.480) (7.529,0.480) (7.538,0.480) (7.546,0.480) (7.555,0.160) (7.564,0.160) (7.572,0.160) (7.581,0.160) (7.589,0.160) (7.598,0.160) (7.607,0.320) (7.615,0.480) (7.624,0.160) (7.633,0.160) (7.641,0.160) (7.650,0.160) (7.659,0.160) (7.667,0.160) (7.676,0.160) (7.684,0.160) (7.693,0.160) (7.702,0.160) (7.710,0.160) (7.719,0.160) (7.728,0.160) (7.736,0.160) (7.745,0.160) (7.754,0.160) (7.762,0.160) (7.771,0.160) (7.779,0.160) (7.788,0.160) (7.797,0.160) (7.805,0.160) (7.814,0.160) (7.823,0.160) (7.831,0.160) (7.840,0.160) (7.849,0.160) (7.857,0.160) (7.866,0.160) (7.874,0.160) (7.883,0.160) (7.892,0.160) (7.900,0.160) (7.909,0.160) (7.918,0.160) (7.926,0.160) (7.935,0.160) (7.943,0.160) (7.952,0.160) (7.961,0.160) (7.969,0.160) (7.978,0.160) (7.987,0.160) (7.995,0.160) (8.004,0.160) (8.013,0.160) (8.021,0.160) (8.030,0.160) (8.038,0.320) (8.047,0.160) (8.056,0.160) (8.064,0.160) (8.073,0.160) (8.082,0.160) (8.090,0.160) (8.099,0.160) (8.108,0.160) (8.116,0.160) (8.125,0.160) (8.133,0.160) (8.142,0.160) (8.151,0.160) (8.159,0.160) (8.168,0.160) (8.177,0.160) (8.185,0.160) (8.194,0.160) (8.203,0.160) (8.211,0.160) (8.220,0.160) (8.228,0.160) (8.237,0.160) (8.246,0.160) (8.254,0.160) (8.263,0.160) (8.272,0.160) (8.280,0.160) (8.289,0.160) (8.297,0.160) (8.306,0.160) (8.315,0.160) (8.323,0.160) (8.332,0.160) (8.341,0.160) (8.349,0.160) (8.358,0.160) (8.367,0.160) (8.375,0.160) (8.384,0.160) (8.392,0.160) (8.401,0.160) (8.410,0.160) (8.418,0.160) (8.427,0.160) (8.436,0.160) (8.444,0.160) (8.453,0.160) (8.462,0.160) (8.470,0.160) (8.479,0.160) (8.487,0.160) (8.496,0.160) (8.505,0.160) (8.513,0.160) (8.522,0.160) (8.531,0.160) (8.539,0.160) (8.548,0.160) (8.557,0.160) (8.565,0.160) (8.574,0.160) (8.582,0.160) (8.591,0.160) (8.600,0.160) (8.608,0.160) (8.617,0.160) (8.626,0.160) (8.634,0.160) (8.643,0.160) (8.651,0.160) (8.660,0.160) (8.669,0.160) (8.677,0.160) (8.686,0.160) (8.695,0.160) (8.703,0.160) (8.712,0.160) (8.721,0.160) (8.729,0.160) (8.738,0.160) (8.746,0.160) (8.755,0.160) (8.764,0.160) (8.772,0.160) (8.781,0.160) (8.790,0.160) (8.798,0.160) (8.807,0.160) (8.816,0.160) (8.824,0.160) (8.833,0.160) (8.841,0.160) (8.850,0.160) (8.859,0.160) (8.867,0.160) (8.876,0.160) (8.885,0.160) (8.893,0.160) (8.902,0.160) (8.911,0.160) (8.919,0.160) (8.928,0.160) (8.936,0.160) (8.945,0.160) (8.954,0.160) (8.962,0.160) (8.971,0.160) (8.980,0.160) (8.988,0.160) (8.997,0.160) (9.005,0.160) (9.014,0.160) (9.023,0.160) (9.031,0.160) (9.040,0.160) (9.049,0.160) (9.057,0.160) (9.066,0.160) (9.075,0.160) (9.083,0.160) (9.092,0.160) (9.100,0.160) (9.109,0.160) (9.118,0.160) (9.126,0.160) (9.135,0.160) (9.144,0.160) (9.152,0.160) (9.161,0.160) (9.170,0.160) (9.178,0.160) (9.187,0.160) (9.195,0.160) (9.204,0.160) (9.213,0.160) (9.221,0.160) (9.230,0.160) (9.239,0.160) (9.247,0.160) (9.256,0.160) (9.265,0.160) (9.273,0.160) (9.282,0.160) (9.290,0.160) (9.299,0.160) (9.308,0.160) (9.316,0.160) (9.325,0.160) (9.334,0.160) (9.342,0.160) (9.351,0.160) (9.359,0.160) (9.368,0.160) (9.377,0.160) (9.385,0.160) (9.394,0.160) (9.403,0.160) (9.411,0.160) (9.420,0.160) (9.429,0.160) (9.437,0.160) (9.446,0.160) (9.454,0.160) (9.463,0.160) (9.472,0.160) (9.480,0.160) (9.489,0.160) (9.498,0.160) (9.506,0.160) (9.515,0.160) (9.524,0.160) (9.532,0.160) (9.541,0.160) (9.549,0.160) (9.558,0.160) (9.567,0.160) (9.575,0.160) (9.584,0.160) (9.593,0.160) (9.601,0.160) (9.610,0.160) (9.619,0.160) (9.627,0.160) (9.636,0.160) (9.644,0.160) (9.653,0.160) (9.662,0.160) (9.670,0.160) (9.679,0.160) (9.688,0.160) (9.696,0.160) (9.705,0.160) (9.714,0.160) (9.722,0.160) (9.731,0.160) (9.739,0.160) (9.748,0.160) (9.757,0.160) (9.765,0.160) (9.774,0.160) (9.783,0.160) (9.791,0.160) (9.800,0.160) (9.808,0.160) (9.817,0.160) (9.826,0.160) (9.834,0.160) (9.843,0.160) (9.852,0.160) (9.860,0.160) (9.869,0.160) (9.878,0.160) (9.886,0.160) (9.895,0.160) (9.903,0.160) (9.912,0.160) (9.921,0.160) (9.929,0.160) (9.938,0.160) (9.947,0.160) (9.955,0.160) (9.964,0.160) (9.973,0.160) (9.981,0.160) (9.990,0.160) (9.998,0.160) (10.007,0.160) (10.016,0.160) (10.024,0.160) (10.033,0.160) (10.042,0.160) (10.050,0.160) (10.059,0.160) (10.068,0.160) (10.076,0.160) (10.085,0.160) (10.093,0.160) (10.102,0.160) (10.111,0.160) (10.119,0.160) (10.128,0.160) (10.137,0.160) (10.145,0.160) (10.154,0.160) (10.162,0.160) (10.171,0.160) (10.180,0.160) (10.188,0.160) (10.197,0.160) (10.206,0.160) (10.214,0.160) (10.223,0.160) (10.232,0.160) (10.240,0.160) (10.249,0.160) (10.257,0.160) (10.266,0.160) (10.275,0.160) (10.283,0.160) (10.292,0.160) (10.301,0.160) (10.309,0.160) (10.318,0.160) (10.327,0.160) (10.335,0.160) (10.344,0.160) (10.352,0.160) (10.361,0.160)};
\node[tgreen,above] at (8.63,4.05) {attempt 3};
\draw[dashed,gray] (10.361,0) -- (10.361,4.00);
\draw[tred,line width=0.55pt] plot coordinates {(10.361,3.200) (10.370,1.600) (10.378,0.960) (10.387,0.480) (10.396,0.320) (10.404,0.320) (10.413,0.320) (10.422,0.320) (10.430,0.320) (10.439,0.320) (10.447,0.960) (10.456,0.640) (10.465,0.800) (10.473,0.960) (10.482,0.480) (10.491,0.960) (10.499,1.440) (10.508,0.960) (10.516,0.960) (10.525,0.960) (10.534,1.120) (10.542,0.960) (10.551,0.800) (10.560,0.800) (10.568,0.640) (10.577,0.160) (10.586,0.160) (10.594,0.160) (10.603,0.160) (10.611,0.160) (10.620,0.160) (10.629,0.160) (10.637,0.160) (10.646,0.160) (10.655,0.160) (10.663,0.640) (10.672,0.800) (10.681,0.640) (10.689,0.640) (10.698,0.480) (10.706,0.480) (10.715,0.480) (10.724,0.480) (10.732,0.320) (10.741,0.320) (10.750,0.800) (10.758,0.480) (10.767,0.480) (10.776,0.640) (10.784,0.320) (10.793,0.320) (10.801,0.320) (10.810,0.320) (10.819,0.320) (10.827,0.320) (10.836,0.320) (10.845,0.640) (10.853,0.480) (10.862,0.480) (10.870,0.480) (10.879,0.800) (10.888,0.640) (10.896,0.480) (10.905,0.640) (10.914,0.640) (10.922,0.960) (10.931,0.640) (10.940,0.800) (10.948,0.480) (10.957,0.480) (10.965,0.480) (10.974,0.640) (10.983,0.320) (10.991,0.320) (11.000,0.000)};
\node[tred,above] at (10.68,4.05) {attempt 4};
\draw (0.000,0) -- (0.000,-0.08) node[below] {0};
\draw (3.454,0) -- (3.454,-0.08) node[below] {20\,000};
\draw (6.907,0) -- (6.907,-0.08) node[below] {40\,000};
\draw (10.361,0) -- (10.361,-0.08) node[below] {60\,000};
\end{tikzpicture}
\caption{Cost $E$ against the number of moves in an annealing search for $\W(68,55)$ (illustrative re-run, see text): three attempts end with $E>0$ and the fourth reaches $E=0$.}\label{fig:trace}
\end{figure}
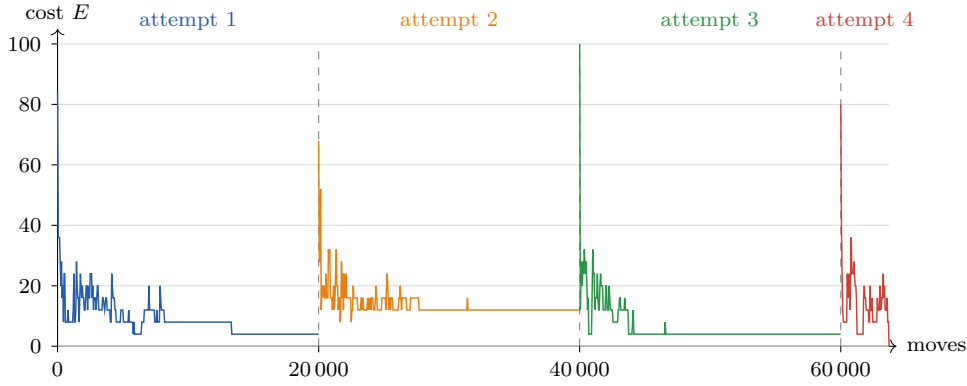

\subsection{Stage 3: skew-Hadamard doubling}\label{sec:skew}

The six pairs $(16,15)$, $(40,39)$, $(56,55)$, $(64,63)$, $(88,87)$, $(96,95)$ get the status \texttt{none\_gs} in the first program, since no Goethals--Seidel matrix exists for them (Proposition~\ref{prop:parity}), and a second program, \texttt{remaining\_cases.py}, constructs them. A \emph{skew-Hadamard matrix} of order $n_0$ is a Hadamard matrix $H$ with $H+H\T=2I$ (equivalently: diagonal entries $1$ and $H_{ij}=-H_{ji}$ for $i\ne j$). By Proposition~\ref{prop:paley1}, for $n_0=q+1$ with $q\equiv3\pmod4$ a prime power, $H=I+S$ is skew-Hadamard: $H+H\T=2I+S+S\T=2I$.

\begin{lemma}\label{lem:double}
If $H$ is a skew-Hadamard matrix of order $n_0$ then
\[
H'=\begin{pmatrix}H&H\\-H\T&H\T\end{pmatrix}
\]
is a skew-Hadamard matrix of order $2n_0$, and $H'-I$ is a $\W(2n_0,2n_0-1)$.
\end{lemma}
\begin{proof}
$H'$ has entries $\pm1$. Its transpose is ${H'}\T=\begin{pmatrix}H\T&-H\\H\T&H\end{pmatrix}$, so
\[
H'{H'}\T=\begin{pmatrix}HH\T+HH\T&-HH+HH\\-H\T H\T+H\T H\T&H\T H+H\T H\end{pmatrix}=\begin{pmatrix}2n_0I&0\\0&2n_0I\end{pmatrix},
\]
using $HH\T=H\T H=n_0I$ (for a square matrix, $HH\T=n_0I$ implies $H\T H=n_0I$). Also $H'+{H'}\T=\begin{pmatrix}H+H\T&H-H\\-H\T+H\T&H\T+H\end{pmatrix}=2I$. In particular the diagonal of $H'$ is $1$, so $H'-I$ has zero diagonal, entries in $\{-1,0,1\}$ and exactly $2n_0-1$ nonzero entries (the $\pm1$ entries off the diagonal) in each row; and
\[
(H'-I)(H'-I)\T=H'{H'}\T-H'-{H'}\T+I=2n_0I-2I+I=(2n_0-1)I.\qedhere
\]
\end{proof}
For the pair $(n,n-1)$ the program takes the first element $n_0$ of the chain $n,n/2,n/4,\dots$ for which $n_0-1$ is a prime power $\equiv3\pmod4$, builds $I+S$ of order $n_0$ and doubles until the order is $n$. In the six cases of the catalogue one doubling suffices: $n_0=n/2$ with $q=n/2-1\in\{7,19,27,31,43,47\}$ for $n=16,40,56,64,88,96$ (method \texttt{skew\_hadamard}; the recipe stores $n$, the base $n_0$ and $q$).

\emph{Example.} $n=16$: $q=7$, $H=I+S$ has order $8$; one doubling gives the $16\times16$ matrix $H'$, and $\W(16,15)=H'-I$ (Figure~\ref{fig:skew}). A remark on conventions: the script \texttt{remaining\_cases.py} indexes the character matrix as $\chi(i-j)$ instead of $\chi(j-i)$ used in \texttt{all\_chains.py} (and in Section~\ref{sec:paley}); for $q\equiv3\pmod4$ this replaces $Q$ by $-Q$, so the Paley matrix used in this stage is a skew conference matrix of the same kind (the proof of Proposition~\ref{prop:paley1} applies verbatim), and it is the one drawn in Figure~\ref{fig:skew}.

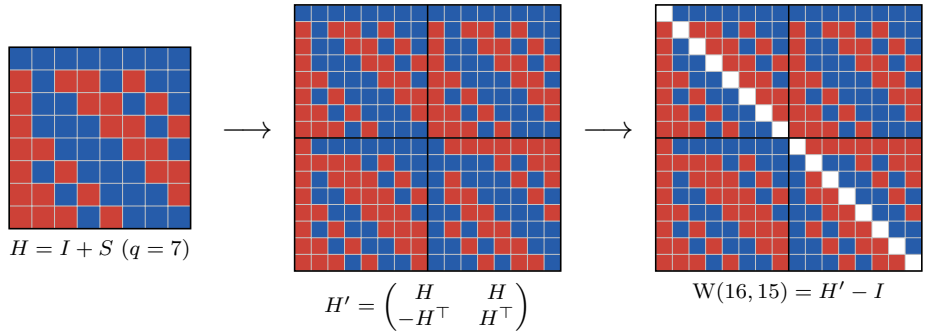
\begin{figure}[ht]
\centering
\begin{tikzpicture}[x=0.300cm,y=0.300cm,baseline=(mid)]
\coordinate (mid) at (4.0,-4.0);
\foreach \c/\r in {0/0,1/0,2/0,3/0,4/0,5/0,6/0,7/0,1/1,4/1,6/1,7/1,1/2,2/2,5/2,7/2,1/3,2/3,3/3,6/3,2/4,3/4,4/4,7/4,1/5,3/5,4/5,5/5,2/6,4/6,5/6,6/6,3/7,5/7,6/7,7/7}{\fill[wplus] (\c,-\r) rectangle ++(1,-1);}
\foreach \c/\r in {0/1,2/1,3/1,5/1,0/2,3/2,4/2,6/2,0/3,4/3,5/3,7/3,0/4,1/4,5/4,6/4,0/5,2/5,6/5,7/5,0/6,1/6,3/6,7/6,0/7,1/7,2/7,4/7}{\fill[wminus] (\c,-\r) rectangle ++(1,-1);}
\draw[xstep=1,ystep=1,gray!45,line width=0.12pt] (0,-8) grid (8,0);
\draw[line width=0.6pt] (0,0) rectangle (8,-8);
\node[font=\scriptsize,below,align=center] at (4.0,-8) {$H=I+S$ ($q=7$)};
\end{tikzpicture}
\hspace{0.15cm}{\large $\longrightarrow$}\hspace{0.15cm}
\begin{tikzpicture}[x=0.220cm,y=0.220cm,baseline=(mid)]
\coordinate (mid) at (8.0,-8.0);
\foreach \c/\r in {0/0,1/0,2/0,3/0,4/0,5/0,6/0,7/0,8/0,9/0,10/0,11/0,12/0,13/0,14/0,15/0,1/1,4/1,6/1,7/1,9/1,12/1,14/1,15/1,1/2,2/2,5/2,7/2,9/2,10/2,13/2,15/2,1/3,2/3,3/3,6/3,9/3,10/3,11/3,14/3,2/4,3/4,4/4,7/4,10/4,11/4,12/4,15/4,1/5,3/5,4/5,5/5,9/5,11/5,12/5,13/5,2/6,4/6,5/6,6/6,10/6,12/6,13/6,14/6,3/7,5/7,6/7,7/7,11/7,13/7,14/7,15/7,1/8,2/8,3/8,4/8,5/8,6/8,7/8,8/8,4/9,6/9,7/9,8/9,9/9,10/9,11/9,13/9,1/10,5/10,7/10,8/10,10/10,11/10,12/10,14/10,1/11,2/11,6/11,8/11,11/11,12/11,13/11,15/11,2/12,3/12,7/12,8/12,9/12,12/12,13/12,14/12,1/13,3/13,4/13,8/13,10/13,13/13,14/13,15/13,2/14,4/14,5/14,8/14,9/14,11/14,14/14,15/14,3/15,5/15,6/15,8/15,9/15,10/15,12/15,15/15}{\fill[wplus] (\c,-\r) rectangle ++(1,-1);}
\foreach \c/\r in {0/1,2/1,3/1,5/1,8/1,10/1,11/1,13/1,0/2,3/2,4/2,6/2,8/2,11/2,12/2,14/2,0/3,4/3,5/3,7/3,8/3,12/3,13/3,15/3,0/4,1/4,5/4,6/4,8/4,9/4,13/4,14/4,0/5,2/5,6/5,7/5,8/5,10/5,14/5,15/5,0/6,1/6,3/6,7/6,8/6,9/6,11/6,15/6,0/7,1/7,2/7,4/7,8/7,9/7,10/7,12/7,0/8,9/8,10/8,11/8,12/8,13/8,14/8,15/8,0/9,1/9,2/9,3/9,5/9,12/9,14/9,15/9,0/10,2/10,3/10,4/10,6/10,9/10,13/10,15/10,0/11,3/11,4/11,5/11,7/11,9/11,10/11,14/11,0/12,1/12,4/12,5/12,6/12,10/12,11/12,15/12,0/13,2/13,5/13,6/13,7/13,9/13,11/13,12/13,0/14,1/14,3/14,6/14,7/14,10/14,12/14,13/14,0/15,1/15,2/15,4/15,7/15,11/15,13/15,14/15}{\fill[wminus] (\c,-\r) rectangle ++(1,-1);}
\draw[xstep=1,ystep=1,gray!45,line width=0.12pt] (0,-16) grid (16,0);
\draw[line width=0.6pt] (8,0) -- (8,-16);
\draw[line width=0.6pt] (0,-8) -- (16,-8);
\draw[line width=0.6pt] (0,0) rectangle (16,-16);
\node[font=\scriptsize,below,align=center] at (8.0,-16) {$H'=\begin{pmatrix}H&H\\-H^{\top}&H^{\top}\end{pmatrix}$};
\end{tikzpicture}
\hspace{0.15cm}{\large $\longrightarrow$}\hspace{0.15cm}
\begin{tikzpicture}[x=0.220cm,y=0.220cm,baseline=(mid)]
\coordinate (mid) at (8.0,-8.0);
\foreach \c/\r in {1/0,2/0,3/0,4/0,5/0,6/0,7/0,8/0,9/0,10/0,11/0,12/0,13/0,14/0,15/0,4/1,6/1,7/1,9/1,12/1,14/1,15/1,1/2,5/2,7/2,9/2,10/2,13/2,15/2,1/3,2/3,6/3,9/3,10/3,11/3,14/3,2/4,3/4,7/4,10/4,11/4,12/4,15/4,1/5,3/5,4/5,9/5,11/5,12/5,13/5,2/6,4/6,5/6,10/6,12/6,13/6,14/6,3/7,5/7,6/7,11/7,13/7,14/7,15/7,1/8,2/8,3/8,4/8,5/8,6/8,7/8,4/9,6/9,7/9,8/9,10/9,11/9,13/9,1/10,5/10,7/10,8/10,11/10,12/10,14/10,1/11,2/11,6/11,8/11,12/11,13/11,15/11,2/12,3/12,7/12,8/12,9/12,13/12,14/12,1/13,3/13,4/13,8/13,10/13,14/13,15/13,2/14,4/14,5/14,8/14,9/14,11/14,15/14,3/15,5/15,6/15,8/15,9/15,10/15,12/15}{\fill[wplus] (\c,-\r) rectangle ++(1,-1);}
\foreach \c/\r in {0/1,2/1,3/1,5/1,8/1,10/1,11/1,13/1,0/2,3/2,4/2,6/2,8/2,11/2,12/2,14/2,0/3,4/3,5/3,7/3,8/3,12/3,13/3,15/3,0/4,1/4,5/4,6/4,8/4,9/4,13/4,14/4,0/5,2/5,6/5,7/5,8/5,10/5,14/5,15/5,0/6,1/6,3/6,7/6,8/6,9/6,11/6,15/6,0/7,1/7,2/7,4/7,8/7,9/7,10/7,12/7,0/8,9/8,10/8,11/8,12/8,13/8,14/8,15/8,0/9,1/9,2/9,3/9,5/9,12/9,14/9,15/9,0/10,2/10,3/10,4/10,6/10,9/10,13/10,15/10,0/11,3/11,4/11,5/11,7/11,9/11,10/11,14/11,0/12,1/12,4/12,5/12,6/12,10/12,11/12,15/12,0/13,2/13,5/13,6/13,7/13,9/13,11/13,12/13,0/14,1/14,3/14,6/14,7/14,10/14,12/14,13/14,0/15,1/15,2/15,4/15,7/15,11/15,13/15,14/15}{\fill[wminus] (\c,-\r) rectangle ++(1,-1);}
\draw[xstep=1,ystep=1,gray!45,line width=0.12pt] (0,-16) grid (16,0);
\draw[line width=0.6pt] (8,0) -- (8,-16);
\draw[line width=0.6pt] (0,-8) -- (16,-8);
\draw[line width=0.6pt] (0,0) rectangle (16,-16);
\node[font=\scriptsize,below,align=center] at (8.0,-16) {$\W(16,15)=H'-I$};
\end{tikzpicture}
\caption{Skew-Hadamard doubling for $n=16$: the Hadamard matrix $H=I+S$ of order $8$ ($q=7$), its double $H'$ of order $16$, and $\W(16,15)=H'-I$ (stored matrix).}\label{fig:skew}
\end{figure}
\section{Exact verification}\label{sec:verify}

The validity of the catalogue rests on checks that are exact and independent of the search. Two checks are made.

\paragraph{At storing time.} Both construction programs (\texttt{all\_chains.py} and \texttt{remaining\_cases.py}) rebuild every candidate matrix $W$ and apply an exact check before accepting it: $W$ is $n\times n$, all entries lie in $\{-1,0,1\}$, every row and every column has exactly $k$ nonzero entries, and $WW\T=kI_n$ and $W\T W=kI_n$ hold with exact equality (no tolerance). The products are computed with floating-point (\texttt{float64}) matrix multiplication; since all entries are integers with absolute value at most $1$ and $n\le108$, every intermediate and final value is an integer of absolute value at most $108$, far below $2^{53}$, so these computations are exact. A candidate that fails the check is discarded (and the next method is tried).

\paragraph{Separate verification.} The script \texttt{verify\_weighing.py} then re-reads every stored file and repeats the checks using integer (\texttt{int64}) arithmetic only: for entries in $\{-1,0,1\}$ and $n\le108$ no overflow is possible. For Goethals--Seidel matrices it also checks that the stored matrix equals the array built from the stored sequences, and that $\sum_iw_i=k=\sum_is_i^2$. All $1512$ files pass, with no warnings (file \texttt{report.csv}).

\begin{proof}[Proof of Theorem~\ref{thm:main}]
For each $(n,k)\in\mathcal S$ the repository contains an explicit matrix $W$ with entries in $\{-1,0,1\}$, and the verification above shows $WW\T=kI_n$. This is the definition of a $\W(n,k)$. The $1512$ matrices together with the checks constitute the proof; no assumption about the success probability of the annealing is needed.
\end{proof}

\section{Results}\label{sec:results}

\paragraph{Methods.} Table~\ref{tab:methods} shows which method produced the stored matrix of each of the $1512$ pairs. In particular $961$ matrices come from classical constructions and $551$ from Goethals--Seidel arrays.

\paragraph{Computational effort.}
The searches were run on a laptop (Intel Core i7-9750H, 16~GB RAM, Python~3.11.4, NumPy~1.26.4). The times recorded in the log \texttt{all\_chains\_status.csv} (the elapsed time of the processing of each pair) add up to about $26.6$~h for $n\le100$ (of which $19.3$~h for $n=100$), $86.2$~h for $n=104$ and $99.7$~h for $n=108$; practically all of this is annealing. The hardest pair, $(104,93)$, needed $284{,}396$ annealing attempts in the logged run ($27.3$~h). The number of attempts includes the final, successful one. For comparison, the pair $(68,55)$ of Section~\ref{sec:anneal} needed four attempts and $1.65$~s.

\paragraph{Thirteen pairs that appear as unresolved in historical tables.}
Table~\ref{tab:cases} lists the stored Goethals--Seidel matrices for thirteen pairs with $n\le108$, which appear as unresolved in the historical tables consulted by the author, namely those of~\cite{handbook} and tables derived from~\cite{ks97}; the tables do not agree on every pair. We have not carried out an exhaustive literature search, and we do not claim that all of these cases were open: an unresolved entry in an older table does not establish present-day novelty, and later computational work such as~\cite{bdgks24,kkps12,kks12} reports resolving open cases of the Handbook~\cite{handbook}, without our having been able to check whether it covers any of these pairs. We would be grateful for pointers to earlier explicit constructions. The stored profile $(w_i,s_i)$ of each matrix is given in the table (in the order of the stored sequences, which need not be sorted), and the sequences, together with a picture of each matrix, are listed in Appendix~\ref{app:seq}.

\begin{table}[ht]
\centering
\small
\begin{tabular}{rrrlll}
\toprule
$n$ & $k$ & $m$ & $(w_0,\dots,w_3)$ & $(s_0,\dots,s_3)$ & attempts \\
\midrule
92 & 79 & 23 & (23,12,23,21) & (7,2,5,1) & 181 \\
100 & 87 & 25 & (24,25,21,17) & (6,5,5,1) & 4\,740 \\
100 & 91 & 25 & (24,21,21,25) & (0,9,1,3) & 23\,409 \\
100 & 93 & 25 & (25,24,22,22) & (7,6,2,2) & 23\,312 \\
100 & 95 & 25 & (23,23,24,25) & (5,5,6,3) & 8\,404 \\
104 & 95 & 26 & (23,25,22,25) & (9,3,2,1) & 184\,326 \\
108 & 79 & 27 & (13,17,26,23) & (3,3,6,5) & 6\,200 \\
108 & 87 & 27 & (13,21,26,27) & (1,7,6,1) & 72\,922 \\
108 & 93 & 27 & (18,26,22,27) & (2,2,2,9) & 12\,284 \\
108 & 95 & 27 & (17,25,26,27) & (1,3,2,9) & 177\,239 \\
108 & 99 & 27 & (27,27,25,20) & (9,1,1,4) & 51\,121 \\
108 & 101 & 27 & (24,27,24,26) & (2,9,0,4) & 50\,943 \\
108 & 103 & 27 & (27,25,26,25) & (1,1,10,1) & 135\,574 \\
\bottomrule
\end{tabular}
\caption{Stored Goethals--Seidel matrices for the thirteen pairs discussed in Section~\ref{sec:results}. The last column is the number of annealing attempts (including the successful one) logged for the successful run of the original computation.}\label{tab:cases}
\end{table}

\section{Availability and limitations}\label{sec:availability}

The code, the $1512$ matrices (one JSON file per pair, containing the matrix, the method and the recipe), the verification report and the search log are in the repository \url{https://github.com/adrianmaciasquintero/weighing-matrices} (version~1.1.0), archived at Zenodo~\cite{zenodo}, \href{https://doi.org/10.5281/zenodo.23484194}{doi:10.5281/zenodo.23484194}, under the Apache License~2.0.

\paragraph{Verifying the archived catalogue.} This is quick and deterministic and does not repeat the search:
\begin{center}\small\ttfamily
python verify\_weighing.py solutions --first-only
\end{center}

\paragraph{Repeating the discovery.} The construction programs skip every pair that is already stored in the output directory, so they must be run on a \emph{new, empty} output directory, and the search is stochastic and expensive (hundreds of hours on a laptop):
\begin{center}\small\ttfamily
python all\_chains.py --nmin 4 --nmax 108 --out regenerated\_solutions\\
python remaining\_cases.py --out regenerated\_solutions\\
python verify\_weighing.py regenerated\_solutions --first-only
\end{center}
The annealing uses no fixed random seed by default, so a new run will in general produce different (equally valid) matrices and different running times; the archived files, not a particular random trajectory, are the reference certificates. The catalogue gives one matrix per pair and no classification up to equivalence. A failure of the search to find a matrix is not evidence of non-existence, except in the cases excluded by a valid mathematical or exhaustive argument (Proposition~\ref{prop:parity} and Section~\ref{sec:exh}).

\section*{Use of artificial intelligence}

The idea of the project and the algorithm described in this note are the author's own: it was the author who specified the algorithm used (the order of the methods, the Goethals--Seidel search with profile filter and simulated annealing, and the exact verification). An artificial-intelligence model, Claude (Anthropic; model Claude Sonnet 5.5), was used to write in Python the programs (\texttt{.py} files) that implement that algorithm, following the author's instructions. In preparing this manuscript, Claude Sonnet 5.5 (Anthropic) and GPT-5.6 Luna (OpenAI) also helped to suggest revisions to the exposition and assist with LaTeX typesetting and the figures, working from the author's code and data and following the author's instructions and corrections; the author has reviewed the content and takes full responsibility for it. No AI model is involved in the execution of the search programs. The search is performed algorithmically using exhaustive search and simulated annealing, and each resulting matrix is verified using exact integer arithmetic.

\appendix
\section{Sequences for the pairs of Table~\ref{tab:cases}}\label{app:seq}

For each pair we draw the stored matrix (thick lines separate the sixteen blocks of size $m\times m$) and list the four ternary sequences $A,B,C,D$ of length $m=n/4$ of the stored Goethals--Seidel matrix, written with \texttt{+}, \texttt{-}, \texttt{0} for $1,-1,0$ and starting at position $t=0$ (the entry $x_0$ is the diagonal entry of the circulant $\circm(x)$, whose $(i,j)$ entry is $x_{(j-i)\bmod m}$). The matrix is the array $G$ of~\eqref{eq:G}. For each pair, $\sum w_i=k=\sum s_i^2$ was checked, and the matrix rebuilt from the sequences was checked to satisfy $WW\T=W\T W=kI$ in exact integer arithmetic.

\Needspace{14\baselineskip}
\paragraph{$\W(92,79)$, $m=23$.} Weights $(23,12,23,21)$, sums $(7,2,5,1)$, $7^2+2^2+5^2+1^2=79$.\par\smallskip
\noindent\begin{minipage}[c]{0.60\textwidth}\centering

\end{minipage}\hfill
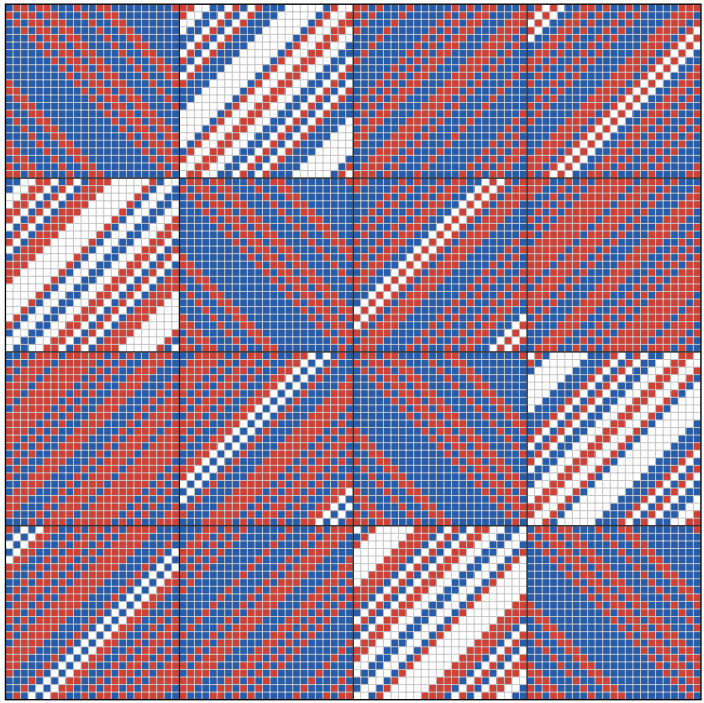
\begin{minipage}[c]{0.38\textwidth}{\ttfamily\footnotesize\noindent A = +--+--+++-++--++++++++-\\B = 00-+00000+++-0+-0++00--\\C = -+++--+-+-+++++-+++-+--\\D = ---++++-+-++-+--++0-0-+\par}\end{minipage}\par\medskip
\Needspace{14\baselineskip}

\newpage

\paragraph{$\W(100,87)$, $m=25$.} Weights $(24,25,21,17)$, sums $(6,5,5,1)$, $6^2+5^2+5^2+1^2=87$.\par\smallskip
\noindent\begin{minipage}[c]{0.60\textwidth}\centering
%
\end{minipage}\hfill
\begin{minipage}[c]{0.38\textwidth}{\ttfamily\footnotesize\noindent A = ---+--0++-+++-+-++++++++-\\B = -+-+--+++--+-++++--++++-+\\C = --0+-++-+0-0-++-++0-+++++\\D = 00+-000+--++-++++0--00-+-\par}\end{minipage}\par\medskip
\Needspace{14\baselineskip}
\paragraph{$\W(100,91)$, $m=25$.} Weights $(24,21,21,25)$, sums $(0,9,1,3)$, $0^2+9^2+1^2+3^2=91$.\par\smallskip
\noindent\begin{minipage}[c]{0.60\textwidth}\centering
%
\end{minipage}\hfill
\begin{minipage}[c]{0.38\textwidth}{\ttfamily\footnotesize\noindent A = ---+---++-++-0-+-+++-++-+\\B = -+0+++--00-++++-++++-++0+\\C = -0---++00----+++++++-+0-+\\D = ++--+++---++-+--++-+-+++-\par}\end{minipage}\par\medskip
\Needspace{14\baselineskip}
\paragraph{$\W(100,93)$, $m=25$.} Weights $(25,24,22,22)$, sums $(7,6,2,2)$, $7^2+6^2+2^2+2^2=93$.\par\smallskip
\noindent\begin{minipage}[c]{0.60\textwidth}\centering
%
\end{minipage}\hfill
\begin{minipage}[c]{0.38\textwidth}{\ttfamily\footnotesize\noindent A = ---++--+-+-++-+-+++++++++\\B = --0+-++--++-+-+++++-+++-+\\C = -+---+0-+++-+++---0+++0+-\\D = ++0-+-++-++0+--0+++-+----\par}\end{minipage}\par\medskip
\Needspace{14\baselineskip}
\paragraph{$\W(100,95)$, $m=25$.} Weights $(23,23,24,25)$, sums $(5,5,6,3)$, $5^2+5^2+6^2+3^2=95$.\par\smallskip
\noindent\begin{minipage}[c]{0.60\textwidth}\centering
%
\end{minipage}\hfill
\begin{minipage}[c]{0.38\textwidth}{\ttfamily\footnotesize\noindent A = 0+-0+-+++-++-++++--++---+\\B = ++++0-+0+++-+-++--+--+-+-\\C = -++++++0----+++-+++--+-++\\D = -+-++---+-+++-+++-++++---\par}\end{minipage}\par\medskip
\Needspace{14\baselineskip}
\paragraph{$\W(104,95)$, $m=26$.} Weights $(23,25,22,25)$, sums $(9,3,2,1)$, $9^2+3^2+2^2+1^2=95$.\par\smallskip
\noindent\begin{minipage}[c]{0.60\textwidth}\centering
%
\end{minipage}\hfill
\begin{minipage}[c]{0.38\textwidth}{\ttfamily\footnotesize\noindent A = +-+-+++++++++-+0--+0++-0-+\\B = ++---+++--+----+++-++-0+++\\C = 0+--0-++0+-++-++-0-+++-+--\\D = +-+---++--++-+-+---+-+++0+\par}\end{minipage}\par\medskip
\Needspace{14\baselineskip}
\paragraph{$\W(108,79)$, $m=27$.} Weights $(13,17,26,23)$, sums $(3,3,6,5)$, $3^2+3^2+6^2+5^2=79$.\par\smallskip
\noindent\begin{minipage}[c]{0.60\textwidth}\centering
%
\end{minipage}\hfill
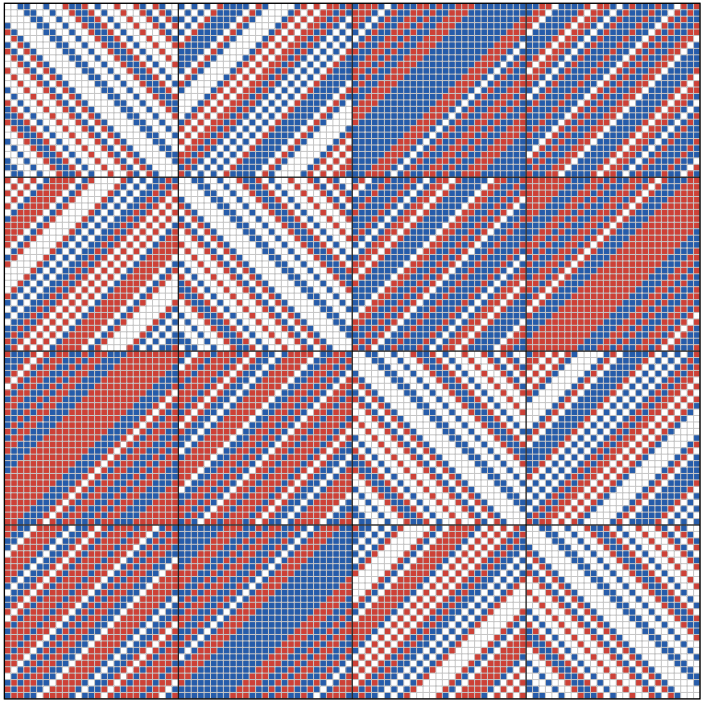
\begin{minipage}[c]{0.38\textwidth}{\ttfamily\footnotesize\noindent A = 00++00+00-++-0+00-00-+0-+00\\B = -+--0-0-+0000+-0++++-0+0+0+\\C = ++++++++---++-+--++-+0+---+\\D = -+-0++-+++-+0+--0-++++0--++\par}\end{minipage}\par\medskip
\Needspace{14\baselineskip}
\paragraph{$\W(108,87)$, $m=27$.} Weights $(13,21,26,27)$, sums $(1,7,6,1)$, $1^2+7^2+6^2+1^2=87$.\par\smallskip
\noindent\begin{minipage}[c]{0.60\textwidth}\centering
%
\end{minipage}\hfill
\begin{minipage}[c]{0.38\textwidth}{\ttfamily\footnotesize\noindent A = 0-00-+0-0+00-++++00+-00000-\\B = +0+++++--0+--0-+++0+-+++-00\\C = +++-+++++--+-+-+-++0+--+--+\\D = +++----+-++---+-++-+++--++-\par}\end{minipage}\par\medskip
\Needspace{14\baselineskip}
\paragraph{$\W(108,93)$, $m=27$.} Weights $(18,26,22,27)$, sums $(2,2,2,9)$, $2^2+2^2+2^2+9^2=93$.\par\smallskip
\noindent\begin{minipage}[c]{0.60\textwidth}\centering
%
\end{minipage}\hfill
\begin{minipage}[c]{0.38\textwidth}{\ttfamily\footnotesize\noindent A = 0++---0+++0+0+000+---0+-+-0\\B = +--++0++++++---+-+--+--+--+\\C = +-+-++-++00-+0-0-++-++-0-+-\\D = -++++++++---+++-++--++++--+\par}\end{minipage}\par\medskip
\Needspace{14\baselineskip}
\paragraph{$\W(108,95)$, $m=27$.} Weights $(17,25,26,27)$, sums $(1,3,2,9)$, $1^2+3^2+2^2+9^2=95$.\par\smallskip
\noindent\begin{minipage}[c]{0.60\textwidth}\centering
%
\end{minipage}\hfill
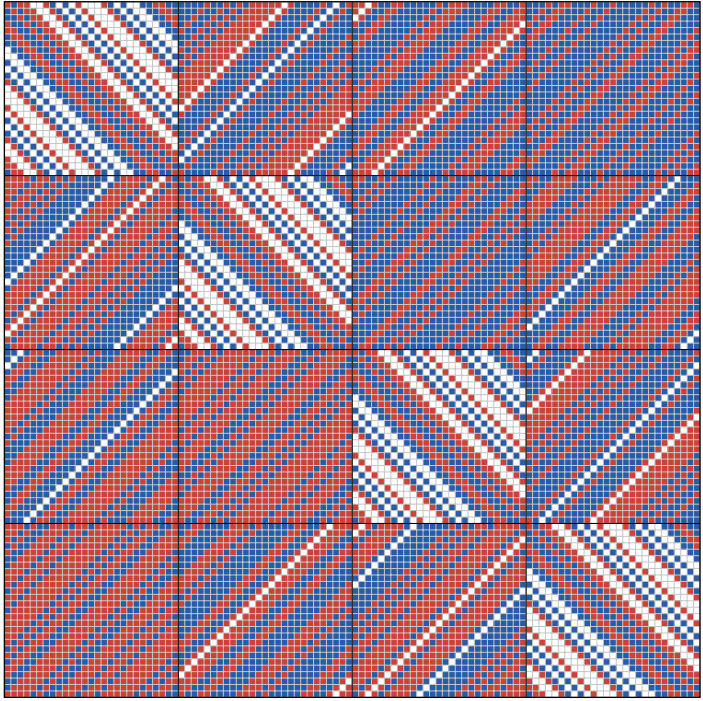
\begin{minipage}[c]{0.38\textwidth}{\ttfamily\footnotesize\noindent A = -+--00-+0+0-000-+0+00++--++\\B = -+0+-++++-0-----+-+-++-++++\\C = -++--++++---+-+++++--++-0--\\D = +--++-+-+++++-++-+++--+-+++\par}\end{minipage}\par\medskip
\Needspace{14\baselineskip}
\paragraph{$\W(108,99)$, $m=27$.} Weights $(27,27,25,20)$, sums $(9,1,1,4)$, $9^2+1^2+1^2+4^2=99$.\par\smallskip
\noindent\begin{minipage}[c]{0.60\textwidth}\centering
%
\end{minipage}\hfill
\begin{minipage}[c]{0.38\textwidth}{\ttfamily\footnotesize\noindent A = +---+--+++-+++++-+-++++++-+\\B = +++--+++++-+--+---+++----+-\\C = ++-0-+-+--++-+--+---+++++-0\\D = -+0-+-++0++0-0+0++-+0+---+0\par}\end{minipage}\par\medskip
\Needspace{14\baselineskip}
\paragraph{$\W(108,101)$, $m=27$.} Weights $(24,27,24,26)$, sums $(2,9,0,4)$, $2^2+9^2+0^2+4^2=101$.\par\smallskip
\noindent\begin{minipage}[c]{0.60\textwidth}\centering
%
\end{minipage}\hfill
\begin{minipage}[c]{0.38\textwidth}{\ttfamily\footnotesize\noindent A = --+---++0+-+-+--+-+++0-++0+\\B = +-++-+---+++-+++++++++-+--+\\C = --++-+++0-++0+---+++-+---0-\\D = +++++--++--+-+-+0--++++--+-\par}\end{minipage}\par\medskip
\Needspace{14\baselineskip}
\paragraph{$\W(108,103)$, $m=27$.} Weights $(27,25,26,25)$, sums $(1,1,10,1)$, $1^2+1^2+10^2+1^2=103$.\par\smallskip
\noindent\begin{minipage}[c]{0.60\textwidth}\centering
%
\end{minipage}\hfill
\begin{minipage}[c]{0.38\textwidth}{\ttfamily\footnotesize\noindent A = +---++--++-++--+-+--++--+++\\B = ++0-0+-+++-----+-+++-+-+--+\\C = -+++++0+++++-++-+----+-++++\\D = -+--0+0+++--+-++-+-+-+++---\par}\end{minipage}\par


\medskip

\newpage 

\begin{thebibliography}{99}
\bibitem{bdgks24} R.~Ben-Av, G.~Dula, A.~Goldberger, I.~Kotsireas and Y.~Strassler, New weighing matrices via partitioned group actions, \emph{Discrete Math.} \textbf{347} (2024), no.~5, article 113908.
\bibitem{handbook} C.~J.~Colbourn and J.~H.~Dinitz (eds.), \emph{Handbook of Combinatorial Designs}, 2nd ed., Chapman \& Hall/CRC, 2007; see in particular R.~Craigen and H.~Kharaghani, Orthogonal designs, Section~V.2, pp.~280--295.
\bibitem{dinitz} J.~H.~Dinitz, New results in Part V (updates to the \emph{Handbook of Combinatorial Designs}), web page, last edited 15 February 2018, \url{https://site.uvm.edu/jdinitz/?page_id=404} (accessed 23 October 2026).
\bibitem{gsa23} S.~D.~Georgiou, S.~Stylianou and H.~Alrweili, On symmetric weighing matrices, \emph{Mathematics} \textbf{11} (2023), no.~9, article 2076.
\bibitem{gpw74} A.~V.~Geramita, N.~J.~Pullman and J.~S.~Wallis, Families of weighing matrices, \emph{Bull. Austral. Math. Soc.} \textbf{10} (1974), 119--122.
\bibitem{gsbook} A.~V.~Geramita and J.~Seberry, \emph{Orthogonal Designs: Quadratic Forms and Hadamard Matrices}, Lecture Notes in Pure and Applied Mathematics~45, Marcel Dekker, New York, 1979.
\bibitem{gw74} A.~V.~Geramita and J.~Seberry Wallis, Orthogonal designs III: Weighing matrices, \emph{Utilitas Math.} \textbf{6} (1974), 209--236.
\bibitem{gs70} J.~M.~Goethals and J.~J.~Seidel, A skew Hadamard matrix of order 36, \emph{J. Austral. Math. Soc.} \textbf{11} (1970), 343--344.
\bibitem{gbds26} A.~Goldberger, R.~Ben-Av, G.~Dula and Y.~Strassler, Constructing, classifying and studying the space of small integer weighing matrices, arXiv:2603.17552 (2026).
\bibitem{hm12} M.~Harada and A.~Munemasa, On the classification of weighing matrices and self-orthogonal codes, \emph{J. Combin. Des.} \textbf{20} (2012), 40--57.
\bibitem{kgv} S.~Kirkpatrick, C.~D.~Gelatt~Jr.\ and M.~P.~Vecchi, Optimization by simulated annealing, \emph{Science} \textbf{220} (1983), 671--680.
\bibitem{kkps12} I.~S.~Kotsireas, C.~Koukouvinos, P.~M.~Pardalos and D.~E.~Simos, Competent genetic algorithms for weighing matrices, \emph{J. Comb. Optim.} \textbf{24} (2012), 508--525.
\bibitem{kks12} I.~S.~Kotsireas, C.~Koukouvinos and J.~Seberry, New weighing matrices constructed from two circulant submatrices, \emph{Optim. Lett.} \textbf{6} (2012), 211--217.
\bibitem{ks93} C.~Koukouvinos and J.~Seberry, Some new weighing matrices using sequences with zero autocorrelation function, \emph{Australas. J. Combin.} \textbf{8} (1993), 143--152.
\bibitem{ks97} C.~Koukouvinos and J.~Seberry, Weighing matrices and their applications, \emph{J. Statist. Plann. Inference} \textbf{62} (1997), 91--101.
\bibitem{zenodo} A.~Mac\'ias-Quintero, \emph{Weighing Matrices $\W(n,k)$ for $n\equiv0\pmod4$, $n\le108$: Constructions and Exact Verification}, software, version~1.1.0, Zenodo, 2026, \href{https://doi.org/10.5281/zenodo.23484194}{doi:10.5281/zenodo.23484194}. Source repository: \url{https://github.com/adrianmaciasquintero/weighing-matrices}.
\bibitem{paley} R.~E.~A.~C.~Paley, On orthogonal matrices, \emph{J. Math. Phys.} \textbf{12} (1933), 311--320.
\bibitem{sy20} J.~Seberry and M.~Yamada, \emph{Hadamard Matrices: Constructions using Number Theory and Algebra}, Wiley, Hoboken, 2020.
\end{thebibliography}
\end{document}